\documentclass[11pt]{article}

\usepackage[margin=1in]{geometry}
\usepackage{amsmath,amsthm,amssymb,amsfonts,float}

\makeatletter
\newcommand{\subjclass}[2][2010]{%
  \let\@oldtitle\@title%
  \gdef\@title{\@oldtitle\footnotetext{#1 \emph{Mathematics subject classification.} #2}}%
}
\makeatother

\usepackage[colorlinks=true, pdfstartview=FitV, linkcolor=blue,citecolor=blue, urlcolor=blue]{hyperref}

\usepackage[abbrev,lite,nobysame]{amsrefs}
\usepackage{times}
\usepackage[usenames,dvipsnames]{color}

\usepackage{esint}

\usepackage{mathtools,enumitem,mathrsfs}

\usepackage[compact]{titlesec}

\usepackage[title]{appendix}

\usepackage{comment}

\mathtoolsset{showonlyrefs=true}
    
\newcommand{\ep}{\epsilon}
\newcommand{\eps}{\epsilon}

\newcommand{\leqc}{\lesssim}

\newcommand{\geqc}{\gtrsim}

\newcommand{\grad}{\nabla}

\newcommand{\strat}{\circ}

\newcommand{\lnorm}[1]{\| #1 \|}
\newcommand{\norm}[1]{\left\| #1 \right\|}
\newcommand{\abs}[1]{\left| #1 \right|}
\newcommand{\set}[1]{\left\{ #1 \right\}}
\newcommand{\brak}[1]{\left\langle #1 \right\rangle} 

\newcommand{\R}{\mathbb{R}}

\newcommand{\N}{\mathbb{N}}

\newcommand{\Z}{\mathbb{Z}}

\renewcommand{\S}{\mathbb{S}}

\newcommand{\tensor}{\otimes}
\newcommand{\loc}{\mathrm{loc}}

\newcommand{\cL}{\mathcal{L}}
\newcommand{\cJ}{\mathcal{J}}

\newcommand{\cM}{\mathcal{M}}

\newcommand{\xbf}{\mathbf{x}}

\newcommand{\dee}{\mathrm{d}}
\newcommand{\ds}{\dee s}
\newcommand{\dt}{\dee t}

\newcommand{\dx}{\dee x}

\newcommand{\dy}{\dee y}

\newcommand{\dq}{\dee q}

\newcommand{\ad}{\textup{ad}}

\DeclareMathOperator{\Div}{\mathrm{div}}

\DeclareMathOperator{\tr}{\mathrm{tr}}

\DeclareMathOperator{\Span}{\mathrm{span}}

\usepackage{bbm}

\renewcommand{\P}{\mathbf{P}}

\newcommand{\E}{\mathbf{E}}
\newcommand{\EE}{\mathbf E}
\newcommand{\PP}{\mathbf P}
\newcommand{\Leb}{\operatorname{Leb}}

\newtheorem{theoremAlpha}{Theorem}

\newtheorem{corollaryAlpha}[theoremAlpha]{Corollary}

\newtheorem{theorem}{Theorem}[section]
\newtheorem{proposition}[theorem]{Proposition}
\newtheorem{corollary}[theorem]{Corollary}
\newtheorem{lemma}[theorem]{Lemma}
\newtheorem*{lemma*}{Lemma}

\newtheorem{assumption}{Assumption}

\theoremstyle{definition}
\newtheorem{definition}[theorem]{Definition}
\newtheorem{remark}[theorem]{Remark}

\numberwithin{equation}{section}
    
\begin{document}

\title{A regularity method for lower bounds on the Lyapunov exponent for stochastic differential equations} 
\subjclass{Primary: 37H15, 35H10. Secondary: 37D25, 58J65, 35B65}
\author{Jacob Bedrossian\thanks{\footnotesize Department of Mathematics, University of Maryland, College Park, MD 20742, USA \href{mailto:jacob@math.umd.edu}{\texttt{jacob@math.umd.edu}}. J.B. was supported by National Science Foundation CAREER grant DMS-1552826 and National Science Foundation RNMS \#1107444} \and Alex Blumenthal\thanks{\footnotesize School of Mathematics, Georgia Institute of Technology, Atlanta, GA 30332, USA \href{mailto:ablumenthal6@gatech.edu}{\texttt{ablumenthal6@gatech.edu}}. A.B. was supported by National Science Foundation grant DMS-2009431} \and Sam Punshon-Smith\thanks{\footnotesize Division of Applied Mathematics,  Brown University, Providence, RI 02906, USA \href{mailto:punshs@brown.edu}{\texttt{punshs@brown.edu}}. This material was based upon work supported by the National Science Foundation under Award No. DMS-1803481.}}

\maketitle

\begin{abstract}
We put forward a new method for obtaining quantitative lower bounds on the top Lyapunov exponent of stochastic differential equations (SDEs). Our method
combines (i) a new identity connecting the top Lyapunov exponent to a Fisher information-like functional of the stationary density of the Markov process tracking tangent directions with (ii) a novel, quantitative version of H\"ormander's hypoelliptic regularity theory in an $L^1$ framework which estimates this (degenerate) Fisher information from below by a $W^{s,1}_{\mathrm{loc}}$ Sobolev norm. This method is applicable to a wide range of systems beyond the reach of currently existing mathematically rigorous methods. As an initial application, we prove the positivity of the top Lyapunov exponent for a class of weakly-dissipative, weakly forced stochastic differential equations; in this paper we prove that this class includes the Lorenz 96 model in any dimension, provided the additive stochastic driving is applied to any consecutive pair of modes. 
\end{abstract}

\setcounter{tocdepth}{1}
{\small\tableofcontents}

\section{Introduction}\label{sec:Intro}

Many nonlinear systems of physical origin exhibit chaotic behavior. Although there is no single mathematical definition of `chaos', a commonly observed and studied aspect of chaotic behavior is \emph{sensitivity with respect to the initial conditions}. One way to quantify this sensitivity is the \emph{Lyapunov exponent} along a given trajectory, i.e., the asymptotic exponential rate at which nearby trajectories converge (a negative exponent) or diverge (a positive exponent, implying sensitivity). 
Despite the ubiquity of chaotic behavior in systems of physical interest, and in contrast with the rather well-developed abstract theory for the description of chaotic states and associated statistical properties, it is notoriously challenging to verify, for a given system, whether or not a positive Lyapunov exponent is present along a large (e.g., positive-volume) subset of phase space. 

The purpose of this paper is to put forward a method for providing quantitative lower bounds for the Lyapunov exponents of weakly-damped, weakly-driven stochastic differential equations. Our method combines two new ingredients: 
\begin{itemize}
\item[(i)] an apparently new identity (Theorem \ref{thm:fishInfoIntro}) connecting the largest Lyapunov exponent to a certain degenerate Fisher information functional on the stationary density of the Markov process tracking tangent directions; and 
\item[(ii)] a quantitative hypoellipticity argument in $L^1$ for showing that this Fisher information uniformly controls local fractional Sobolev regularity of the tangent-direction stationary density (Theorem \ref{thm:unifHypoRegEstIntro}), implying by (i) a quantitative lower bound on the Lyapunov exponent in terms of the local regularity of this stationary density.
\end{itemize} 
As a first application of our methods, we prove positivity of the largest Lyapunov exponent for a broad class of weakly-dissipative weakly-driven SDE on $\R^n$ (with fluctuation dissipation scaling) with bilinear drift term that conserves both volume and the norm on $\R^n$ (Theorem \ref{thm:critForEulerLikeIntro}) . This result incorporates several interesting models of fundamental interest, such as the Lorenz-96 (L96) model (Corollary \ref{L96}) and Galerkin truncations of the Navier-Stokes equation (GNSE) on a periodic box (supplemented by a recent and separate work \cite{BPS21} of the first and third author; see the beginning of Section \ref{sec:Conc}). 
These results constitute the first mathematically rigorous results affirming a positive Lyapunov exponent for either of these models, even in the stochastic case. 

Below in Section \ref{subsec:nontechIntro} we give an informal, nontechnical overview of the main results (Theorems \ref{thm:fishInfoIntro} -- \ref{thm:critForEulerLikeIntro} and Corollary \ref{L96}), followed
by a brief discussion of context within prior work (Section \ref{Sec:Context}) and a discussion of the future outlook for the ideas introduced in this paper (Section \ref{sec:Conc}). See Section \ref{subsec:planIntro} for a brief outline of the rest of the paper. 

\subsection{Overview of results}\label{subsec:nontechIntro}

Although several of our most general results (Theorem \ref{thm:fishInfoIntro} and Theorem \ref{thm:unifHypoRegEstIntro}) apply to SDEs posed on essentially any (potentially non-compact) geodesically complete manifold $M$, in order to simplify the exposition and to fix ideas, here we will consider just SDEs posed on $\R^n$ (already sufficient for Theorem \ref{thm:critForEulerLikeIntro} and Corollary \ref{L96}). See the main body of the paper for more general statements. 

Let $n \geq 1$, and consider the stochastic process $(x_t)_{t \geq 0}$ in $\R^n$ defined by 
the Stratonovich SDE
\begin{equation}\label{eq:general-SDE}
	\dee x_t = X_0(x_t)\,\dt + \sum_{k=1}^r X_k(x_t)\strat\dee W^k_t \, ,
\end{equation}
where $\{X_k\}_{k=0}^r$ are a family of smooth vector fields on $\R^n$ and $\{W^k\}_{k=1}^r$ are independent standard Wiener processes.

Under mild conditions , trajectories of the Markov process $(x_t)$ are realized by an associated stochastic flow of diffeomorphisms $\Phi^t_\omega : \R^n \to \R^n$
indexed by the corresponding sample path $\omega \in \Omega$, where $\Omega$ is
the canonical path space for $\{ W^k\}$. 
 
Under the (natural) assumption that there is a unique stationary probability measure $\mu$
for $(x_t)$, it follows from standard tools in ergodic theory that the limits
\begin{align*}
	\lambda_1 & = \lim_{t\to \infty}\frac{1}{t} \log{|D_x\Phi_\omega^t|} \, , \\
	\lambda_\Sigma & = \lim_{t\to \infty} \frac{1}{t}\log{|\det D_x\Phi_\omega^t|} \,
\end{align*}
exist with probability 1 for $\mu$ almost every $x \in \R^n$, while the limiting values $\lambda_1, \lambda_\Sigma$
are constant (independent of $x$) and deterministic (independent of the random sample $\omega$); for more details, see Theorem \ref{thm:non-random-MET} below. 

The value $\lambda_1$ is known as the \emph{top Lyapunov exponent}; the condition $\lambda_1 > 0$
implies exponential sensitivity with respect to initial conditions, as well as local moving-frame saddle-type behavior, 
for the (random) trajectories of $\mu$-typical initial $x$ corresponding to almost every random sample $\omega \in \Omega$ (see \cite{young1995ergodic, barreira2002lyapunov}; see also \cite{arnold1995random, liu2006smooth, kifer2012ergodic} for emphasis on random dynamics). 
This nonlinear moving-frame behavior is the subject of \emph{smooth ergodic theory}, and leans on the Multiplicative Ergodic Theorem \cite{raghunathan1979proof, oseledets1968multiplicative, walters1993dynamical}, which, roughly speaking, provides a decomposition of the tangent bundle $T\R^d$ into (random) sub-bundles along which various exponential growth rates (Lyapunov exponents) are realized. The value $\lambda_\Sigma$ is the {\em sum Lyapunov exponent} and describes the asymptotic exponential rate at which Lebesgue volume is contracted/expanded by the dynamics.  For more information, see, e.g., the expositions \cite{young2013mathematical, wilkinson2017lyapunov}. 

\subsubsection*{Fisher information identity}

The purpose of this paper is to put forward a new method for obtaining lower bounds on $\lambda_1$ based on the regularity of a certain stationary density. Results are framed in terms of the augmented Markov process $(x_t, v_t)$ tracking a trajectory in phase space $(x_t)$ and the tangent direction
\begin{align}
v_t := \frac{D_x\Phi_\omega^tv}{|D_x\Phi_\omega^tv|},  \label{def:vt}
\end{align}
for a fixed initial unit vector $v_0 := v$. 
The process $(v_t)$ is natural in the study of Lyapunov exponents, on noticing by the chain rule that
convergence of the finite-time Lyapunov exponents $\frac1n \log | D_x \Phi^n_\omega v|$ is 
connected to a strong law of large numbers for the observable $g_\omega(x, v) = \log |D_x \Phi^1_\omega v|$ of the Markov process $(x_t, v_t)$. We call $(w_t) = (x_t, v_t)$ the 
\emph{projective process} on the unit tangent bundle\footnote{The distinction between $v_t$ or $-v_t$ is irrelevant for Lyapunov exponents, and so morally $(w_t)$ should be thought of as evolving on the projective bundle $P \R^n = \R^n \times P^{n-1}$, where $P^{n-1} = \S^{n-1} / \pm$ is the projective space of $\R^n$.   However, in this work we adopt the (technically more convenient) convention that $(w_t)$ lives on $\S \R^n$. }
 $\S \R^n = \R^n \times \S^{n-1}$. 
For more background, see, e.g., \cite{ledrappier1984quelques, walters1993dynamical}; see also 
 Theorem III.1.2 of \cite{kifer2012ergodic} and \cite{carverhill1986nonrandom} in the context of 
 random dynamical systems. 
The projective process $(w_t)$ solves a corresponding SDE on $\S \R^n$, 
\begin{equation}\label{eq:spherebundleeq}
	\dee w_t = \tilde{X}_0(w_t)\dt + \sum_{k=1}^r\tilde{X}_k(w_t)\strat \dee W^k_t \, ,
\end{equation}
where $\tilde{X}_j$ denotes the `lift' of the vector field $X_j$ to a vector field on $\S \R^n$ (see Section \ref{subsec:sde-lyap} for a precise definition). 
Our first main result is an identity connecting stationary densities of $(w_t)$
 to the exponent $\lambda_1$ through a partial \emph{Fisher information}-type 
 quantity defined for a probability density $f$ on $\S\R^n$ by
\[
FI(f) := \frac{1}{2}\sum_{k=1}^r \int_{\S M} \frac{|\tilde{X}_k^*f|^2}{f} \,\dee q \, , 
\]
where $\dee q$ is the (Riemannian) volume measure on $\S \R^n$ and $\tilde X^*$ denotes the $L^2(\dee q)$-adjoint of $\tilde X$ as a differential operator.

\begin{theoremAlpha}\label{thm:fishInfoIntro}
Assume $(w_t)$ admits a unique stationary density $f > 0$ on $\S \R^n$ satisfying some additional mild moment estimates (see Proposition \ref{prop:FI-gen-intro} for details). 
Let $\rho$ denote the corresponding stationary density of $(x_t)$ on $\R^n$. Then, 
\begin{align*}
	FI(f) & = n\lambda_1 - 2\lambda_{\Sigma}, \\ 
    FI(\rho)& = -\lambda_{\Sigma}.
\end{align*}
\end{theoremAlpha}
This formula can be interpreted as the time-infinitesimal analogue of the well-known relative entropy formula for Lyapunov exponents for SDE \cite{furstenberg1963noncommuting, baxendale1989lyapunov} (see also \cite{ledrappier1986positivity}) and in many ways provides a missing link between these relative entropy formulas and the Furstenberg-Khasminskii formula (Lemma \ref{lem:FK}; see also \cite{khasminskii2011stochastic, arnold1995random}). For more on the relative entropy formula, see equation \eqref{eq:RE} and the discussion in Section \ref{Sec:Context}. A complete statement and proof of Theorem \ref{thm:fishInfoIntro} 
is given in Section \ref{sec:fishStatementNew}. 

\subsubsection*{Fisher information and hypoelliptic regularity}

If, hypothetically, $\{\tilde X_1, ..., \tilde{X}_k\}$ were to span the tangent space of $\S \R^n$ everywhere
in a `uniform' way, then we would obtain the easy lower-bound
\begin{equation}\label{eq:easy-bound}
\norm{\grad f}_{L^1}^2 \lesssim \norm{f}_{L^1} FI(f) = FI(f) \, , 
\end{equation}
i.e., the Fisher information $FI(f)$ dominates $W^{1,1}$ regularity of the stationary density $f$ over all of $\S \R^n$, thereby providing a non-vanishing lower bound on the Lyapunov exponent. 
Unfortunately, for most cases of interest, the vectors $\{ \tilde X_k(w)\}_{k=1}^r$ do not span $T_w \S \R^n$ at each $w\in \S \R^n$ (even if $\{X_k(x)\}_{k=1}^r$ do span $T_x\R^n$), and so the functional $FI(f)$ only directly controls \emph{partial regularity} along the forced directions $\{ \tilde X_k\}_{k = 1}^r$ and is degenerate along unforced directions. In particular, there is no simple way to relate $FI(f)$ to, e.g., Sobolev regularity of $f$ in all directions on $\S\R^n$ without using more information about $f$. 

Our second main result shows that this degeneracy in $FI(f)$ can be resolved if one assumes that the degenerate elliptic Kolmogorov equation governing $f$,
\begin{equation}\label{eq:Hor-PDE}
    X_0^*f + \frac{1}{2}(X_k^*)^2f =0,
\end{equation}
is {\em hypoelliptic}.
Hypoellipticity of this kind was studied by H\"ormander \cite{H67}, who showed that under a checkable \emph{Lie bracket spanning} 
condition on the Lie algebra generated by the vector fields $\{ \tilde X_k\}_{k = 0}^r$ at each point $w\in \S \R^n$, 
regularity of $f$ along the $\{ \tilde X_k\}_{k =1}^r$ directions propagates to regularity along \emph{all} directions in $\S \R^n$. 
One of our main contributions in this paper is an analogous and quantitative version of H\"ormander's result for the Fisher information $FI(f)$: assuming $\{\tilde X_k\}_{k = 0}^r$ 
satisfies the parabolic version of H\"ormander's bracket spanning condition (Definition \ref{def:Hormander} below) and that $f$ is a stationary probability density for $(w_t)$ solving the PDE \eqref{eq:Hor-PDE}, we establish that there exists a regularity parameter $s \in (0,1)$ such that for any open, bounded ball $U \subset \S \R^n$, there exists a constant
$C = C_U > 0$ with
\begin{align}\label{eq:nonQuantHypoRegSec1}
\| \chi_U f\|_{W^{s, 1}}^2 \leq C( 1 + FI(f)) \, , 
\end{align}
where $\chi_U$ is a smooth cutoff subordinate\footnote{i.e. $\exists V \subset \S\R^n$ open and bounded such that $\overline{U} \subset \subset V$, $\chi_U(x) = 1$ for $x \in U$ and $\chi_U(x) = 0$ for $x \not\in V$. The constant in \eqref{eq:nonQuantHypoRegSec1} will depend on the choice of $V$ and $\chi_U$, but this dependence has been suppressed as it is unimportant.} to $U$.

Equation \eqref{eq:nonQuantHypoRegSec1} is not immediately useful in isolation, as it is unclear how to directly use this to obtain information about $n\lambda_1$ from $FI(f)$. For instance, the constant $C$ depends on the details of precisely how the Lie brackets are put together to span the tangent space and is difficult to control. Moreover \eqref{eq:nonQuantHypoRegSec1} does not preclude the vanishing of $FI(f)$ as an inequality like \eqref{eq:easy-bound} would.
However, we can obtain a far more useful estimate applicable for SDE in small noise regimes, e.g., 
\begin{align}\label{eq:smallNoiseSDESec11111}
\dee x_t^\eps = X_0^\eps(x_t^\eps)\,\dt + \sqrt{\eps} \sum_{k=1}^r X_k^\eps(x_t^\eps)\strat\dee W^k_t \, ,
\end{align}
where $\epsilon \in (0,1)$ is a small parameter and $\{ X_k^\epsilon\}_{k = 0}^r$ is a parametrized family of uniformly smooth vector fields. In this context, the Fisher information identity in Theorem \ref{thm:fishInfoIntro} reads
\begin{align}
FI(f^\epsilon) = \frac{n \lambda_1^\epsilon - 2 \lambda_\Sigma^\epsilon}{\epsilon} \, , \label{eq:FIintroEps}
\end{align}
provided that for all $\epsilon$, the corresponding projective process $(x_t^\epsilon, v_t^\eps)$
admits a unique stationary density $f^\epsilon$ on $\S \R^n$. 
Our second main result is that assuming that $\{\widetilde{X}_0^\eps,...,\widetilde{X}_r^\ep\}$ satisfies H\"ormander's hypoellipticity condition uniformly-in-$\eps$ (locally in $x$), we may obtain an estimate which is independent of $\eps$.  

\begin{theoremAlpha}\label{thm:unifHypoRegEstIntro}
Assume that $\{ \tilde X_k^\epsilon\}_{k = 0}^r$ satisfies H\"ormander's parabolic bracket spanning condition uniformly in $\epsilon \in (0,1)$ (see Definition \ref{def:UniHormander}). Then, there exists $s \in (0,1)$ and for each bounded open set of the form $U := B_R(x_0) \times \S^{n-1} \subset \S \R^n$, a constant $C = C_U > 0$ (independent of $\ep>0$) such that for all $\epsilon \in (0,1)$ and for any absolutely continuous stationary measure $f^\eps$ of $(x_t^\eps, v_t^\eps)$, we have the estimate 
\[
\norm{\chi_U f^\eps}_{W^{s,1}}^2 \leq C( 1 + FI(f^\epsilon)) \, . 
\]
\end{theoremAlpha}

In view of \eqref{eq:FIintroEps}, it is clear that Theorem \ref{thm:unifHypoRegEstIntro} is most useful for weakly-damped systems, for example, systems where $\lambda_\Sigma = -O(\eps)$.
Crucially however, it does not have to be exactly zero. In this manner, we can treat systems which are close  to, but not exactly, volume preserving (see the discussion in Section \ref{Sec:Context} for more detail). 

The proof of Theorem \ref{thm:unifHypoRegEstIntro} is carried out in Section \ref{sec:Hor} and requires a significant re-working of H\"ormander's work to pass from the original $L^2$ framework to an $L^1$-compatible
framework more suited to estimates involving Fisher information. Of course, there is a large literature of works extending H\"ormander's theory in various ways, e.g., to handle rough coefficients: we refer the reader to, e.g., \cite{Polidoro16,Anceschi2019,LanconelliEtAl2020,FarhanGiulio19,GIMV16,Mouhot2018,BL20} and the references therein.
However, as far as the authors are aware, there are no previous works that fundamentally rework the theory into $L^1$.

\subsubsection*{Application to ``Euler-like'' bilinear systems} \label{sec:EulerLikeIntro}

We anticipate that Theorems \ref{thm:fishInfoIntro} and \ref{thm:unifHypoRegEstIntro} are applicable to a wide class of weakly damped, weakly-driven
SDE. In this manuscript, we present an application to a natural class of examples 
with bilinear drift term, which we call \emph{Euler-like} systems: 
\begin{align}\label{eq:SDEintroFD22}
\dee x_t^\epsilon = (B(x_t^\epsilon,x_t^\eps) + \epsilon A x_t^\epsilon)\dt + \sum_{k=1}^r X_k \dee W_t^{k} \, .
\end{align}
Here, $\set{X_k}_{k = 1}^r$ is a collection of constant ($x$-independent) forcing vector fields, while $B : \R^n \times \R^n \to \R^n$ is bilinear, nontrivial (not identically zero), and satisfies
\[
\Div B = 0 \, , \quad x \cdot B(x,x) = 0 \, , 
\]
so in particular the $\epsilon = 0$ dynamics is norm (``energy'') and volume preserving. 
Meanwhile, the term $\epsilon A$ provides weak linear damping, where
 $A$ is assumed to be a symmetric, negative-definite $n \times n$ matrix. 
Stochastically forced versions of both Lorenz 96 and Galerkin truncations of 2d and 3d Navier-Stokes on a torus can be cast in this form. 

\begin{remark}
Note the absence of $\sqrt{\epsilon}$ in front of the noise term in \eqref{eq:SDEintroFD22}; 
in this class of models, the dynamics is subjected to weak dissipation of order $\epsilon$ 
at a constant level of noisy driving. 
However, due to the bilinear form of the nonlinearity in the drift term, 
\eqref{eq:SDEintroFD22} is equivalent  to the weakly-driven, weakly damped form
\begin{align}\label{eq:SDEintroFD22resc}
\dee x_t^\epsilon = (B(x_t^\epsilon,x_t^\eps) + \epsilon A x_t^\epsilon)\dt + \sqrt{\eps}\sum_{k=1}^r X_k \dee W_t^{k}, 
\end{align}
 by rescaling $x_t \mapsto \sqrt{\eps} x_{\sqrt{\eps} t}^\eps$, 
replacing $\eps \mapsto \eps^{3/2}$, and using the self-similarity of Brownian motion. 
This rescaling does not affect our results on 
Lyapunov exponents, since upon setting $\hat \eps = \eps^{3/2}$, the Lyapunov exponent $\hat \lambda^{\hat \eps}_1$ of \eqref{eq:SDEintroFD22resc} with parameter $\hat \epsilon$ is related to the Lyapunov exponent $\lambda_1^\eps$ of \eqref{eq:SDEintroFD22} by the identity 
\[\frac{\hat {\lambda}_1^{\hat \eps}}{\hat \eps} = \frac{\lambda_1^\eps}{\eps}.
\] 
Moreover, while \eqref{eq:SDEintroFD22} is common among models of complex real-world
systems, we note that in this scaling, the stationary measure $\mu$
 has characteristic energy $\int \abs{x}^2 \dee\mu(x) \approx \eps^{-1}$. Since we are concerned
with the regime $\epsilon \ll 1$, it is natural to consider the weakly-damped, weakly-driven
rescaling \eqref{eq:SDEintroFD22resc} which has characteristic energy $O(1)$.
\end{remark}

For this class of systems \eqref{eq:SDEintroFD22}, we give a sufficient condition for a positive Lyapunov exponent 
in terms of the bracket spanning condition for the lifted vector fields $\{\tilde X_0^\epsilon, \tilde X_1,\ldots \tilde X_r\}$
on $\S \R^n$ corresponding to the projective process $(x_t^\epsilon, v_t^\epsilon)$, where here
we follow the convention $X_0^\epsilon(x) = B(x,x) + \epsilon A x$. 
\begin{theoremAlpha}\label{thm:critForEulerLikeIntro}
Assume $\{\tilde X_k\}_{k = 0}^r$ satisfies H\"ormander's parabolic bracket spanning condition uniformly in $\epsilon \in (0,1)$, and additionally, 
assume that for all $\epsilon \in (0,1)$ we have that $(x_t^\epsilon, v_t^\epsilon)$
admits a unique stationary density $f^\epsilon$. Then, 
\begin{align*}
\lim_{\eps \to 0} \frac{\lambda_1^\eps}{\eps} = \infty \, . 
\end{align*}
\end{theoremAlpha}
The basic idea of the proof is as follows. If $\liminf \lambda^\epsilon_1/\epsilon < \infty$, then 
Theorem \ref{thm:unifHypoRegEstIntro} implies that the 
stationary densities $f^\epsilon$ of the projective process are uniformly bounded in $W^{s, 1}_\loc$ along some subsequence of $\epsilon \to 0$. Using the locally compact embedding of $W^{s, 1}$ in $L^1$ (Lemma \ref{lem:Ws1Precompact}), one easily establishes that the sequence of $f^\eps$ are precompact\footnote{We stress that it is not enough that $\{f^\epsilon \dq\}$ is merely tight as a sequence of measures on $\S\R^n$, which would not ensure that the limiting measure has a density . Rather, we vitally use strong $L^1$ convergence coming from the uniform regularity estimate to ensure that the limiting measure we obtain has a density w.r.t. $\dq$. Of course, one could replace strong convergence in $L^1$ with weak $L^1$ convergence via an weaker uniform integrability assumption and still obtain a limit in $L^1$.} in $L^1$, and on refining the subsequence
we can find an $L^1$ limit $f^0$, which as one easily checks, is an invariant probability density
for the \emph{zero noise, deterministic} projective process $(x_t^0, v_t^0)$. In contrast with the noisy setting, the existence of an invariant density for the projective process of a deterministic flow is 
extremely rigid and is equivalent, roughly speaking, to the flow being an isometry with respect to a certain Riemannian metric (Theorem \ref{sec:prelimEuler5}). This can be ruled out for the zero-noise flow due to shearing between energy surfaces $\{ |x| = \textrm{const}\}$ (a natural consequence of the quadratic nonlinearity), ruling out the possibility that $(D_x \Phi^t)$ is an isometry.
Hence, the proof of Theorem \ref{thm:critForEulerLikeIntro} is by a compactness-rigidity argument. 
For more details and the full proof of Theorem \ref{thm:critForEulerLikeIntro}, see Section \ref{sec:Inviscid-Limit}. 

In this manuscript we confirm the sufficient conditions in Theorem \ref{thm:critForEulerLikeIntro} for the Lorenz 96 (L96) model\footnote{Note that L96 is distinct from the Lorenz 63 ``butterfly attractor'' model, an ODE on $\R^3$, introduced in Lorenz's seminal 1963 work \cite{lorenz1967nature}. 
}
\begin{align} \label{def:L96intro}
\dee u_m =  \big((u_{m+1}-u_{m-2})u_{m-1} - \epsilon u_m\big) \dt + q_m \dee \hat{W}_t^{m} \, , \quad 1 \leq m \leq J,
\end{align}
on $\R^J$, where $\set{q_m}$ are fixed parameters and the $u_m$ are $J$-periodic in $m$, i.e., $u_{m + k J} := u_m$ \cite{Lorenz1996}. 

\begin{corollaryAlpha}\label{L96}
Assume $J \geq 7$ and that $q_1, q_2 \neq 0$.
Then, the top Lyapunov exponent $\lambda^\epsilon_1$ of the L96 model \eqref{def:L96intro}
satisfies $\lambda_1^\epsilon/\epsilon \to \infty$ as $\epsilon \to 0$. In particular, $\lambda_1^\epsilon > 0$
for all $\epsilon$ sufficiently small. 
\end{corollaryAlpha}

Remarkably, the problem of proving $\lambda^\epsilon_1 > 0$ for L96 was previously open in spite of overwhelming numerical evidence to support this \cite{Majda16,KP10,PazoEtAl08,BCFV02,OttEtAl04}. 
To the best of the authors' knowledge, ours is the first mathematically rigorous result in this direction. Moreover, $\lambda_1^\epsilon > 0$ does not follow from existing techniques, e.g., Furstenberg's criterion for the Lyapunov exponents of a random system; see Section \ref{Sec:Context} for more discussion. 
See also the beginning of Section \ref{sec:Conc} for discussion of a very recent application of Theorem \ref{thm:critForEulerLikeIntro} to Galerkin truncations of the 2d Navier-Stokes equations on a periodic box in the recent work \cite{BPS21} of the first and third authors of this paper. See Section \ref{sec:Inviscid-Limit} 
for the proof of Corollary \ref{L96}. 


\subsection{Context within prior work} \label{Sec:Context}

As remarked earlier, for a given system it can be extremely challenging to estimate its Lyapunov exponents and provide a mathematically rigorous account of its time-asymptotic behavior. Indeed, in principle Lyapunov exponents require infinitely precise information on infinitely many trajectories, and in practice the convergence of Lyapunov exponents to their `true' values can exhibit long stretches of intermittent behavior. This is especially so for deterministic systems in the absence of stochastic driving, for which one anticipates that ``chaotic'' and ``orderly'' regimes coexist in a convoluted way in both phase space as well as `parameter space', i.e., as the underlying dynamical system is varied: we refer the interested reader to, e.g., work on Newhouse phenomena in dissipative systems \cite{newhouse1979abundance, newhouse1974diffeomorphisms}; the proliferation of elliptic islands in volume-preserving systems \cite{duarte1999abundance}; known coexistence of chaotic and ordered regimes for the quadratic map family \cite{lyubich2002almost}; and $C^1$ generic dichotomies \cite{bochi2002genericity, bochi2005lyapunov}. For more background on this rich topic, see, e.g., \cite{young2013mathematical, pesin2010open, crovisier2018, wilkinson2017lyapunov}. 

Although it still presents significant challenges, the situation for Lyapunov exponents of stochastically forced systems is notably more tractable. To start, let us first address the body of work \`a la Furstenberg which describes necessary conditions for `degeneracy' of the Lyapunov exponents of a random dynamical system. As before, consider a stochastic flow of diffeomorphisms $\Phi^t_\omega$ on $\R^n$ arising from an SDE. Let $\mu$ be the (presumed unique) stationary measure for $x_t := \Phi^t_\omega(x_0)$, and  let $\lambda_1, \lambda_\Sigma$ be the corresponding Lyapunov exponents. 

Note that unconditionally we have $n \lambda_1 - \lambda_\Sigma \geq 0$. In this context, and brushing aside technical details, the criterion \`a la Furstenberg is due to a variety of authors (e.g., \cite{carverhill1987furstenberg,virtser1980products, royer1980croissance, ledrappier1986positivity}), and can be stated as follows: if $\nu \in \mathcal{P}(\S M)$ is a stationary measure for the projective process $(x_t, v_t)$ and $\dee\nu(x,v) = \dee\nu_x(v) \dee \mu(x)$ the disintegration of $\nu$, then for all $t > 0$ there holds (c.f. Theorem \ref{thm:baxFormulaSec2})
\begin{align}
\EE \int_{M} H( D_x\Phi^t_\ast \nu_{x} | \nu_{\Phi^t(x)} )\, \dee\mu(x) \leq t \left(n\lambda_1 - \lambda_{\Sigma}\right), \label{eq:RE}
\end{align}
where $H$ denotes the relative entropy, defined for any two probability measures $\eta,\lambda$ with $\eta \ll \lambda$, by
\begin{align}\label{eq:defnRelEntropyIntro}
	H(\eta|\lambda) := \int \log \left( \frac{\dee \eta}{\dee \lambda} \right) \,\dee \eta \,,
\end{align}
Unconditionally, relative entropy satisfies $H(\eta | \lambda) \geq 0$, and $H(\eta | \lambda) = 0$ iff
$\eta = \lambda$ by Jensen's inequality. 
From this we see that either 
\begin{align}\label{furstLyapIneq}
n \lambda_1 - \lambda_\Sigma > 0 \, , 
\end{align}
or the probabilistic law governing the stochastic flow admits a strong `degeneracy' in the sense that
\begin{align}\label{furstDegen}
(D_x \Phi^t_\omega)_* \nu_x = \nu_{\Phi^t_\omega(x)},
\end{align}
with probability 1 for all $t \geq 0$ and $\mu$-typical $x$. 
That this situation is very `degenerate'
follows from the fact that for fixed $x$ and $t$, the above right-hand side depends only on the time$-t$ position $\Phi^t_\omega(x)$, while the left-hand side depends additionally on the entire noise path 
$\omega|_{[0,t]}$. 

Observe that in the weakly-damped, weakly-driven setting of \eqref{eq:SDEintroFD22}, $\lambda_\Sigma^\epsilon = \epsilon \tr A < 0$ and so \eqref{furstLyapIneq} is agnostic as to whether $\lambda^\epsilon_1 > 0$ or not. Indeed, the techniques in the above-mentioned works are ``soft'' as the identity \eqref{furstDegen} is non-quantitative in the parameters of the underlying system. Although \eqref{eq:RE} does at least provide some kind of formula for $n \lambda_1 - \lambda_\Sigma$, it is unclear how to glean useful quantitative information directly from \eqref{eq:RE}. 

Interestingly, our Fisher-information identity in Theorem \ref{thm:fishInfoIntro}, (specifically \eqref{eq:FIcond} below), is essentially the time-infinitesimal analogue of \eqref{eq:RE}, as we show below in Section \ref{subsec:relativeEnt}. Hence, like \eqref{eq:RE}, our Theorem \ref{thm:fishInfoIntro} admits an interpretation in terms of the \emph{rate at which the degeneracy \eqref{furstDegen} fails to hold} for the stochastic flow $\Phi^t_\omega$. However, Theorem \ref{thm:fishInfoIntro} recasts the information in terms of the generator of $(w_t)$, which is more amenable now to the use of hypoelliptic PDE methods such as those employed in Theorem \ref{thm:unifHypoRegEstIntro}. This motivates the claim that the methods in this paper constitute a first step towards a quantitative \`a la Furstenberg theory. We remark that Fisher information-type quantities also commonly appear as the time derivatives of the relative entropy in the study of gradient flows and logarithmic Sobolev inequalities (see e.g. \cite{BakryEmery85,RV08,LedouxEtAl2015,Toscani1999}).

Beyond Furstenberg's criterion and its descendants, there is by now a large literature on the Lyapunov exponents of particular models for which we cannot do justice in this space. Instead, we will focus on a class of results most closely related to ours (Theorem \ref{thm:critForEulerLikeIntro}): small-noise expansions of Lyapunov exponents for weakly-driven stochastic systems. To frame the discussion, consider the abstract linear SDE
\begin{align}\label{contextGenSDE}
\dee V_t = A^\epsilon_t V_t \dee t + \sqrt{\epsilon} \sum_{k = 1}^r B_t^{k} V_t \circ \dee W_t^{k} \, , 
\end{align}
where $A_t^\epsilon, B^{k}_t$ are, in general, time-varying and/or themselves randomly driven, and $A_t^\epsilon$ may or may not exhibit some vanishingly weak damping as $\epsilon \to 0$. There are many works studying the scaling behavior of Lyapunov exponent $\lambda_1^\epsilon := \lim_{t \to \infty} \frac{1}{t} \log |V_t|$ of such systems, e.g., \cite{ausMilstein, pardoux1988lyapunov, pinsky1988lyapunov, imkeller1999explicit, moshchuk1998moment} in the constant coefficient case, and \cite{APW1986, baxendale2002lyapunov, baxendale2004vanderpol, malicet2020lyapunov} when the $A_t, B_t^{k}$ are coupled to some other stochastic process. To the authors' best knowledge, however, all of these results are restricted to settings where the $\epsilon = 0$ dynamics are relatively simple and essentially completely known. In comparison, our results are indifferent to any detailed description of the zero-noise dynamics. On the other hand, the sacrifice for our level of generality is that our estimate $\lambda^\epsilon_1 / \epsilon \to \infty$ is far weaker than an asymptotic expansion, and is likely to be sub-optimal for many models of interest. 

Of particular interest is that among models of the form \eqref{contextGenSDE}, scaling laws of the form $\lambda^\epsilon_1 \sim \epsilon^{\gamma}, \gamma \geq 1$ tend to be associated with zero-noise dynamics which are \emph{rigid isometries} (exhibiting no shearing) \cite{ausMilstein, pardoux1988lyapunov, baxendale2003lyapunov, baxendale2004vanderpol}. 
Meanwhile, laws of the form $\lambda^\epsilon_1 \sim \epsilon^\gamma, \gamma < 1$ are associated with zero-noise dynamics exhibiting some shearing mechanism. By way of example, \cite{ausMilstein, pinsky1988lyapunov} derive such scaling laws when $A_t$ as above is given by
\[
A_t \equiv \begin{pmatrix} 0 & 1 \\ 0 & 0 \end{pmatrix} \, , \quad B_t \equiv \begin{pmatrix} 0 & 0 \\ 1 & 0 \end{pmatrix} \, ,
\]
corresponding to the constant application of a horizontal shear in conjunction with a small, stochastically driven vertical shear. This analysis was extended to the setting of fluctuation-dissipation zero-noise limits of certain 2d completely integrable Hamiltonian systems in the work \cite{baxendale2002lyapunov}. 
Despite the consideration of high dimensional systems which are unlikely to be completely integrable, our Theorem \ref{thm:critForEulerLikeIntro} is in fact related to these works, in the sense that
the scaling $\lambda^\epsilon_1 / \epsilon \to \infty$ is derived by taking advantage of \emph{shearing
between the energy surfaces} $|x| = R, R > 0$ (see Section \ref{sec:prelimEuler5} for details).

Of course, shearing has long been regarded as a potential mechanism for the generation of chaotic behavior. As early as the late 70's it was realized that chaotic attractors could arise from time-periodic driving of a system undergoing a Hopf bifurcation \cite{zaslavsky1978simplest}, while subsequent mathematically rigorous work has confirmed this mechanism (see, e.g., \cite{wang2008toward} for an overview of this program). We also point out the work \cite{lin2008shear}, which provides a mix of heuristics, numerics, mathematical analysis and conjectures demonstrating the shearing mechanism as a source of chaotic behavior, as well as the works \cite{doan2018hopf, engel2019bifurcation} resolving some of the open problems in \cite{lin2008shear} 
pertaining to white-noise models. 

\subsection{Conclusion and outlook for future work} \label{sec:Conc}

We provide here some discussion on potential future extensions and applications of the ideas in this paper.

\medskip
\noindent {\bf Applications to different systems. }
As mentioned earlier, the class of Euler-like systems \eqref{eq:SDEintroFD22} includes 
Galerkin truncations of the Navier-Stokes equation on the periodic box in 2d and 3d, while
Theorem \ref{thm:critForEulerLikeIntro} provides a positive lower bound on the top Lyapunov exponent
provided that the corresponding projective process satisfies a uniform H\"ormander bracket spanning condition. Recently, the first and third author \cite{BPS21} affirmed this spanning condition 
for the 2d Galerkin-Navier-Stokes equations with additive forcing on rectangular 2d tori when only a small number of modes are directly forced and assuming that the Galerkin truncation is sufficiently large (i.e. the results hold for all sufficiently high dimensions). 

In contrast with the Lorenz 96 model, the coupling between modes of Galerkin truncations of PDE models 
is more `global', with low modes strongly coupled to high ones, and so for models of this kind it is substantially more difficult to check the projective H\"ormander's condition; the proof in \cite{BPS21} is in fact computer assisted, but applies to all sufficiently high dimensional Galerkin truncations and contains ideas that should be useful for verifying the projective spanning for other (sufficiently high dimensional) Galerkin truncations of PDEs. 
However, these ideas alone are likely inadequate for other kinds of high dimensional systems of interest, such as many-particle systems with weak dissipation. 

Another class of systems of interest is small random perturbations of completely integrable Hamiltonian flows, where estimates of the form $\lambda_1^\eps \gg \eps$ are far closer to optimal. For these, one
would replace shearing between energy shells as in the proof of Theorem \ref{thm:critForEulerLikeIntro}
with shearing between invariant tori, which in practice can be read directly off of the action-angle coordinates
for the system.

\medskip
\noindent {\bf Tighter hypoelliptic regularity estimates. } It is of natural interest to 
attempt to improve the scaling $\lambda_1^\eps \gg \eps$ that naturally falls out from 
Theorem \ref{thm:unifHypoRegEstIntro}. One way to do this is to 
attempt to strengthen the hypoelliptic regularity estimate by refining 
	the $\epsilon$ scaling to derive something like
	\[
	\| f^\eps \|_{W^{s, 1}}^2 \lesssim 1 + \frac{n \lambda_1^\eps - 2 \lambda_\Sigma^\eps}{\eps^{1 - \gamma}},
	\]
for some constant $\gamma > 0$. If such an estimate were true, the same compactness-rigidity argument 
of Theorem \ref{thm:critForEulerLikeIntro} would imply a scaling like $\lambda^\eps_1 \gtrsim \eps^{1 - \gamma}$, a significant improvement. For more discussion on potential improvements to 
Theorem \ref{thm:unifHypoRegEstIntro}, see the discussion in Section \ref{sec:Hor}. 

\medskip
\noindent {\bf Beyond compactness-rigidity.} Another way to strengthen the results of this paper
is to work exclusively with $\eps > 0$ without passing the limit as $\eps \to 0$. This was essentially
the approach of works \cite{pinsky1988lyapunov, baxendale2002lyapunov} estimating Lyapunov exponents for systems for which a nearly-complete understanding of the pathwise random dynamics was available. Although we largely lack such detailed information about the pathwise dynamics of high-dimensional models such as L96, there is some hope for a `middle ground', e.g., partial information such as some finite-time exponential growth mechanism resulting in a lower bound on $\norm {f^\eps}_{W^{s_*, 1}}$. An approach with a 
similar flavor for random perturbations of discrete-time systems, including the Chirikov standard map, was
carried out in the previous work \cite{blumenthal2017lyapunov}. 

\medskip
\noindent {\bf Finer dynamical information: moment Lyapunov exponents. }
Lyapunov exponents themselves provide asymptotic exponential growth rates, but the 
timescales along which these rates are realized can be quite long. Some quantitative control
is provided by large deviations principles in the convergence of the sequences $\frac1t \log | D_x \Phi^t_\omega (v)|$, the rate function of which is the Legendre transform of the 
\emph{moment Lyapunov exponent function} $p \mapsto \Lambda(p) := \lim_{t \to \infty} \frac1t \log \E |D_x \Phi^t_\omega v|^p$ (the limit defining $\Lambda(p)$ exists and is independent of $(x, v)$ under
natural conditions; see, e.g., \cite{arnold1984formula}). It would be highly interesting to see if ideas similar to
those presented in this work could provide quantitative estimates on the moment Lyapunov exponents
of weakly-driven systems. 

\medskip
\noindent {\bf More general noise models. }
One can also expand the noise models to which our work applies. One simple example of this is to extend Theorem \ref{thm:critForEulerLikeIntro} to different types of multiplicative noise.
While Theorems \ref{thm:fishInfoIntro} and \ref{thm:unifHypoRegEstIntro}  apply to multiplicative noise, certain aspects of Theorem \ref{thm:critForEulerLikeIntro} are specialized to additive noise, such as the arguments for projective spanning and irreducibility in Section \ref{ProjBraks} (and in \cite{BPS21}).  
Another extension would be to noise models which are not white-in-time, e.g., jump processes or the models used in \cite{KNS20,KNS20II}.
Our work is clearly deeply tied to the infinitesimal generator $\mathcal{L}$, and so any such extension will  be rather non-trivial. 
A simpler candidate for non-white forcing is constructed from `towers' of coupled Ornstein-Uhlenbeck processes, which can be built to be $C^k$ in time for any $k \geq 0$ (see, e.g., \cite{bedrossian2018lagrangian} for details). While such a noise model is built out of SDE and so largely falls under the purview of the analysis in this paper, it is not clear how to prove the analogue of Theorem \ref{thm:critForEulerLikeIntro}. This will be considered in future research. 

\subsection{Plan for remainder of the paper}\label{subsec:planIntro}

Preliminaries on the SDE setting in this paper are provided in Section \ref{sec:FI}, while
the complete statement and proof of the Fisher information identity in Theorem \ref{thm:fishInfoIntro}
is provided in Section \ref{sec:fishStatementNew}. 
In Section \ref{sec:Hor} we turn attention to the full statement and proof of the hypoelliptic regularity estimate in Theorem \ref{thm:unifHypoRegEstIntro}. 
The remainder of the paper is devoted to the proof of Theorem \ref{thm:critForEulerLikeIntro} and 
Corollary \ref{L96}: in Section \ref{ProjBraks} we address how some of the assumptions 
of Theorem \ref{thm:fishInfoIntro} and \ref{thm:unifHypoRegEstIntro} (projective spanning and uniqueness of stationary densities) are checked in practice, while the compactness rigidity argument is carried out in Section \ref{sec:Inviscid-Limit}.

\section{Preliminaries}\label{sec:FI}

We present here preliminaries on SDE, including a discussion of stationary measures and their properties; 
the existence of Lyapunov exponents; and remarks on the projective process, a crucial tool in Theorem \ref{thm:fishInfoIntro}. 

\medskip

Throughout, $(M, g)$ is a smooth, connected, geodesically complete and orientable Riemannian manifold  without boundary (not necessarily bounded); let $n = \dim M$. Throughout, we abbreviate 
the Riemannian / Lebesgue volume on $(M,g)$ by $\dee x$. Here, and everywhere below unless specified otherwise, we use the notation $\brak{a,b}_x = g_x(a,b)$ for $a,b \in T_x M, x \in M$.

We consider the stochastic process
\begin{equation}\label{eq:general-SDE}
	\dee x_t = X_0(x_t)\,\dt + \sum_{k=1}^r X_k(x_t)\strat\dee W^k_t \, , 
\end{equation}
where $\{X_k\}_{k=0}^r$ are a family of smooth vector fields (potentially unbounded) on $M$ and $\{W^k\}_{k=1}^r$ are independent standard Wiener processes with respect to a canonical stochastic basis $(\Omega,\mathscr{F},(\mathscr{F}_t),\P)$.

\subsection{Background on SDEs}\label{subsec:sde-lyap}

The following are standing hypotheses imposed throughout the paper, and ensure the existence with probability 1 of the limits defining Lyapunov exponents. 

 \begin{assumption} \label{ass:WP} \hspace{1em}
\begin{enumerate}
	\item[(i)] For each initial datum $x\in M$, equation \eqref{eq:general-SDE} has a unique global solution $(x_t)$ with probability 1. The (random) solution maps $x \mapsto x_t =: \Phi^t_\omega(x), t \geq 0$ comprise a (stochastic) flow of $C^r$ diffeomorphisms $(\Phi^t_\omega)$ on $M$, $r \geq 2$. 
 \item[(ii)] The Markov process $(x_t)$ admits a unique, absolutely continuous stationary probability measure $\mu$ on $M$. We write $\rho = \frac{\dee \mu}{\dee x}$ for the density of $\mu$. 
 \item[(iii)] The measure $\mu$ has the integrability condition 
 \[
 	\E \int_M \left[\log^+{|D_x\Phi^{t}|} + \log^+{|(D_x\Phi^{t})^{-1}|} \right]\dee \mu(x) < \infty \, . \footnote{Here, for $a > 0$ we write $\log^+a := \max\{ \log a, 0\}$ for the positive part of $\log$. }
 \]
\end{enumerate}
\end{assumption}
Although not automatic, Assumption \ref{ass:WP}(i) is well-studied and follows from mild conditions on \eqref{eq:general-SDE}, although careful checking is required in the case when $M$ is noncompact and the $\{ X_k\}$ are unbounded; see, e.g., \cite{kunita1997stochastic, arnold1995random}. Given (i) and (ii), item (iii) follows under mild moment assumptions  \cite{kifer1988note}.

For (ii), we separately address existence, absolute continuity and uniqueness. Existence of stationary measures when $M$ is compact is immediate from compactness in the narrow topology on the space of probability measures on $M$, while when $M$ is noncompact some additional assumption
is needed to ensure tightness as $t \to \infty$ of the distribution $\mathrm{Law}(x_t)$ of $x_t$, usually via a Lyapunov-Foster drift condition \cite{meyn2012markov}. 

Absolute continuity is usually checked using H\"ormander's parabolic condition for the vector fields $\{ X_k\}$, which we briefly recall below. 
For a manifold $\cM$, let $\mathfrak{X}(\cM)$ denote the set of smooth vector fields on $\cM$. Elements $X \in \mathfrak X(\cM)$ are regarded in the usual way as first-order differential operators acting on observables $w : \cM \to \R$ via $Xw = dw(X)$.
For vector fields $X,Y$, we write $[X,Y]$ for the standard Lie bracket of $X$ and $Y$.  
\begin{definition} \label{def:Hormander}
  Given a collection of vector fields $Z_0,Z_1,\ldots, Z_r$ on a manifold $\mathcal{M}$, we define collections of vector fields $\mathscr{X}_0\subseteq \mathscr{X}_1 \subseteq\ldots$recursively by
\[
\begin{aligned}
  &\mathscr{X}_0 = \{Z_j\, :\, j \geq 1\},\\
  &\mathscr{X}_{k+1} = \mathscr{X}_k \cup \{[Z_j,Z]\,:\, Z \in \mathscr{X}_k, \quad j \geq 0\}.
  \end{aligned}
\]
We say that $\{Z_i\}_{i=0}^r$ satisfies the {\em parabolic H\"ormander condition} (also called bracket spanning) if there exists $k$ such that for all $w \in \mathcal{M}$,
\begin{align}\label{eq:spanningCondIntro22}
    \mathrm{span}\left\{Z(w)\,:\, Z\in \mathscr{X}_k\right\} = T_w \mathcal{M}. 
\end{align}
\end{definition}
Next, we recall a version of H\"ormander's theorem for transition kernels.
\begin{theorem}[H\"ormander's theorem \cite{H67}]
Assume $\{ X_k\}_{k = 0}^r$ satisfy the parabolic H\"ormander condition. Then, the transition 
kernels $P^t(x, K) := \P(x_t \in K | x_0 = x)$ are all absolutely continuous with respect to $\dee x$. 
\end{theorem}

Given existence and absolute continuity, the Doob-Khasminskii theorem (see e.g. \cite{DPZ96}) ensures 
that uniqueness follows if we have some additional topological 
irreducibility condition ensuring that all initial conditions in $M$ can ``access'' any open set $O \subset M$
with positive probability.
The version we take in Definition \ref{def:Irr} is stronger than necessary but suffices for our purposes here. 

\begin{definition}[Topological irreducibility] \label{def:Irr}
Consider a stochastic process $(z_t)$ defined on a complete metric space $\mathcal{Z}$. 
We say $(z_t)$ is (topologically) irreducible if for every open set $O \subset \mathcal{Z}$, initial condition $z \in \mathcal{Z}$, and $t > 0$ there holds 
\begin{align*}
\PP(z_t \in O | z_0 = z) > 0 \,. 
\end{align*}
\end{definition}

For Markov processes coming from SDEs on manifolds, a very common method of proving topological irreducibility is through the Stroock-Varadhan support theorem \cite{Stroock1972-nc} (see Theorem \ref{thm:support} below), which connects supports of transition kernels to a \emph{controllability problem} treating the Brownian paths $W_t^k$ as control parameters. 
For more information, see the discussion in Section \ref{sec:Irr}, where irreducibility for the class of Euler-like models is treated in more detail. 

We close this discussion by addressing the existence of Lyapunov exponents for the stochastic flow of diffeomorphisms $\Phi^t_\omega$. 
\begin{theorem}\label{thm:non-random-MET}
Assume \eqref{eq:general-SDE} satisfies Assumption \ref{ass:WP}. 
Then there exist positive, deterministic constants $\lambda_1$ and $\lambda_\Sigma$, independent of both the random sample $\omega$ as well as $x \in M$, such that for $\P\tensor \mu$ almost every $(\omega, x) \in \Omega \times M$ the following limits hold:
\begin{align*}
	\lambda_1 & = \lim_{t\to \infty}\frac{1}{t} \log{|D_x\Phi_\omega^t|} \, , \\
	\lambda_\Sigma & = \lim_{t\to \infty} \frac{1}{t}\log{|\det D_x\Phi_\omega^t|} \, .
\end{align*}
\end{theorem}
Theorem \ref{thm:non-random-MET} is classical and follows from the Kingman subadditive ergodic theorem \cite{kingman1973subadditive} and some basic ergodic theory for random dynamical systems \cite{kifer2012ergodic}.

\subsection{The projective process}\label{subsec:prelimProjProcess}

As remarked in Section \ref{sec:Intro}, the
 Fisher-information identity we obtain (Theorem \ref{thm:fishInfoIntro}) is framed in terms
of the \emph{projective process}, i.e., the process on tangent directions
\begin{align*}
v_t := \frac{D_x\Phi_\omega^tv}{|D_x\Phi_\omega^tv|} \, , 
\end{align*}
so that the full process $w_t := (x_t, v_t)$ lives on the unit tangent bundle $\S M \subset TM$ 
with fibers $\S_x M \subset T_x M$. 
Our aim in this Section \ref{subsec:prelimProjProcess} will be to 
lay out preliminaries framing $(w_t)$ as the solution to an SDE on $\S M$. 

We begin with some basic geometry for the manifold $\S M$. 
Using the Riemannian structure on $M$ and the Levi-Civita connection $\nabla$, we equip $\S M$ with the unique lifted metric $\tilde g$ (the Sasaki metric \cite{Sasaki1962-rt}) such that the bundle projection $\pi: \S M \to M$ is a Riemannian submersion. In particular, for each $w = (x,v)\in \S M$ we can decompose $T_w\S M$ into a horizontal subspace $H_w\S M$  of directions transverse to the fibers, identified with $T_x M$, and a vertical $V_w\S M$ subspace of directions along the fibers identified with $T_v (\S_x M)$, which itself is isomorphic to the orthogonal complement of $v$ in $T_x M$. The spaces $H_w \S M$ and $V_w \S M$ are orthogonal with respect to $\tilde{g}$, giving the orthogonal decomposition
\[
	T_w\S M = T_xM\oplus T_v(\S_xM) \, .
\]

With these preliminaries in place, 
the process $(v_t)$ satisfies the SDE
\begin{align}
	\dee v_t = V_{\nabla X_0(x_t)}(x_t,v_t)\dt  + \sum_{k=1}^r V_{\nabla X_k(x_t)}(x_t,v_t)\strat \dee W^k_,
\end{align}
where $\nabla$ denotes the covariant derivative and, for $x \in M$ and $A:T_xM \to T_xM$ linear, the `vertical' vector field $V_A$ on $\S M$, $V_A(x, v) \in T_v (\S_x M)$ for each $(x, v) \in \S_x M$, is defined by
\[
V_{A}(x,v) := Av - v\brak{v,Av}_x =: \Pi_{(x,v)}Av \, .
\]
 The full projective process $(w_t)$ evolves according to
\begin{equation}\label{eq:spherebundleeq}
	\dee w_t = \tilde{X}_0(w_t)\dt + \sum_{k=1}^r\tilde{X}_k(w_t)\strat \dee W^k_t \, ,
\end{equation}
where the $\tilde{X}_k$ are vector fields on $\S M$, which when expressed in terms of the orthogonal
horizontal/vertical splitting read as
\[
	\tilde{X}_k(x,v) := (X_k(x),V_{\nabla X_k(x)}(x,v)). 
\]

Throughout, we take on the following assumption regarding $(w_t)$. 
\begin{assumption}\label{ass:WP2}
The SDE \eqref{eq:spherebundleeq} defining the process $(w_t)$ satisfies Assumptions \ref{ass:WP} (i) and (ii). In particular, the SDE defining $(w_t)$ is globally well-posed for a.e. random sample and every initial data; and the Markov process $(w_t)$ admits a unique, absolutely continuous stationary measure $\nu$ on $\S M$.
\end{assumption}

As with other SDE, it is common to check existence and uniqueness of stationary measures  
via H\"ormander's parabolic condition and topological irreducibility, respectively. 
In particular, bracket spanning for the projective process $(w_t)$ appears routinely in the random dynamics literature (see, e.g. \cite{baxendale1989lyapunov, dolgopyat2004sample}). 

\subsubsection{Bracket spanning for the projective process}\label{subsubsec:brackSpanProjGen2}

Below we give a sufficient condition for the projective process $(w_t)$ to satisfy bracket spanning assuming
bracket spanning for the base process $(x_t)$. 

In what follows, we will find it convenient to reformulate the parabolic H\"ormander condition using slightly different notation. Below, we write $\mathfrak{X}(M)$ for the space of smooth vector fields on $M$. 
Recall that given $X \in \mathfrak X(M)$, the \emph{adjoint representation} $\mathrm{ad}(X)$
is the linear operator $\mathfrak X(M) \to \mathfrak X(M)$ sending $Y \mapsto \mathrm{ad}(X) Y := [X,Y]$. 
\begin{definition}
For a given collection $\mathcal{F} \subseteq \mathfrak{X}(M)$ define the {\em Lie algebra generated by $\mathcal{F}$} by
\begin{equation}\label{eq:lie-def}
	\mathrm{Lie}(\mathcal{F}) := \Span\{\mathrm{Lie}^m(\mathcal{F})\,:\, m\geq 1\},
\end{equation}
where 
\[
	\mathrm{Lie}^m(\mathcal{F}) := \Span\{\mathrm{ad}(X_r)\ldots\mathrm{ad}(X_{2})X_1\,:\, X_i \in \mathcal{F}\,,\, 1\leq r\leq m\}.
\]
\end{definition}

\begin{definition}
Let $Z_0, Z_1, \cdots, Z_r\in \mathfrak{X}(M)$ be smooth vector fields on $M$, and define
$\mathcal X = \{ Z_1 ,\cdots, Z_r\}$. The {\em zero-time ideal} $\mathrm{Lie}(Z_0; \mathcal X)$
is defined to be the Lie algebra generated by $\mathcal X$ and $[\mathcal X, Z_0] := \{ [Z, Z_0] : Z \in \mathcal X\}$; that is, 
\[
\mathrm{Lie}(Z_0; \mathcal X) = \mathrm{Lie}(\mathcal X, [\mathcal X, Z_0]) \, .
\]
\end{definition}
The ideal $\mathrm{Lie}(Z_0; \mathcal X)$ plays a significant role in geometric control theory, which we will revisit in Section \ref{sec:Irr} when we discuss irreducibility for Euler-like systems. We also note that it is straightforward to check, using the Jacobi identity, that $\{Z_0,Z_1,\ldots,Z_r\}$ satisfies the parabolic H\"ormander condition as in Definition \ref{def:Hormander} if 
\[
	\mathrm{Lie}_x(X_0;\mathcal{X}) := \{X(x) \,:\, X \in\mathrm{Lie}(X_0;\mathcal{X})\} = T_{x}M,
\]
for all $x \in M$. 

We now turn to a general sufficient condition (Proposition \ref{prop:class-proj-span} below) on the vector fields $\{X_k\}_{k=0}^r$ so that their lifts $\{\tilde{X}_k\}_{k=0}^r$ satisfy the parabolic H\"ormander condition on $\S M$. Given $X\in \mathfrak{X}(M)$, define 
\begin{align}
	M_X(x) := \nabla X(x) - \frac{1}{n}\Div X(x) I \, . \label{def:MX} 
\end{align}
We view $M_X(x)$ as an element of $\mathfrak{sl}(T_x M)$, the Lie algebra of 
traceless linear operators on $T_x M$. 
Observe that since the projective vector field $V_{\nabla X}(v)$ includes a
projection orthogonal to $v$, we have the identity $V_{\nabla X} \equiv V_{M_X}$.
Define
\begin{equation}\label{eq:m-lie-alg-def}
	\mathfrak{m}_x(X_0;X_1,\ldots,X_r) := \{M_X(x) \,:\, X\in \mathrm{Lie}(X_0;X_1,\ldots, X_r)\,,\, X(x) =0\}.
\end{equation} 
Lemma \ref{lem:MatProj} in the Appendix implies that $\mathfrak{m}_x(X_0;X_1,\ldots,X_r)$ is indeed a Lie sub-algebra of $\mathfrak{sl}(T_xM)$ with respect to the usual commutator $[A,B] = AB - BA$ for linear operators.

The following relates bracket spanning for $\{ \tilde X_k\}_{k = 0}^r$ to the matrix Lie algebra
$\mathfrak m_x(X_0; X_1, \cdots, X_r)$. 
\begin{proposition}\label{prop:class-proj-span}
Let $\{X_k\}_{k=0}^r$ be a collection of smooth vector fields on $M$. Their lifts $\{\tilde{X}_k\}_{k=0}^r$ satisfy the parabolic H\"ormander condition on $\S M$ if and only if $\{X_k\}_{k=0}^r$ satisfy the parabolic H\"ormander condition on $M$ and for each $(x,v)\in \S M$ we have
\begin{equation}\label{eq:proj-trans}
	\{ V_A(x,v) \,:\, A \in \mathfrak{m}_x(X_0;X_1,\ldots,X_r)\} = T_v\S_xM.
\end{equation}
\end{proposition}
Proposition \ref{prop:class-proj-span} was stated without proof in Baxendale's paper \cite{baxendale1989lyapunov}. For the sake of completeness, a self-contained proof is included in Appendix \ref{subsec:projSpanSuffCondAppendix}.

\begin{remark}\label{rem:so-remark}
In the theory of Lie algebra actions on manifolds, the condition \eqref{eq:proj-trans} means that $\mathfrak{m}_x$ acts transitively on $\S_x M$ through the Lie algebra action $A\mapsto V_A$. It is straightforward to show that the Lie algebra $\mathfrak{so}(T_xM)$ of skew-symmetric linear operators (depending on the metric) also acts transitively on $\S_xM$, and therefore a sufficient condition for transitive action of $\mathfrak{m}_x(X_0;X_1,\ldots,X_r)$ on $\S_xM$ is that
\[
	\mathfrak{so}(T_xM)\subseteq \mathfrak{m}_x(X_0;X_1,\ldots,X_r).
\]
\end{remark}

\subsubsection{Generator for the projective process}

We close this preliminary section with some brief comments on the generator for the projective process. Observe that the infinitesimal generator for $(w_t)$ on $\S M$ has the following H\"ormander form: 
\[
	\tilde{\mathcal{L}} = \tilde{X}_0 + \frac{1}{2}\sum_{k}\tilde{X}_k^2 \, . 
\]
In particular, when they exist, stationary densities $f = \frac{\dee \nu}{\dee q}$ solve the Kolmogorov equation
\begin{align}
	\tilde{\mathcal{L}}^*f  = \tilde{X}_0^*f + \frac{1}{2}\sum_{k=1}^r (\tilde{X}_k^*)^2 f = 0 \, , \label{def:Lstar}
\end{align}
where $\dq$ denotes the Riemannian volume on $\S M$ and for a given vector field $\tilde{X}$ on $\S M$, $\tilde{X}^*$ denotes the formal adjoint operator with respect to $L^2(\dq)$. Note that the differential operator $\tilde{X}^*$ can be related to $\tilde{X}$ and $\Div \tilde X$ through the relation
\begin{equation}\label{eq:X-*-div-id}
	\tilde{X}^*h = - \tilde{X}h - (\Div \tilde{X})h, \quad h\in C^\infty_c(\S M) \, . 
\end{equation}

\section{Fisher information formula for Lyapunov exponents}\label{sec:fishStatementNew}

Here we give a full, general statement of our first main result, a version of Theorem \ref{thm:fishInfoIntro} on manifolds relating the values $\lambda_1, \lambda_\Sigma$ to the partial Fisher information
\[
FI(\phi) := \frac{1}{2}\sum_{k=1}^r \int_{\S M} \frac{|\tilde{X}_k^*\phi|^2}{\phi} \,\dee q,
\]
for $\phi = f$, the density of the stationary measure $\nu$ of $(w_t)$, and for $\phi = \rho$, the 
density of the stationary measure $\mu$ for $(x_t)$ (abusing notation somewhat and viewing $\rho$ as a function on $\S M$). 

\begin{remark}
Note that $FI(\phi)$ is well defined even when the density $\phi$ is not supported everywhere. Indeed, it can be re-written as 
\[
	FI(\phi) = \frac{1}{2}\sum_{k=1}^r\int_{\S M} |X_k\log \phi + \Div X_k|^2 \phi\, \dq,
 \]
 which is implicitly integrated over $\{\phi>0\}$ so that $\log \phi$ is well-defined (see Remark \ref{rem:log-div} below for how this logarithmic gradient form relates to metric independence of the Fisher information).
\end{remark}
Throughout, the setting is as laid out in the beginning of Section \ref{sec:FI}. 

In what follows, for a given partition of unity $\{\chi_i\}_{i \geq 1}$ of $\S M$, we denote $\chi_{\leq N} = \sum_{i=1}^{N}\chi_i$.

\begin{proposition}[Fisher Information Identity] \label{prop:FI-gen-intro}
Let Assumptions \ref{ass:WP} and \ref{ass:WP2} hold. Moreover, assume that 
\begin{enumerate}
\item[(a)] there exists a smooth partition of unity $\{\chi_i\}_{i \geq 1}$ of $\S M$ and a measurable function $m:\S M\to \R_+$ such that $|\tilde{\mathcal{L}}\chi_{\leq N}(w)| \leq m(w)$ and 
\[
	\int_{\S M} f|\log f|m\,\dee q < \infty;
 \]
\item[(b)] $Q \in L^1(\mu)$ and $\tilde{Q}\in L^1(\nu)$, where $Q, \tilde{Q}$ are defined by 
\begin{align*}
	Q(x) & := \Div X_0(x) + \frac12 \sum_{k = 1}^r X_k \Div X_k(x) \, , \\
	\tilde{Q}(w) & := \Div \tilde{X}_0 (w) + \frac{1}{2}\sum_{k=1}^r\tilde{X}_k\Div \tilde{X}_k(w) \, .
\end{align*}
\end{enumerate}
Then, the following formulas hold: 
\begin{align}\label{eq:gen-FI-ID}
\begin{aligned}
	FI(\rho) &= -\int_{M} Q \, \dee \mu = - \lambda_\Sigma \, , \\
	FI(f) & = - \int_{\S M} \tilde{Q}\, \dee \nu = n\lambda_1 - 2\lambda_{\Sigma} \, .
\end{aligned}
\end{align}
Equivalently, writing $h_x(v) = f(x,v)/ \rho(x)$ for the conditional density on $\S_xM$, we have
\begin{align}
FI(f) - FI(\rho) = \frac{1}{2}\sum_{k=1}^r\int_{M}\left(\int_{\S_xM}\frac{|(X_k - V_{\nabla X_k(x)}^*) h_x(v)|^2}{h_x(v)}\dee v\right) \dee\mu(x) = n\lambda_1 - \lambda_{\Sigma}.  \label{eq:FIcond}
\end{align}
\end{proposition}
In each line of \eqref{eq:gen-FI-ID}, the second equality is a version of the famous Furstenberg-Khasminskii formula (see, e.g., \cite{khasminskii2011stochastic, arnold1995random}) for the Lyapunov exponents of an SDE satisfying Assumptions \ref{ass:WP} and \ref{ass:WP2}; for the convenience of the reader we recall this formula in Lemma \ref{lem:FK} in the Appendix. Assumptions (a) and (b) are technical and needed only when $M$ is noncompact to justify a certain integration-by-parts step; see below for details. 

After some remarks, Proposition \ref{prop:FI-gen-intro} is derived below. We give two proofs: the first (Section \ref{subsec:FI-proof}) 
is a combination of the Furstenberg-Khasminskii formula \cite{khasminskii2011stochastic, arnold1995random}  
with the Kolmogorov equation \eqref{def:Lstar} for $f$; the second proof (Section \ref{subsec:relativeEnt}) is a sketch
connecting $FI(f)$ to a certain relative entropy formula for Lyapunov exponents \cite{furstenberg1963noncommuting, baxendale1989lyapunov} (see also \cite{ledrappier1986positivity}) 
intimately connected with Furstenberg's criterion (see Section \ref{Sec:Context}).

\begin{remark}
Equation \eqref{eq:FIcond} is an equivalent formulation highlighting the relation to the natural quantity $n \lambda_1 - \lambda_\Sigma$ and similarities to Furstenberg's criterion for the Lyapunov exponents of stochastic systems (Section \ref{Sec:Context}). In particular, note that one always has $FI(f) - FI(\rho) \geq 0$ and that $n \lambda_1 > \lambda_\Sigma$ if and only if $FI(f) - FI(\rho) > 0$.
\end{remark}

\begin{remark}\label{rem:log-div}
We emphasize that the Fisher information $FI(f)$ does not depend on the choice of Riemannian 
metric $g$ on $M$, as one would expect: after all, it is straightforward to check that 
Lyapunov exponents themselves are independent of the
Riemannian metric $g$ when $M$ is compact or under a mild integrability condition when $M$ is noncompact. 

To see this, we will cast $FI(f) = \int_{\S M}(|\tilde{X}_k^*f|^2/f)\,\dee q$ as a {\em logarithmic derivative} of $\nu$ with respect to $\widetilde{X}_k^*$. 
To make this precise, define the weak derivative $\widetilde{X}_k^*\nu$, a signed measure on $\S M$, by duality (c.f. Chapter 3 of \cite{Bogachev1997-wz}) via 
\[
	\int \phi \,\dee(\widetilde{X}_k^*\nu) = \int \widetilde{X}_k\phi \,\dee \nu, \quad \phi \in C^\infty_c(\S M) \,. 
\]
When $\widetilde{X}_k^*\nu \ll \nu$, the Radon-Nikodym derivative $\beta_{\widetilde{X}_k^*}^\nu := \frac{\dee \tilde X^*_k \nu}{\dee \nu}$ exists and satisfies 
\[
	\beta_{\widetilde{X}_k^*}^\nu = \frac{\widetilde{X}_k^*f}{f} = -X_k\log f - \Div X_k,
	\] when $\dee\nu = f\dq$, as we assume in this paper. This means that we can write the Fisher information as
\[
	FI(f) = \frac{1}{2}\sum_{k=1}^r \|\beta_{\widetilde{X}_k^*}^\nu\|^2_{L^2(\nu)},
\]
which has no explicit dependence on the Riemannian metric $\tilde{g}$.
Remarkably, this form makes sense even when $\nu$ is not absolutely continuous with respect to $\dq$ and 
 only seems to requires differentiability of the measure along $\widetilde{X}_k$. This suggests the potential extension of
  Proposition \ref{prop:FI-gen-intro} to infinite dimensional settings
where no notion of volume measure $\dee q$ is available. However, this is 
beyond our scope here and left to future work. 
\end{remark}

\begin{remark}
The Fisher information is a fundamental quantity in the theory of statistical inference and information geometry (see \cite{Amari2000-ik}). It is typically used to measure the amount of information a parametrized family of laws (e.g., the law of one variable conditioned on the value of another) carries about the inference parameter. In our case, the Fisher information formula \eqref{eq:FIcond} can be interpreted as the amount of information about $x$, on average, which can be inferred by making observations on the family of conditional densities $\{h_x(\cdot)\}_{x \in M}$ in the $v$ coordinate. 
\end{remark}

\subsection{Proof of Proposition \ref{prop:FI-gen-intro}}\label{subsec:FI-proof}

We begin with a formal argument.  
Pairing the Kolmogorov equation \eqref{def:Lstar} for $f$ with $\log{f}$ and integrating gives
\[
	- \frac{1}{2}\sum_{k=1}^r \int_{\S M} (\log f) (\tilde{X}_k^*)^2 f \,\dee q = \int_{\S M} (\log f) \tilde{X}_0^* f \,\dee q. 
\]
Ignoring for the moment that $f$ is not compactly supported, integrating by parts and using \eqref{eq:X-*-div-id} gives for the left hand side
\[
\begin{aligned}
- \frac{1}{2}\sum_{k=1}^r \int_{\S M} (\log f) (\tilde{X}_k^*)^2 f\, \dee q &= -\frac{1}{2}\sum_{k=1}^r \int_{\S M} \frac{(\tilde{X}_kf)(\tilde{X}_k^* f)}{f}\, \dee q\\
&= FI(f) + \frac{1}{2}\sum_{k=1}^r\int_{\S M} (\Div \tilde{X}_k) X_k^* f\dee q\\
 &= FI(f) + \frac{1}{2}\sum_{k=1}^r\int_{\S M}(\tilde{X}_k\Div \tilde{X}_k)\,f\,\dee q \, .
\end{aligned}
\]
For the right hand side integration by parts gives
\[
	\int (\log f) \tilde{X}_0^* f \dee q = \int \tilde{X}_0f \dee q = - \int (\Div \tilde{X}_0)f\dee q,
\]
where we used the fact that $\int_{\S M} \tilde{X}_0^* f \dee q = 0$.
Putting these two identities together yields $FI(f) = - \int \tilde{Q}\dee \nu$ and therefore \eqref{eq:gen-FI-ID}. The formula for $FI(\rho)$ with $\rho = \frac{\dee \mu}{\dee x}$ the stationary density for $(x_t)$, follows from an identical argument, omitted for brevity, once one observes that $\rho$ solves the Kolmogorov equation
\begin{align*}
X_0^\ast \rho + \frac{1}{2}\sum_{j=1}^r (X_j^*)^2 \rho = 0. 
\end{align*}

Making the above argument rigorous will require us to justify the integration-by-parts steps, which along the way will make use of hypothesis (a) regarding integrability of $f\log f$ with respect to $m \dq$. Let $\{\chi_i\}_{i=1}^\infty$ be the partition of unity as in (a) and recall that $\chi_{\leq N} := \sum_{i=1}^N\chi_i$. Multiplying both sides of \eqref{def:Lstar} by $(\log f)\chi_{\leq N}$ and repeating the computation above gives
\[
	\frac{1}{2}\sum_k \int_{\S M} \frac{|\tilde{X}_k^* f|^2}{f} \chi_{\leq N}\, \dee q  = - \int_{\S M} \tilde{Q} f\chi_{\leq N}\,\dee q+  \int_{\S M} (\tilde{\mathcal{L}}\chi_{\leq N}) (f\log f - f)\,\dee q.
\]
As $N \to \infty$, the LHS converges to $FI(f)$ by the monotone convergence theorem. 
The first RHS term converges to $- \int_{\S M} \tilde Q \dee \nu$ by the dominated convergence theorem, using  that $\tilde Q \in L^1(\nu)$ (hypothesis (b)). Finally, the second RHS term above converges to 0 using 
the dominated convergence theorem and the facts that $|\tilde {\cL} \chi_{\leq N}| \leq m$
and $(f \log f - f) m \in L^1(\dee q)$ (hypothesis (a)).

It remains to check the conditional version \eqref{eq:FIcond}. Here we give only the formal proof, which can be justified under hypotheses (a) and (b) using an argument parallel to that presented above. 
For this we observe that 
\[
	\tilde{X}^*_k (h\rho) =(V_{\nabla X_k}^*h - \tilde{X}_kh)\rho + (X_k^*\rho)h
\]
and therefore since $\int_{\S_xM} h_x(v) \dee v=1$ (writing $\dee v$ for Lebesgue measure on $\S_x M$)  we find
\begin{align*}
FI(f) - FI(\rho) & = \frac{1}{2}\sum_{k=1}^r \int_{\S M} \frac{|(V_{\nabla X_k}^*h - \tilde{X}_k h)\rho + (X_k^*\rho)h|^2}{h\rho} \, \dee q - \frac{1}{2}\sum_{k=1}^r\int_{\S M}\frac{|X_k^*\rho|^2}{\rho}h\dq\\
&= \frac{1}{2}\sum_{k=1}^r\int_{M}\left(\int_{\S_xM}\frac{|(X_k - V_{\nabla X_k(x)}^*) h_x(v)|^2}{h_x(v)}\dee v\right) \dee\mu(x)\\
&\hspace{1in} + \int_{\S M} (V^\ast_{\grad X_k}h - X_kh) \left( X_k^\ast \rho \right) \, \dee q.
\end{align*}
Meanwhile, for the second term on the RHS, we have
\[
\int_{\S M} (V^\ast_{\grad X_k}h - X_kh) \left( X_k^\ast \rho \right) \, \dee q = -\int_{\S M}  X_k h \left( X_k^\ast \rho \right) \, \dee q = -\int_{\S M}  h (X_k^\ast) ^2 \rho  \, \dee q   = 0 \, , 
\]
completing the proof of Proposition \ref{prop:FI-gen-intro}.

\subsection{Relation to Baxendale's relative entropy formula}\label{subsec:relativeEnt}

This section is purely for exposition purposes to highlight the connection between the Fisher information identity in Proposition \ref{prop:FI-gen-intro} and the relative entropy formula \eqref{eq:RE}. Readers who are not interested may skip to Section \ref{sec:Hor}. 

In this section we give an argument of the Fisher information identity from the relative entropy formula \eqref{eq:RE} measuring the \emph{degree to which the degeneracy} \eqref{furstDegen} fails to hold. We already have given a complete proof above, and so our purpose here is provide some additional intuition behind the meaning of Proposition \ref{prop:FI-gen-intro}. For this reason, we will not provide results in their full generality, and for technical simplicity we only consider the case in which $M$ is compact and boundaryless (recall that Proposition \ref{prop:FI-gen-intro} holds for noncompact $M$). 
       
In preparation, recall that given two measures $\lambda, \eta$ on $\S M$, $\eta \ll \lambda$, we write $H(\eta | \lambda)$ for the relative entropy of $\eta$ given $\lambda$, defined as in \eqref{eq:defnRelEntropyIntro} above. In what follows, we abuse notation somewhat and write $H( f | g)$ for the relative entropy of $f \dq$ given $g \dq$ when $f, g$ are densities on $\S M$. Recall that $H(f | g) = 0$ iff
$f \equiv g$ almost everywhere. 

In what follows, we let $\hat \Phi^t_\omega$ be the stochastic flow of diffeomorphisms on $\S M$ induced by the SDE governing the projective process $(w_t)$. Given a smooth density $f\in L^1(\S M)$, let
\[
	f_t := (\hat{\Phi}^t_\omega)_*f = f\circ  (\hat{\Phi}^t_\omega)^{-1} |\det D(\hat{\Phi}^t_\omega)^{-1}|
\]
be the pushforward of $f$ as a density on $\S M$, and analogously, $\rho_t := (\Phi^t_\omega)_* \rho$. 
In \cite{baxendale1989lyapunov}, Baxendale derived the following continuous-time version of a 
formula derived by Furstenberg \cite{furstenberg1963noncommuting} in the context of IID compositions of random matrices. 
\begin{theorem}[Baxendale \cite{baxendale1989lyapunov}]\label{thm:baxFormulaSec2}
Under Assumptions \ref{ass:WP} \& \ref{ass:WP2}, one has the following:  
\begin{align}\label{eq:baxEntForm1}
\begin{aligned}
\E H(\rho_t | \rho) & = - t \lambda_\Sigma \, , \\
\E \left(H( f_t | f ) - H(\rho_t | \rho) \right) & = t (n \lambda_1 - \lambda_\Sigma) \, .
\end{aligned}
\end{align}
\end{theorem}
It is straightforward to show that the second line of \eqref{eq:baxEntForm1} is an equivalent formulation of \eqref{eq:RE}, with ``$\leq$'' replaced by ``='', and directly encodes the Furstenberg
criterion \eqref{furstDegen}; that is, $n \lambda_1 - \lambda_\Sigma = 0$ only if $f_t \equiv f$ for all $t$ and with full probability. 

We will provide here an independent argument showing the following: 
\begin{proposition}\label{prop:FIintTimeAnalogue}
Assume $M$ is compact and boundaryless. Then, 
\begin{align}\label{eq:FItimeIntAnalogue}
	FI(\rho) = \frac{1}{t} \E H(\rho_t | \rho) \, , \quad \text{ and } \quad 
	FI(f)  = \frac{1}{t} \E H(f_t|f)  \, .  
\end{align}
\end{proposition}
Proposition \ref{prop:FI-gen-intro} for compact $M$ immediately follows. The identity \eqref{eq:FItimeIntAnalogue} confirms that the Fisher information identity for Lyapunov exponents is the time-infinitesimal analogue of \eqref{eq:baxEntForm1}. 

\begin{proof}[Proof sketch for Proposition \ref{prop:FIintTimeAnalogue}]
We present here the proof for $f_t$; the proof for $\rho_t$ follows the same lines and is omitted. 

The density $f_t$ can be readily seen to solve the stochastic continuity equation
\[
	\dee f_t = \tilde{\mathcal{L}}^* f_t \,\dt + \sum_{k=1}^r \tilde{X}_k^* f_t \, \dee W^k_t \, , 
\]
in It\^{o} form, while $\bar{f}_t = \E f_t$ solves the forward Kolmogorov equation
\[
	\partial_t \bar{f}_t = \tilde{\mathcal{L}}^*\bar{f_t} \, .
\]
Therefore, since the initial data $\bar{f}_0 = f$ is a stationary density satisfying $\tilde {\mathcal L}^* f = 0$, we have $\E f_t = f$ for all $t\geq 0$. Using the formula for $f_t$, we can apply It\^{o}'s lemma to obtain the following stochastic equation that holds pointwise on $\S M$:
\[
	f_t\log(f_t/f) = \frac{1}{2}\sum_{k}\int_0^t\frac{|X_k^*f_s|^2}{f_s}\ds +\int_0^t(\tilde{\mathcal{L}}^* f_s) \left(\log\left(f_s/f\right) + 1\right)\ds+  M_t \, ; 
\]
here, the first RHS term is the It\^{o} correction and $M_t$ is a mean-zero martingale whose exact form is not important. Integrating over $\S M$ and averaging with respect to $\E$ gives
\begin{equation}\label{eq:relative-entropy-dissipation}
	\frac{1}{t}\E H(f_t| f) = \frac{1}{t}\E \int_0^t FI(f_s)\ds + \frac{1}{t}\E \int_0^t \int_{\S M} f_s \,\tilde{\mathcal{L}}\log\left(f_s/f\right)\dq \ds,
\end{equation}
having used duality to bring $\tilde {\mathcal{L}}$ onto the $\log$. A standard calculation (using the fact that $\tilde{\mathcal{L}}^* f =0$) implies the second term on the right-hand side is just the relative entropy dissipation $FI(f_s|f)$ (or relative Fisher information)
\[
	\int_{\S M} f_s \,\tilde{\mathcal{L}}\log\left(f_s/f\right)\dq = -FI(f_s|f) := -\frac{1}{2}\sum_k\int_{\S M} \frac{\left|X_k\left(f_s/f\right)\right|^2}{f_s/f}\,f \,\dee q.
\]
Moreover by a straightforward manipulation involving basic calculus and integration by parts, $FI(f_s|f)$ can be expressed as
\[
	FI(f_s|f) = FI(f_s) - \frac{1}{2}\int_{\S M} \frac{|X_k^*f|^2}{f^2} f_s\,\dq - \int_{\S M} X_k^2\left(\frac{f_s}{f}\right) f\,\dee q.
\]
Using Fubini and that $\E f_s = f$ (equivalent to the fact that $f$ is stationary), hence $\E X_k^2\left(f_s/ f\right) = 0$, gives
\[
	\E FI(f_s|f)= \E FI(f_s) - FI(f).
\]
Substituting this into \eqref{eq:relative-entropy-dissipation} gives the result.
\end{proof}

\section{Uniform hypoelliptic regularity via the Fisher information} \label{sec:Hor}

We now turn our attention to the proof of Theorem \ref{thm:unifHypoRegEstIntro},
which estimates the Fisher information quantity $FI(f)$ from below in terms of 
fractional Sobolev regularity. 

\subsection{Full statement of Theorem \ref{thm:unifHypoRegEstIntro}}

As discussed in Section \ref{sec:Intro}, our results are stated 
for parametrized families of SDE of the form
\begin{align}\label{eq:smallNoiseSDESec1}
\dee x_t^\eps = X_0^\eps(x_t^\eps)\,\dt + \sqrt{\eps} \sum_{k=1}^r X_k^\eps(x_t^\eps)\strat\dee W^k_t \, ,
\end{align}
where for each $\epsilon \in (0,1]$, the vector fields $\{ X_k^\epsilon\}_{k = 0}^r$ on $M$ are smooth. 
The following is a ``uniform-in-$\eps$'' version of the parabolic H\"ormander condition (Definition \ref{def:Hormander}). 

\begin{definition}[Uniform parabolic H\"ormander] \label{def:UniHormander}
Let $\cM$ be a manifold, and let $\set{Z_0^\epsilon,Z_1^\epsilon,...,Z_r^\epsilon} \subset \mathfrak{X}(\cM)$ be a set of vector fields parameterized by $\epsilon \in (0,1]$.
  With $\mathscr{X}_k$ defined as in Definition \ref{def:Hormander} we say $\set{Z_0^\epsilon,Z_1^\epsilon,...,Z_r^\epsilon}$ satisfies the uniform parabolic H\"ormander condition on $\cM$ if $\exists k \in \mathbb N$, such that for any open, bounded set $U \subseteq \cM$ there exist constants $\set{K_n}_{n=0}^\infty$, such that for all $\epsilon \in (0,1]$ and all $x \in U$, there is a finite subset $V(x) \subset \mathscr{X}_k$ such that $\forall \xi \in T_x \cM$,
\begin{align*}
\abs{\xi} \leq K_0 \sum_{Z \in V(x)} \abs{Z(x) \cdot \xi} \qquad \sum_{Z \in V(x)}\norm{Z}_{C^n} \leq K_n. 
\end{align*}
Equivalently, in the notation of Section \ref{subsubsec:brackSpanProjGen2}, we could stipulate
that the subset $V(x)$ be drawn from $\mathrm{Lie}^k(Z_0^\eps; Z_1^\eps, \cdots, Z_r^\eps)$
for some fixed $k \geq 1$. 
\end{definition}

In what follows, we will assume that the SDE defining the projective process $w_t = (x_t, v_t)$ satisfies
this uniform spanning condition: 

\begin{assumption}[Projective spanning condition]\label{ass:PHC}
  The vector fields $\{ \tilde X_0^\epsilon, \tilde X_1^\epsilon, \cdots, \tilde X_r^\epsilon\}$ satisfy the parabolic H\"ormander condition on $\S M$  uniformly in $\epsilon \in (0,1]$ on bounded sets in the sense of Definition \ref{def:UniHormander}. 
\end{assumption}

We are now positioned to state the general version of Theorem \ref{thm:unifHypoRegEstIntro}
relating Fisher information and weak Sobolev regularity. 
Below, $B_r(w_0; \S M)$ is the geodesic ball of radius $r > 0$ centered at $w_0 \in \S M$. 
\begin{theorem} \label{thm:HypoFI}
Let $(M,g)$ be as in the beginning of Section \ref{sec:FI}, and let $\{ X_0^\eps, \cdots, X_r^\eps\}$
be smooth vector fields. Assume that for all fixed $\eps \in (0,1]$, these vector fields satisfy 
Assumptions \ref{ass:WP}, \ref{ass:WP2}, and that the parametrized family $\{ X_k^\eps\}_{\eps \in (0,1]}$
satisfies Assumption \ref{ass:PHC}.
Let $f^\epsilon$ denote the unique stationary density for the
projective process $w_t = (x_t, v_t)$ on $\S M$, and assume that for all $\epsilon \in (0,1]$ we have
$FI(f^\epsilon) < \infty$. 
Then $\exists s_* \in (0,1)$ such that $\forall s\in (0,s_*)$, $\forall R >0$, $\forall w_0 \in \S M$, $\exists C > 0$ such that
$\forall \ep \in (0,1]$,
\begin{align}
\norm{\chi_R f^\ep}_{W^{s, 1}}^2  \leq C\left(1 + FI(f^\ep) \right), 
\end{align}
where $\chi_R$ is a smooth cutoff function equal to $1$ on $B_R(w_0; \S M)$ and zero on $\S M \setminus B_{R+1}(w_0; \S M)$. 
Moreover, the constant $C$ can be chosen to depend only on $s,R, w_0, \dim M$ and the constants $k$ and $\set{K_n}_{n=0}^J$ (for a $J$ depending only on $k$) in Definition \ref{def:UniHormander}.
\end{theorem}
The definition of fractional Sobolev space $W^{s, 1}$ is recalled for functions on manifolds in \eqref{def:Wsp}. 

\begin{remark}
Theorem \ref{thm:HypoFI} requires only a uniform H\"ormander condition concerning spanning of
the Lie algebra $\mathrm{Lie}(X_0^\eps, \cdots, X_r^\eps)$ rather than the slightly stronger parabolic H\"ormander condition concerning the zero-time ideal $\mathrm{Lie}(X_0^\eps; X_1^\eps, \cdots, X_r^\eps)$.
However, the parabolic H\"ormander condition is required anyway for other purposes such as  
verifying Assumption \ref{ass:WP2}, c.f. the discussion in Section \ref{sec:FI}.
\end{remark} 

\begin{remark}
Above, the value $s_*$ is determined exclusively by the number of `generations' of brackets needed to satisfy Assumption \ref{ass:PHC} (though note this will generally depend on the dimension of the manifold $M$ itself). One can check that the proof gives $\frac{1}{2k} < s_* \leq \frac{1}{k}$ where $k$ is as in Definition \ref{def:UniHormander}, being the maximum depth of brackets needed to uniformly span.
\end{remark}

\begin{remark}
Lastly, the hypoelliptic lower bound Theorem \ref{thm:unifHypoRegEstIntro} can potentially be improved.
In the proof of the theorem below, the duality semi-norm  $\mathfrak{D}(f)$ is of size $O(\eps \sqrt{FI})$.
In analogy with other kinds of interpolation inequalities, it is reasonable to seek a version of the H\"ormander inequality \eqref{ineq:wLs} that gives extra weight to $\mathfrak{D}$ relative to $FI$, i.e. something along the lines of $\norm{\chi_U f}_{W^{s,1}}^2 \lesssim 1 + \eps^{-\alpha} \mathfrak{D}(f)^2 + \eps^{\beta}FI(f) \lesssim 1 + \frac{n\lambda_1 - 2\lambda_\Sigma}{\eps^{1-\gamma}}$, thus improving the scaling available to the compactness-rigidity scheme.
Perhaps in this scheme one can use a weaker regularity on the left-hand-side but nevertheless is good enough to prove equi-integrability (all that is really required for the compactness-rigidity approach).
\end{remark}

The remainder of this section is devoted to proving Theorem \ref{thm:HypoFI}.  

\subsubsection{Notation and setup}

Theorem \ref{thm:HypoFI} is not directly related to the projective process, and so for the sake
of conceptual clarity we present a proof here establishing a more general version. Throughout
the remainder of Section \ref{sec:Hor}, we will assume $(\cM, \tilde g)$ is a smooth, connected, 
geodesically complete, orientable manifold without boundary. We write $d = \dim \cM$ and 
$\dq$ for the volume element on $\cM$. 
The geodesic ball of radius $r > 0$ at a point $x \in \cM$ is written $B_r(x; \cM)$, while 
we write $B_r(0) \subset \R^d$ for the Euclidean ball of radius $r$. 

Abusing notation somewhat, we will assume $\{ X_k^\epsilon\}_{k = 0}^r, \epsilon \in (0,1],$
is a family of smooth vector fields on $\cM$ satisfying the uniform parabolic H\"ormander
condition (Definition \ref{def:UniHormander}), while $f^\epsilon$ is any\footnote{Uniqueness of $f^\epsilon$ is not used in the proof of Theorem \ref{thm:HypoFI} and so we need not assume any 
topological irreducibility property of the corresponding SDE. } smooth solution of the 
Kolmogorov equation 
\begin{align}
(X_0^\epsilon)^\ast f^\epsilon + \frac{\epsilon}{2}\sum_{j=1}^r ((X_j^\epsilon)^*)^2 f^\epsilon = 0 \, .\label{eq:gen}
\end{align} 

Recall that given $X, Y \in \mathfrak X(\cM)$, the adjoint action of $X$ on $Y$ is defined through the Lie bracket $\ad (X) Y = [X,Y]$. For a multi-index $I = (i_1,...,i_k), i_j \in \{ 0,\cdots, r\}$ for each $1 \leq j \leq k$, we denote
\begin{align}
X_I = \ad (X_{i_1}) \ldots \ad (X_{i_{k-1}}) X_{i_k}. 
\end{align}
In what follows, set $s_0 = \frac{1}{2}$, $s_j = 1$, $1 \leq j \leq r$ and for a multi-index $I = (i_1,...,i_k)$ we write, 
\begin{align}
m(I) := \frac{1}{s(I)} := \sum_{j=1}^k \frac{1}{s_{i_j}}. 
\end{align}
Note that $m(I)$ provides a measure of how ``deep'' a bracket is (i.e. the larger $m(I)$ the more brackets that were taken), weighted in a way that will be consistent with available regularity. 

We denote by $\mathfrak{X}^s(\cM) \subset \mathfrak{X}(\cM)$ the $C^\infty(\cM)$-submodule of vector fields generated from successive brackets with $s  \leq s(I)$, that is,  
\begin{align*}
\mathfrak{X}^s(\cM) = \Bigg\{Z \in \mathfrak{X}(\cM) : Z = \sum_j h_j X_{I_j}, \quad s(I_j) \geq s, \, h_j \in C^\infty(\mathcal{M}) \Bigg\}. 
\end{align*}
Recall that $\set{X_j}_{j=0}^r = \{X_j^\epsilon\}_{j = 0}^r \subset \mathfrak{X}(\cM)$ depend in a general manner on a parameter $\epsilon \in (0,1)$, hence $\mathfrak X^s(\cM)$ also depends on $\epsilon$.
However, assuming that $\{ X_j^\epsilon\}$ satisfies the uniform parabolic H\"ormander condition, a simple consequence is that  $\exists s > 0$ such that $\forall \epsilon \in (0,1)$, $\mathfrak{X}^s(\cM) = \mathfrak{X}(\cM)$ \footnote{Indeed, if $\set{Z_1,...,Z_m}$ span $T_w \cM$ then $\exists \delta > 0$ such that $\forall v \in B_\delta(w) = \set{v \in \cM :d(w,v) < \delta}$ the same vector fields span, and so for $V \in \mathfrak{X}(\cM)$, $\exists c_j \in C^\infty$ such that $V = \sum c_j Z_j$ on $B_\delta$. The result then follows by a suitable partition of unity.}. 
Once and for all, fix $s_* > 0$ so that $\mathfrak X^{s_*}(\cM) = \mathfrak X(\cM)$.  

Theorem \ref{thm:HypoFI} estimates the Fisher information of $f^\epsilon$ from below by 
fractional Sobolev regularity $W^{s_*,1}$ (defined on general manifolds in \eqref{def:Wsp}). 
However, for our purposes, it is easier to work with an $L^1$ H\"older-type regularity class $\Lambda^s$ embedded in $W^{s', 1}$ for all $0 < s' < s < 1$ (Lemma \ref{lem:WsLs}). 
Due to subtleties involved in intrinsic definitions of Besov spaces on manifolds (see e.g. \cite{Triebel}), 
we will construct the $\Lambda^s$ seminorm on bounded sets in $\cM$ using a coordinate atlas, as follows. 

Fix $R > 0$, set $U := B_{R}(z_0; \cM)$\footnote{More general sets $U$ can be used, but this is sufficient for all our purposes. In fact, at this step, any bounded, open set can be used.} at some arbitrary $z_0 \in \cM$, and set $U' = B_{R+1}(z_0;\cM)$. Note that by the Hopf-Rinow theorem, $\overline{U}$ and $\overline{U'}$ is compact, hence
there exists $\delta > 0$ and a finite family $\set{\mathbf{x}_j}$ smooth injective mappings $\mathbf{x}_j:B_{4\delta}(0) \to U$ such that both $\tilde U_j := \mathbf{x}_j(B_{4 \delta}(0))$ and $U_j := \mathbf{x}_j(B_{\delta}(0))$ are open covers of $U$.
Let $\set{\chi_j}$ be a smooth partition of unity on $B_R(z_0;\cM)$ subordinate to the
cover $\{\tilde U_j\}$, i.e., (i) $0 \leq \chi_j \leq 1$ everywhere, (ii) 
$\chi_j|_{U_j} \equiv 1$, and (iii) $\chi_j$ is supported in $\tilde U_j$.
With this, we define the following semi-norm for $w \in C^\infty_c(U)$: 
\begin{align}
  \norm{w}_{\Lambda^s}  = \norm{w}_{\Lambda^s(U,\set{\mathbf{x}_j})}  & = \norm{w}_{L^1} + \sup_{h \in \R^d : \abs{h}< \delta}\sum_{j} \int_{\R^d} \frac{\abs{\tilde{w}_j(x+h) - \tilde{w}_j(x)}}{\abs{h}^{s}} J_j(x) \dx, \label{def:Ls}
\end{align}
where $\tilde{w}_j = (\chi_j w) \circ \mathbf{x}_j$ and $J_j : B_\delta(0) \to \R_{> 0}$ is the coordinate representation of the volume element
in the chart $(U_j, \mathbf{x}_j)$. Note that $J_j$
The choice of $z_0, R$ and the coordinate charts $\set{\mathbf{x}_j}$ is henceforth fixed and thus the dependence of $\Lambda^s$ on these parameters\footnote{While the value of $\norm{w}_{\Lambda^s}$ depends on the choice of coordinate chart, for any alternative choice of charts $\set{\mathbf{x}'_j}$, note that the following holds: $\forall w \in C^\infty_c(U)$, $\norm{w}_{\Lambda^s(U,\set{\mathbf{x}_j})} \approx \norm{w}_{\Lambda^s(U,\set{\mathbf{x}'_j})}$. } will be suppressed in what follows. 

The embedding
\begin{align}
\| w\|_{W^{s', 1}}  \lesssim_U \| w \|_{\Lambda^s}
\end{align}
for all $0 < s' < s < 1$ and $w \in C^\infty_c(U)$ is intuitively clear. However, for the sake
of completeness we include a proof in the Appendix (Lemma \ref{lem:WsLs}; see also [pg 301; \cite{Triebel}]). 

\subsubsection{Outline of the proof of Theorem \ref{thm:HypoFI}}

Crucial to both H\"ormander's original approach and our own is the ability to measure
\emph{partial} regularity of a function along some given set of directions.
Given $Y \in \mathfrak X(\cM)$, let $Y^*$ denote its formal adjoint\footnote{Note that $Y^\ast f = -Yf  - (\mathrm{div} Y)f$ and in particular $Y^\ast$ is not a vector field, but a more general first order differential operator.} and let $e^{t Y^*}$ denote the linear propagator solving the partial differential equation
$\partial_t - Y^* = 0$ (this is well-defined as long as $t > 0$ is taken sufficiently small 
depending on $Y$ and $U$). For $s > 0$, we define the family of `partial' $L^1$ H\"older seminorms\footnote{Note that these seminorms are slightly different from those used in \cite{H67}, where the linear propagator $e^{t Y}$ solving $\partial_t - Y = 0$ is used directly. Note, though, that the regularity defined is essentially the same in the sense that 
\begin{align*}
\norm{w}_{L^1} + \sup_{\abs{t} \leq \delta_0} \abs{t}^{-s} \|e^{tY^*} w - w\|_{L^1} \approx_{\delta_0,U,H} \norm{w}_{L^1} + \sup_{\abs{t} \leq \delta_0} \abs{t}^{-s} \|e^{tY} w - w\|_{L^1}.  
\end{align*}
}
\begin{align*}
  \abs{w}_{Y,s} & = \sup_{\abs{t} \leq \delta_0} \abs{t}^{-s} \|e^{tY^*} w - w\|_{L^1} \, .
\end{align*}
Note the dependence on the parameter $\delta_0 > 0$: in practice, given $U$, this parameter
is fixed and depends only on the regularity of $\set{X_I}$ as $I$ ranges over the multi-indices 
with $s(I) \geq s_*$.
We may henceforth fix $\delta_0$, as the vector fields in the proof vary in a uniformly bounded set in $C^J$ for a $J$ depending only the constants in Definition \ref{def:UniHormander}.  

Turning back to the proof of Theorem \ref{thm:HypoFI}: ultimately, for $f = f^\epsilon$
solving the Kolmogorov equation \eqref{eq:gen}, we seek to control 
$\| f^\epsilon\|_{\Lambda^{s_*}(U)}$ from above in terms of $FI(f^\epsilon)$. 
Starting from the latter, it is straightforward (Lemma \ref{lem:FIvsXj1}) to obtain the general functional inequality
\begin{equation} \label{ineq:Xj1FISk}
\sum_{j=1}^r |w|_{X_j, 1} \lesssim \norm{w}_{L^1}^{1/2} \sqrt{FI(w)} \, , 
\end{equation}
for any $w \in C^\infty_c(U)$.
Hence, for all intents and purposes it suffices to control the regularity $\norm{f}_{\Lambda^{s_*}}$ from above in terms of
$\sum_{j \geq 1} |f|_{X_j, 1}$. 
For this, we turn to the ideas laid out by H\"ormander.
First, the spanning condition $\mathfrak X^{s_*} = \mathfrak X$ allows to ``fill in'' the missing directions not spanned by 
the original $\set{X_0, \cdots, X_r}$, leading to the general functional inequality: 
\begin{equation}
\norm{w}_{\Lambda^{s_*}} \lesssim_U \norm{w}_{L^1} + \sum_{j=0}^r \abs{w}_{X_j,s_j}, \label{ineq:LsHypSk}
\end{equation}
for $w \in C^\infty_c(U)$.
This is a relatively straightforward adaptation of [Section 4; \cite{H67}]-- see Lemma \ref{lem:HorAllCtrl} below.

While \eqref{ineq:Xj1FISk} controls $\abs{w}_{X_j,1}, 1 \leq j \leq r$ in terms of the Fisher information $FI(w)$, 
it remains (as in \cite{H67}) to obtain an upper estimate on $\abs{f^\epsilon}_{X_0,1/2}$. 
The starting point is the derivation of an a priori estimate on $f^\epsilon$ from \eqref{eq:gen}.
In \cite{H67}, H\"ormander observed that one naturally obtains an a priori regularity estimate on $X_0 f$ in a \emph{negative} regularity $L^2$ space in terms of $X_j f \in L^2$ (see also discussions in \cite{AM19}). 
In our case, we cannot work in $L^2$, and instead have to work in a negative-type regularity which is essentially the dual to that in \eqref{ineq:Xj1FISk}-- this is the only a priori estimate available that will be useful. 
Pairing \eqref{eq:gen} with a test function $v \in C^\infty_c$ we obtain the following a priori estimate, which is essentially the $W^{-1,\infty}$  norm with respect to the $X_j^\ast$ directions: 
\begin{align}
\begin{aligned}
\mathfrak{D}_\ep(f^\ep) & := \sup\left\{\abs{\int f^\epsilon X_0^\epsilon v \,\dee q} \,:\, v \in C^\infty_c,\, \norm{v}_{L^\infty} + \sum_{j=1}^r \norm{X_j v}_{L^\infty} \leq 1\right\} \\
&  \leq \epsilon \sum_{j=1}^r \|X_j^* f^\epsilon\|_{L^1} \lesssim \epsilon \sqrt{FI}. 
\end{aligned}
\label{ineq:mDaP}   
\end{align}

Using this, the missing $X_0$ regularity is recovered by the following, which is the main difficulty in the proof: for any $0 < \sigma < s_*$, $U$ bounded geodesic ball and $w \in C^\infty_c(U)$, we show that
\begin{align}
|w|_{X_0,1/2} \lesssim_U \sum_{j=1}^r \abs{w}_{X_j,1} + \mathfrak{D}_\ep(w) + \norm{w}_{\Lambda^\sigma}. \label{ineq:X012Sk}
\end{align}
That is, we recover the $\abs{w}_{X_0, 1/2}$ regularity by a combination of the negative
$\mathfrak{D}_\ep$ regularity in conjunction with the positive $\abs{w}_{X_j, 1}$ regularity, 
accruing only a remainder term $\norm{w}_{\Lambda^\sigma}$. 
Combining with \eqref{ineq:LsHypSk} (along with interpolation of $\Lambda^\sigma$ between $\Lambda^{s_*}$ and $L^1$), we obtain the following: 
$\forall U \subset \cM$ open, bounded geodesic balls $\exists C > 0$ such that for all $w \in C^\infty_c(U)$, there holds
\begin{align}
\norm{w}_{\Lambda^{s_*}(U)} \leq C\bigg(\norm{w}_{L^1} + \mathfrak{D}_\ep(w) + \sum_{j=1}^r \abs{w}_{X_j,1}\bigg). \label{ineq:wLs}
\end{align}
From here, our estimate on $\norm{f^\ep}_{\Lambda^{s_*}(U)}$  
in Theorem \ref{thm:HypoFI} is an easy consequence of the functional inequality \eqref{ineq:Xj1FISk} and the apriori estimate in \eqref{ineq:mDaP}. 

In Section \ref{sec:prelim} we review the available apriori estimates and basic functional inequalities that are used in the proof. 
In Section \ref{sec:Sec4Adapt} we briefly recall \eqref{ineq:LsHypSk} and a closely related inequality which are straightforward adaptations of estimates in [Section 4; \cite{H67}]. In Section \ref{sec:X0Xj} we give the proof of \eqref{ineq:X012Sk}, leaving the main lemma to be proved in Section \ref{sec:RegLem}.
As in the corresponding step in \cite{H67}, \eqref{ineq:X012Sk} is based on a careful regularization procedure, though it is more subtle to perform this procedure in the $W^{-1,\infty}$-type framework we work with here. Section \ref{sec:RegLem} is dedicated to the details of this regularization. 

\subsection{Preliminary estimates} \label{sec:prelim}
To start, we record some useful estimates for the $L^1$ H\"older-type seminorms $|\cdot|_{Y, s}$. 
Let $Y \in \mathfrak X(M)$ and let $e^{tY}$ be the linear propagator of the partial differential operator $\partial_t - Y$.
By the method of characteristics, the smooth family of diffeomorphisms $h_Y(t):\cM \to \cM$ 
solving the initial value problem $\dot x = Y(x)$ satisfies the identity
\begin{align}
e^{tY} w = w \circ h_Y(t) \, .
\end{align}
With $Y^*$ the formal adjoint of $Y$, again by the method of characteristics there is a smooth family of strictly positive densities $H_Y(t):\cM \to (0,\infty)$ such that
\begin{align}
e^{tY^*} w = H_Y(t) \,e^{-tY} w \, . \label{eq:etYstar}
\end{align}
In particular, for $\abs{t} \lesssim 1$, 
\begin{align}
\abs{H_Y(t) - 1} \lesssim \abs{t}\, ,  \label{ineq:HYm1}
\end{align}
with similar estimates on higher derivatives. The choice of the parameter $\delta_0$ above will be such that $h_Y(t,x)$ is well-defined (i.e. no trajectories reach infinity in finite time) for all $t \in (-\delta_0,\delta_0)$ and $x\in U$
for every vector field that appears in the proof which follows (more accurately, we will re-scale the fields so that we may take $\delta_0 = 1$).
Accordingly, we have for $t \in (-\delta_0,\delta_0)$ and $w \in C^\infty_c(U)$ we have, 
\begin{align}
\lnorm{e^{tY}w}_{L^p} + \lnorm{e^{tY^*}w}_{L^p}  \lesssim \norm{w}_{L^p}. \label{ineq:etYLpbd}  
\end{align}

Next, we prove \eqref{ineq:Xj1FISk}: that $\|X_j^* w\|_{L^1}$ controls one derivative in the $L^1$-H\"older norms.  
\begin{lemma} \label{lem:FIvsXj1}
Let $U \subset \cM$ be a bounded geodesic ball. Then, $\forall w \in C^\infty_c(U)$ there holds, 
\begin{align*}
\|w\|_{X_j,1} \leqc_U \|X_j^*w\|_{L^1} \lesssim \norm{w}_{L^1}^{1/2}\sqrt{FI(w)}. 
\end{align*}
\end{lemma}
\begin{proof}
Let $v \in L^\infty$, then
\[
	\left|\int_{\cM} v(e^{t X_j^*}w - w) \dq\right| = \left|\int_0^t\int_{\cM} v e^{s X_j^*} X_j^* w\,\dq  \ds\right| \leq |t| \|v\|_{L^\infty}\|X_j^* w\|_{L^1}. 
\]
Taking the supremum over $\|v\|_{L^\infty}\leq 1$ and dividing by $|t|$ gives the first inequality whereas the second follows by Cauchy-Schwarz. 
\end{proof}

Lastly, we record the simple observation that the negative regularity $\mathfrak{D}_\ep$ can be localized. 
\begin{lemma} \label{lem:Dloc} 
Let $U\subseteq \cM$ be a bounded geodesic ball and $\chi \in C^\infty_c(U)$. Then, for any $h \in L^1(\cM)$, 
we have
\begin{align}
\mathfrak{D}_\ep(\chi h) \lesssim_U \norm{h}_{L^1} + \mathfrak{D}_\ep(h). 
\end{align}
\end{lemma} 
\begin{proof} 
Set $w = \chi h$. For test functions $v \in C^\infty_c(U)$, we estimate
\begin{align}
\abs{\int (X_0v) w\dq} & \leq \abs{\int v(X_0\chi)h \,\dee q} + \abs{\int X_0(\chi v) h\,\dee q} \lesssim \norm{h}_{L^1} + \mathfrak{D}_\ep(h). 
\end{align}
Note that the estimate is uniform in $\| X_0\|_{C^1(U)}$. 
\end{proof}

\subsection{Controlling $\Lambda^s$ with $\abs{\cdot}_{X_j,s_j}$} \label{sec:Sec4Adapt}
The first steps to Theorem \ref{thm:HypoFI} are several lemmas that are nearly the same as those in [Section 4; \cite{H67}], except (A) we need them in $L^1$, (B) we need them uniform in the parameter $\epsilon$ hidden in $X_0$, (C) we need to generalize the proof to compact subsets of manifolds $(\cM,g)$.
However, these changes are seen to be relatively small a careful reading of \cite{H67} and so we only sketch the changes below after the statements; see \cite{BL20} for more discussion on the uniformity.

\begin{lemma} \label{lem:Hor4XIctrl}
Let $U$ be a bounded geodesic ball and $0 < \sigma < s_\ast$ with $\mathfrak{X}^{s_\ast}(\cM) = \mathfrak{X}(\cM)$.
For all $\delta > 0$, $\exists C_\delta > 0$ such that for all multi-indices $I$ such that at least one index is zero, the following holds $\forall w \in C^\infty_c(U)$, 
\begin{align*}
\abs{w}_{X_I,s(I)} \leq \delta \abs{w}_{X_0,\frac{1}{2}} + C_\delta \bigg(\sum_{j=1}^r \abs{w}_{X_j,1} + \norm{w}_{\Lambda^\sigma}\bigg).
\end{align*}
Moreover, $C_\delta$ depends on $\set{X_0,X_1,...,X_r}$ only in the manner stated in Theorem \ref{thm:HypoFI}. 
\end{lemma}

The next lemma shows that one can control regularity in $\Lambda^s$ by controlling the original vector fields. 

\begin{lemma} \label{lem:HorAllCtrl}
  Let $U$ be an open, bounded geodesic ball and $s_\ast$ be such that $\mathfrak{X}^{s_\ast}(\cM) = \mathfrak{X}(\cM)$.
  Then, for $w \in C^\infty_c(U)$ there holds  
\begin{align*}
\norm{w}_{\Lambda^{s_\ast}} \lesssim_U \norm{w}_{L^1} + \sum_{j=0}^r \abs{w}_{X_j,s_j}. 
\end{align*}
Moreover, the implicit constant depends on $\set{X_0 ,X_1,...,X_r}$ only in the manner stated in Theorem \ref{thm:HypoFI}. 
\end{lemma}

\begin{proof}[\textbf{Summary of Proofs of Lemma \ref{lem:Hor4XIctrl} and Lemma \ref{lem:HorAllCtrl}}]

First we observe that the analogues of [Lemma 4.1; \cite{H67}] and [Lemma 4.2; \cite{H67}] can both be extended to $\Lambda^s(U)$ for general locally bounded geometries (see the proof of Lemma \ref{lem:RegssBds} below for more details on how this would be done). 
With this update, the proof of Lemma \ref{lem:Hor4XIctrl} follows as in [Lemma 4.6; \cite{H67}] (note the remark that follows the proof) whereas the proof of Lemma \ref{lem:HorAllCtrl} follows as in [Theorem 4.3; \cite{H67}]; indeed the only steps specific to $\R^d$ are  [Lemma 4.1; \cite{H67}]  and [Lemma 4.2; \cite{H67}] and all steps apply equally well for any $L^p$, $p \in [1,\infty)$.
The uniformity is also seen noting how H\"ormander's condition is applied and by noting that the errors depend only on a fixed, finite number of local $C^k$ norms of the vector fields (see \cite{BL20} for more discussion).  
\end{proof}

\subsection{Positive $X_0$ regularity from negative $X_0$ and positive $X_j$ regularity} \label{sec:X0Xj}
In this subsection, we prove the apriori estimate \eqref{ineq:X012Sk} and then use it to complete the proof of Theorem \ref{thm:HypoFI}. 
Fix $0 < \sigma < s_\ast$ arbitrary. 
Having fixed $U$ we may, by rescaling $\set{X_0,X_1,...,X_r}$, assume that $e^{tX_I}$ (and hence $e^{tX_I^*}$) is well-posed for $w \in C^\infty_c(U)$ for $t \in [-1,1]$ for $\sigma \leq s(I)$ (and hence we may choose $\delta_0 = 1$).

Analogous to [Section 5; \cite{H67}], the primary intermediate step is to first deduce the estimate assuming the natural control on essentially all other vector fields in $\mathfrak{X}^\sigma$.
\begin{definition}
Denote by $\mathcal{J}$ the set of all multi-indices $I$ with $\sigma \leq s(I)$ \emph{except} for the singleton $\set{0}$. 
\end{definition}
Note this definition is slightly different from that in \cite{H67}. Define the following semi-norm
\begin{align}
|w|_{M} := \sum_{I \in \mathcal{J}} \abs{w}_{X_I,s(I)}. 
\end{align}
The main step in the proof of \eqref{ineq:X012Sk} (and hence Theorem \ref{thm:HypoFI} as a whole) is to prove the following. 
\begin{lemma} \label{lem:Mest}
For any bounded, geodesic ball $U \subset \mathcal{M}$, and every $w\in C^\infty_c(U)$, the following holds uniformly in $\epsilon$,
\begin{align}
|w|_{X_0,\frac{1}{2}} \lesssim_U |w|_M +\|w\|_{\Lambda^\sigma} + \mathfrak{D}_\epsilon(w). 
\end{align}
\end{lemma}
As in the corresponding [Section 5; \cite{H67}] (and in \cite{AM19}), we use an approach based on a carefully selected regularization, but our choice is even a little more delicate than \cite{H67}.
As the regularization procedure is quite technically subtle, we first give the proof of Lemma \ref{lem:Mest} assuming the existence of a regularizer satisfying the desired properties.
\begin{lemma} \label{lem:Sreg}
There exists a family of uniformly bounded smoothing operators $S_\tau : L^p \to L^p$ for $\tau \in (0,1)$ and $p \in [1,\infty]$ with the following properties: for all $w \in C^\infty_c(U)$, 
\begin{align}
\norm{S_\tau^* w - w}_{L^1} & \lesssim \tau \abs{w}_{M}, \\
\sum_{j=1}^r \norm{X_j S_\tau w}_{L^\infty} & \lesssim \frac{1}{\tau} \norm{w}_{L^\infty}, \label{ineq:tSLinfReg} \\
\norm{[X_0, S_\tau]^*w }_{L^1} & \lesssim \frac{1}{\tau}( \abs{w}_{M} + \|w\|_{\Lambda^\sigma}). \label{ineq:tScomm} 
\end{align}
\end{lemma}

Assuming this lemma for now, we proceed.
\begin{proof}[Proof of Lemma \ref{lem:Mest} assuming Lemma \ref{lem:Sreg}]
We will first obtain regularity estimates by evaluating the fractional time derivative of $e^{t X_0^\ast} w$ (omitting the $\epsilon$ for notational simplicity).  
Observe that for any $t,\tau> 0$, 
\begin{align}
\norm{e^{t X_0^\ast} w - w}_{L^1} \leq \norm{e^{t X_0^\ast} \left(S_\tau^* w - w\right)}_{L^1} + \norm{S_\tau^* w - w}_{L^1} + \norm{e^{tX_0^*} S_\tau^*w- S_\tau^* w}_{L^1}. \label{ineq:etX0star} 
\end{align}
Therefore, by Lemma \ref{lem:Sreg} and $L^1$ boundedness of the group $e^{-tX_0^\ast}$ on $U$, 
\begin{align}
\sup_{\abs{t} \leq 1}\norm{e^{tX_0^\ast} \left(S_\tau^* w- w\right)}_{L^1}  \lesssim \norm{S_\tau^* w - w}_{L^1} & \lesssim \tau\abs{w}_{M}. \label{ineq:L1conv1} 
\end{align}
This will suffice for the first two terms in \eqref{ineq:etX0star}. Next, we estimate the last term in \eqref{ineq:etX0star}. We will do this using the inequality,
\begin{equation}\label{eq:L1-dual-rep}
	\norm{e^{tX_0^*} S_\tau^*w- S_\tau^* w}_{L^1} \leq  \sup_{\|v\|_{L^\infty} \leq 1} \left| \int_0^t\int_{\cM} (e^{s X_0^*}v)\, X_0^*S_\tau^*w\, \dq\ds \right|, 
\end{equation}
which can be easily deduced from $\|w\|_{L^1} = \sup_{\|v\|_{\infty} \leq 1}\int v\,w\,\dq$ and the property $e^{t X_0^*} = (e^{t X_0})^*$.
For a fixed $v\in L^\infty$, we find that
\[
\begin{aligned}
\left|\int_{\cM} (e^{s X_0}v) X_0^* S_\tau^*w\,\dq\right|  & \leq \left|\int_{\cM} (e^{s X_0}v)[X_0,S_\tau]^*w\,\dq\right| + \left|\int_{\cM} (S_\tau e^{s X_0}v)  X_0^* w\,\dq\right|\\
 &\leq \|e^{s X_0}v\|_{L^\infty}\|[X_0,S_\tau]^*w\|_{L^1} + \bigg(\norm{S_\tau e^{s X_0} v}_\infty+ \sum_{j=1}^r\|X_j S_\tau e^{s X_0}v\|_{L^\infty}\bigg) \mathfrak{D}(w).
\end{aligned}
\]
Using Lemma \ref{lem:Sreg} and the boundedness of $e^{tX_0}$ in $L^\infty(U)$, we conclude that
\[
	\left|\int_{\cM} (e^{s X_0}v) X_0^* S_\tau^* w\,\dq\right| \lesssim_U \frac{1}{\tau}\|v\|_{L^\infty}\left(|w|_M + \|w\|_{\Lambda^\sigma}+ \mathfrak{D}(w)\right)
\]
and from \eqref{eq:L1-dual-rep} we deduce 
\[
	\norm{e^{t X_0^*} S_\tau^*w- S_\tau^* w}_{L^1} \leqc_U  \frac{|t|}{\tau} \left(|w|_M + \|w\|_{\Lambda^\sigma}+ \mathfrak{D}(w)\right). 
\]
Therefore, setting $\tau = \sqrt{|t|}$ and using \eqref{ineq:L1conv1} implies
\begin{align}
\norm{e^{tX_0^\ast} w - w}_{L^1} \lesssim \sqrt{|t|} \left(|w|_M + \mathfrak{D}(w)\right). \label{ineq:etX0astbd}
\end{align}
By \eqref{eq:etYstar}, \eqref{ineq:HYm1}, and the boundedness of $U$, this implies the desired result.\end{proof} 

To complete the section, we explain in more detail how Lemma \ref{lem:Mest} implies Theorem \ref{thm:HypoFI}.
\begin{proof}[Proof of Theorem \ref{thm:HypoFI}]
By Lemma \ref{lem:Mest}, followed by Lemma \ref{lem:Hor4XIctrl} to absorb the effect of the higher order brackets by choosing $\delta$ sufficiently small, implies \eqref{ineq:X012Sk}, that is for any $w \in C^\infty_c(U)$, 
\begin{align}
\norm{w}_{X_0,\frac{1}{2}} \lesssim \norm{w}_{\Lambda^\sigma} + \sum_{j=1}^r \abs{w}_{X_j,1} + \mathfrak{D}_\ep(w). 
\end{align}
Applying Lemma \ref{lem:HorAllCtrl} then implies 
\begin{align}
\norm{w}_{\Lambda^{s_*}} \lesssim \sum_{j=1}^r \abs{w}_{X_j,1} + \norm{w}_{\Lambda^\sigma} + \mathfrak{D}_\ep(w). \label{ineq:Lsinter}
\end{align}
Next, note the interpolation (from H\"older's inequality and Definition \ref{def:Ls}): $\forall \sigma \in (0,s^*)$ and all $\delta > 0$, $\exists C_\delta$ such that
\begin{align}
\norm{w}_{\Lambda^\sigma} \leq \delta \norm{w}_{\Lambda^{s_*}} + C_\delta\norm{w}_{L^1},  \label{ineq:LssigIn}
\end{align}
which by \eqref{ineq:Lsinter} implies H\"ormander inequality \eqref{ineq:wLs}. 
Let $U \subset\subset U' \subset \cM$ where $U'$ is another bounded (open) geodesic ball and let $\chi \in C^\infty_c(U')$ with $\chi(x) =1$ for all $x \in U$. 
Then, Lemma \ref{lem:Dloc} implies  
\begin{align}
\norm{\chi f^\epsilon}_{\Lambda^{s^*}} \lesssim 1 + \sum_{j=1}^r \abs{\chi f^\epsilon}_{X_j,1} + \mathfrak{D}_\ep(f^\ep). \label{ineq:Xj1D}
\end{align}
Putting Lemma \ref{lem:FIvsXj1} together with \eqref{ineq:Xj1D} and \eqref{ineq:mDaP}, completes the proof of Theorem \ref{thm:HypoFI}.
\end{proof} 

\subsection{Regularization: Lemma \ref{lem:Sreg}} \label{sec:RegLem}
In this subsection we prove Lemma \ref{lem:Sreg}. 
First, we define a suitable ``isotropic'' mollifier via the parameterization. Let $\phi \in C^\infty_0((-1,1))$ with $\phi \geq 0$, $\int_{-1}^1 \phi(t) \dt = 1$, and $\phi(-t) = \phi(t)$, denoting $\tilde{w}_j = \chi_j w \circ \mathbf{x}_j$, and for each $x\in \R^d$ let $\phi_\tau(x) = \frac{1}{\tau^d}\phi(|x|/\tau)$. We define the regularization of $\chi_j w$ as follows for $0< \tau \leq \delta$, 
\[
\Phi_\tau^{j}  w \circ \mathbf{x}_j := \frac{1}{J(x)}\int_{\R^d}\phi_\tau(|x-y|) \tilde{w}_j(y) J_j(y) \dy, \label{def:Phij}
\]
where as above we write $J_j = \left(\textup{det} \, \tilde{g}\right)^{1/2}$, the volume element on $\cM$ in local coordinates. 
We write
\[
\Phi_\tau w(q) = \sum_{j} \Phi_\tau^{j} w(q),
\]
and note that since $w$ is compactly supported, the above summation is finite.
Note that $\Phi_\tau$ is not $L^2(\dee q)$ self adjoint with the adjoint given by $(\Phi_\tau)^*= \sum_j (\Phi^{j}_\tau)^*$
\[
  (\Phi^{j}_\tau)^*w\circ \xbf_j = \int_{\R^d} \phi_{\tau}(|x-y|)\tilde{w}_j(y)\dy.
\]
The basic properties of these kinds of mollifiers are classical, however, we include sketches for completeness.
Due to the compatibility between definitions \eqref{def:Ls} and \eqref{def:Phij}, and the fact that the properties we are interested in are purely local, the results follow from the corresponding statements on $\R^d$.
We sketch the details of this in the first lemma, which is the most delicate of the estimates we require.

\begin{lemma} \label{lem:PhiComm}
For all $\sigma \in [0,1)$, $U \subset \cM$ bounded, geodesic balls, there holds the following uniformly in $\tau \in (0,\delta)$ and uniformly in $C^3$ bounded sets of $Y \in \mathfrak{X}(U)$, for all $w \in C^\infty_c(U)$ (identifying $\Lambda^0 = L^1$), then the commutator $[Y,\Phi_\tau] = Y\Phi_\tau - \Phi_\tau Y$ satisfies
\begin{align}
\norm{[Y, \Phi_{\tau}]^* w }_{L^1} &\lesssim_U \tau^\sigma\norm{w}_{\Lambda^\sigma} \label{ineq:FL}. 
\end{align}
\end{lemma}
\begin{proof}
It suffices to show that the lemma holds for all $\Phi^{j}$. Additionally, since $Y^*=  -Y - \Div Y$ and by the fact that $Y$ is smooth and bounded, we easily obtain
\[
  \|[\Div Y, (\Phi^{j})^*]w\|_{L^1}\leqc_U \tau \|w\|_{ L^1},
\]
and therefore since $[Y,\Phi_\tau]^* = -[Y,\Phi_\tau^*] - [\Div Y,\Phi_\tau^*]$ it suffices to prove that
\begin{equation}\label{eq:sufficient-mollifier}
  \|[Y,(\Phi_\tau^{j})^*]w\|_{L^1} \leqc_U \tau^\sigma \|w\|_{\Lambda^\sigma}.
\end{equation}

With this goal in mind, writing $Y(x) = a_k(x) \partial_{x_k}$ (using Einstein notation summation) as the local $\xbf_j$ parameterization of the vector field $Y$, we find
\[
\norm{[Y, (\Phi_{\tau}^{j})^*] w}_{L^1} = \int_{\R^d} \abs{ \int_{\R^d} \Big(a_k(x) \partial_{x_k}\phi_\tau(|x-y|) \tilde{w}_j(y) - \phi_\tau(|x-y|) a_k(y) \partial_{y_k} \tilde{w}_j(y)\Big)\dy } J_j(x) \dx. 
\]
Integrating by parts in $\partial_{y_k}$, using $\partial_{x_k}\phi_\tau(|x-y|)=-\partial_{y_k}\phi_{\tau}(|x-y|)$, $\partial_{y_k} \tilde{w}_j(x) = 0$ and that
\[
  \int_{\R^d} a_k(x)\partial_{x_k}\phi_\tau(|x-y|)\tilde{w}_j(x)\dy = 0,
\] 
we obtain
\begin{align*}
 \norm{[Y, (\Phi_{\tau}^{j})^*] w}_{L^1} & \lesssim \int_{\R^d} \abs{ \int_{\R^d} (a_k(x) - a_k(y))\partial_{x_k}\phi_\tau(|x-y|) \left(\tilde{w}_j(y) - \tilde{w}_j(x)\right) \dy } J_j(x)\dx \\
 & \hspace{-2cm} + \int_{\R^d} \abs{ \int_{\R^d} \partial_{y_k}(a_k(y))\,\phi_\tau(|x-y|) \left(\tilde{w}_j(y) - \tilde{w}_j(x)\right) \dy } J_j(x) \dx.
\end{align*}
Using that $\abs{a_k(x) - a_k(y)}\lesssim \abs{x-y}$ and $\partial_{y_k}(a_k(y))\leqc 1$ gives
\[
\norm{[Y, (\Phi_{\tau}^{j})^*]w}_{L^1} \lesssim \int_{\R^d} \int_{\R^d} \abs{\frac{\abs{x-y}}{\tau^{d+1}} \phi'\left(\frac{\abs{x-y}}{\tau} \right)} + \abs{\frac{1}{\tau^d} \phi\left(\frac{\abs{x-y}}{\tau} \right)} \abs{\tilde{w}_j(y) - \tilde{w}_j(x)}  J_j(x)\dy\dx. 
\]
Then, making the change of variables $y = x + h$, we obtain from the definition of $\Lambda^\sigma$ \eqref{def:Ls}, we obtain \eqref{eq:sufficient-mollifier}.
\end{proof}

Next we prove the following regularization estimate. 

\begin{lemma}
For all $\sigma \in [0,1)$, for all $U \subset \cM$ open, bounded geodesic balls, there holds uniformly over $\tau \in (0,1)$ and uniformly over bounded $C^2$ sets of $Y \in \mathfrak{X}(K)$, for all $w \in C^\infty_c(U)$ and $p \in [1,\infty]$,
\begin{align}
\norm{\tau Y \Phi^\ast_{\tau} w}_{L^p} + \norm{\tau Y \Phi_{\tau} w}_{L^p}& \lesssim_U  \norm{w}_{L^p} \label{ineq:DRegPhiLp} \\ 
\norm{\tau Y \Phi^\ast_{\tau} w}_{L^1} + \norm{\tau Y \Phi_{\tau} w}_{L^1} & \lesssim_U  \tau^\sigma \norm{w}_{\Lambda^\sigma}.  \label{ineq:DRegPhiLams} \\
\norm{ \Phi_{\tau}^\ast \tau Y w}_{L^1} + \norm{\Phi_{\tau} \tau Y w }_{L^1} & \lesssim_U  \tau^\sigma \norm{w}_{\Lambda^\sigma}.  \label{ineq:DRegPhiLamsInt}
\end{align}
\end{lemma} 
\begin{proof}
We will only consider the case of $\Phi$; the corresponding estimates involving $\Phi^\ast$ are similar. 
We proceed with a proof that is similar to, but much simpler than, that used in Lemma \ref{lem:PhiComm}.
We consider first \eqref{ineq:DRegPhiLams}; \eqref{ineq:DRegPhiLamsInt} follow from similar arguments. 
Using the average zero property again,  
\begin{align*}
\norm{\tau Y \Phi_\tau^{j} w_j}_{L^1} & \lesssim \int_{\R^d} \abs{\int_{\R^d} \tau a_k(x) \partial_{x_k} \phi_\tau(\abs{x-y}) \left(\tilde{w}_j(y) - \tilde{w}_j(x)\right) J_j(x) \dy} J_j(x) \dx \\
& \quad + \int_{\R^d} \abs{\int_{\R^d} \tau a_k(x) \partial_{x_k} \phi_\tau(\abs{x-y}) \left(J_j(y) - J_y(x) \right) \tilde{w}_j(y)  \dy} J_j(x) \dx \\ 
&\lesssim \tau^\sigma \norm{w}_{\Lambda^\sigma}, 
\end{align*}
which is the desired estimate.

Turn next to \eqref{ineq:DRegPhiLp}. The $p=1$ case follows from a straightforward variant of the argument used for \eqref{ineq:DRegPhiLams}.
For the end-point $p=\infty$ we have by Young's inequality
\[
\abs{\tau Y \Phi_\tau^{j} w_j\circ\xbf_j(x)} \lesssim  \int_{\R^d} \tau \abs{a_k(x) \partial_{x_k} \phi_\tau(\abs{x-y}) \tilde{w}_j(y)  }  \dy  \lesssim \norm{w}_{L^\infty}. 
\]
The intermediate $p$ follows by the Riesz-Thorin interpolation theorem. 
\end{proof}

Next, we introduce directional regularizations with respect to a given vector field $Y \in \mathfrak{X}$, as done in [Section 5; \cite{H67}]. 
Accordingly, for each $\varphi \in C^\infty_c([-1,1])$ and $\tau \in (0,1)$ define
\begin{align}
\varphi_{\tau Y} w := \int_\R (e^{t Y} w)\, \varphi_{\tau}(t) \dt, 
\end{align}
where $\varphi_\tau (t) := \frac{1}{\tau}\varphi(\tau^{-1}t)$. Note that since $(e^{t Y})^* = (e^{t Y^*})$, we have
\[
	(\varphi_{\tau Y})^* w = \varphi_{\tau Y^*} w = \int_\R (e^{tY^*}w)\,\varphi_\tau(t)\dt, 
\]
a property that will be used repeatedly in the sequel. 

First we record the basic property that these regularizers are bounded on $L^p$.
The proof is straightforward and is omitted for brevity. 
\begin{lemma}\label{lem:phiXLpBds}
For any $Y \in \mathfrak{X}$, for any open, bounded geodesic ball $U \subset \cM$, and $\varphi\in C^\infty_c([-1,1])$ there holds for all $p \in [1,\infty]$, and $w \in C^\infty_c(U)$, 
\begin{align}
\norm{(\varphi_{\tau Y})^* w}_{L^p} & \lesssim \norm{w}_{L^p} \\
\norm{\Phi_\tau w}_{L^p} + \norm{\Phi_\tau^\ast w}_{L^p} & \lesssim \norm{w}_{L^p}. 
\end{align}
\end{lemma}

Next, we note that the regularizations, the adjoint regularizations, and vector field exponentials are bounded in the $\Lambda^s$ space.

\begin{lemma} \label{lem:RegssBds}
For $\abs{t} \leq 1$, $\tau \in (0,1)$ and $\sigma \in [0,1)$,
for all open, bounded geodesic balls $U \subset \cM$ and $w \in C^\infty_c(U)$, there holds
\begin{align}
\norm{e^{tY} w}_{\Lambda^\sigma} & \lesssim \norm{w}_{\Lambda^\sigma} \label{ineq:etYLSbdd} \\
\norm{e^{tY^*} w}_{\Lambda^\sigma} & \lesssim \norm{w}_{\Lambda^\sigma} \\
\norm{(\varphi_{\tau Y})^\ast w}_{\Lambda^\sigma} & \lesssim \norm{w}_{\Lambda^\sigma}.  \label{ineq:vphinuLams} \\ 
\norm{\Phi_\tau w}_{\Lambda^\sigma} + \norm{\Phi_\tau^\ast w}_{\Lambda^\sigma} & \lesssim \norm{w}_{\Lambda^\sigma}.  \label{ineq:PhiLams}
\end{align}
\end{lemma}
\begin{proof}

The last estimate \eqref{ineq:PhiLams} follows analogously to the corresponding H\"older-type estimate on $\R^d$; we omit the details for brevity.
The second and third estimates follow easily from the first estimate, and hence we turn to the proof of \eqref{ineq:etYLSbdd}. 

After applying coordinate parameterization to reduce to the case of $\R^d$, the first estimate will follow from an adaptation of [Lemma 4.2; \cite{H67}] to $L^1$ and more general geometry.
We include the proof for the readers' convenience. 
Recalling the coordinate charts defined in for \eqref{def:Ls}, denote $\tilde{w}_j = (\chi_j w)\circ \mathbf{x}_j$
and note that $\exists g_{Y,j}(t,x) \in C^\infty(\R^d)$ such that $(\chi_j e^{tY} w) \circ \mathbf{x}_j = \tilde{w}_j(g_{Y,j}(t,x))$ satisfying exactly analogous properties that $h_Y$ satisfies.
Therefore, by the definition \eqref{def:Ls}
\begin{align*}
\norm{e^{tY} w}_{\Lambda^\sigma} = \norm{e^{tY} w}_{L^1} + \sup_{\tau \in \R^d : \abs{\tau}< 1}\sum_{j} \int_{\R^d} \frac{\abs{\tilde{w}_j( g_{Y,j}(t, x + \tau) ) - \tilde{w}_j(g_{Y,j}(t,x))}}{\abs{\tau}^{s}} J_j(x) \dx. 
\end{align*}
The first term is estimated by \eqref{ineq:etYLpbd}.
We next turn to the latter term.
By the group property of $e^{tY}$ we may restrict $\abs{t} < \delta'$ to any $\delta'$ sufficiently small depending only on $Y$ and the ball $U$ (i.e. local geometrical information).
Furthermore, we may restrict $\abs{\tau} < \delta'$ (as the contribution from $\abs{\tau} > \delta'$ is controlled by $L^1$). 
Next, observe that (using the same trick as in [Lemma 4.2; \cite{H67}])
\begin{align*}
\int_{\R^d} \frac{\abs{\tilde{w}_j( g_{Y,j}(t, x + \tau) ) - \tilde{w}_j(g_{Y,j}(t,x))}}{\abs{\tau}^{s}} J_j(x) \dx & \lesssim \\ & \hspace{-4cm}  \abs{\tau}^{-d}\int_{\abs{\zeta} < \abs{\tau} }\int_{\R^d}  \frac{\abs{\tilde{w}_j( g_{Y,j}(t, x + \tau) ) - \tilde{w}_j(g_{Y,j}(t,x) + \zeta)}}{\abs{\tau}^{s}} J_j(x) \dx \dee \zeta + \norm{w}_{\Lambda^s}.   
\end{align*}
Then for each fixed $j$ we make the $(t,\tau)$-dependent change of variables from $(x,\zeta)\to (\gamma,\eta)$ in the integrals via
$\eta + \gamma= g_{Y,j}(t,x+\tau)$ and $\gamma = g_{Y,j}(t,x) + \zeta$.
Note that for $t=0$ this reduces to $\gamma = x+\zeta$ and $\eta = -\zeta$ and for $\tau = 0$ this becomes $\gamma = g_{Y,j}(t,x) + \zeta$ and $\eta=-\zeta$.
Hence, for $\delta'$ chosen small enough, the Jacobian $\mathcal{J}_{t,\tau}$ for this change of variables satisfies $\abs{\mathcal{J}_{t,\tau}-1} < 1/2$ and we obtain 
\begin{align*}
\int_{\R^d} \frac{\abs{\tilde{w}_j( g_{Y,j}(t, x + \tau) ) - \tilde{w}_j(g_Y(t,x))}}{\abs{\tau}^{s}} J_j(x) \dx \lesssim  \abs{\tau}^{-d}\int_{\R^d}\int_{\mathcal{D}_{t,\tau}} \frac{\abs{\tilde{w}_j( \eta + \gamma ) - \tilde{w}_j(\gamma)}}{\abs{\tau}^{s}}  \dee \eta\dee \gamma + \norm{w}_{\Lambda^s}, 
\end{align*}
over some bounded integration domain $\mathcal{D}_{t,\tau}\subseteq \R^d$ satisfying $\mathrm{vol} \,\mathcal{D}_{t,\tau}\lesssim \abs{\tau}^d$. 
The desired estimate \eqref{ineq:etYLSbdd} follows upon noting that  $\abs{\eta} \lesssim \tau$ on $\mathcal{D}_{t,\tau}$. 
\end{proof} 

In a similar vein, the chain rule implies the following estimates. 
\begin{lemma} \label{lem:DbdsRegs}
For all bounded, geodesic balls $U \subset \cM$, for all $\abs{\tau} \leq 1$ and $\forall k \geq 2$, the following holds $\forall w \in C^\infty_c(U)$, 
\begin{align}
\sup_{Z \in \mathfrak{X}: \norm{Z}_{C^k} \leq 1} \norm{ Z e^{\tau Y} w}_{L^\infty} & \lesssim \sup_{Z \in \mathfrak{X}: \norm{Z}_{C^k} \leq 1} \norm{Z w}_{L^\infty}, \label{ineq:Zet1} \\
\sup_{Z \in \mathfrak{X}: \norm{Z}_{C^k} \leq 1} \norm{ Z^* e^{\tau Y^*} w}_{L^1} & \lesssim \sup_{Z \in \mathfrak{X}: \norm{Z}_{C^k} \leq 1} \norm{Z^* w}_{L^1}, \label{ineq:Zet2} \\
\sup_{Z \in \mathfrak{X}: \norm{Z}_{C^k} \leq 1} \norm{ Z^* (\varphi_{\tau Y})^\ast w}_{L^1} & \lesssim \sup_{Z \in \mathfrak{X}: \norm{Z}_{C^k} \leq 1} \norm{Z^* w}_{L^1}. \label{ineq:ZetVarphi2}
\end{align}
\end{lemma}
\begin{proof}
Estimates \eqref{ineq:Zet1}, \eqref{ineq:Zet2} follow from the chain rule and \eqref{ineq:ZetVarphi2} then follows from the definition of the regularizers. 
\end{proof} 

The next lemma characterizes the regularization property of the regularizers.
\begin{lemma} \label{lem:varphiRegs}
For all bounded, geodesic balls $U \subset \cM$ and $w \in C^\infty_c(U)$,
\begin{align}
\norm{ (Y \varphi_{\tau Y})^\ast w}_{L^1} & \lesssim \sup_{\abs{t} \leq \tau} \norm{e^{t Y^*} w - w}_{L^1}. 
\end{align}
\end{lemma}
\begin{proof}
We have
\begin{align}
  (\varphi_{\tau Y})^\ast Y^* w & = \int_\R (e^{t Y^*}Y^*w) \varphi_\tau (t)\, \dt = \int_\R \frac{\dee}{\dt}\left(e^{t Y^*} w-w\right) \varphi_\tau(t)\, \dt  = -\int_\R \left(e^{t Y^*}w - w\right)\varphi^\prime_\tau(t)\,  \dt. 
\end{align}
The result then follows by Minkowski's inequality.   
\end{proof}

We will also need the $L^\infty$ regularization property.
\begin{lemma} \label{lem:LinfRegVarphi}
For all open, bounded geodesic balls $U \subset \cM$ and $w \in C^\infty_c(U)$, 
\begin{align}
\norm{Y \varphi_{\tau Y} w}_{L^\infty} \lesssim \frac{1}{\tau}\norm{w}_{L^\infty}. 
\end{align} 
\end{lemma}
\begin{proof}
This follows by a straightforward variant of the proof of Lemma \ref{lem:varphiRegs}. 
\end{proof}

Next, we show that the H\"older-type regularity classes are natural for controlling convergence of the operators.
It is natural to specialize to the specific form in which we are using it. 
\begin{lemma} \label{lem:L1varphiConv}
For all open, bounded geodesic balls $U \subset \cM$ and $w \in C^\infty_c(U)$, there holds for $\varphi \in C^\infty_c([-1,1])$, $\varphi \geq 0$ and $\int_\R \varphi(t)\,\dt = 1$,
\begin{align}
\norm{(\varphi_{\tau X_I})^\ast w - w}_{L^1} & \leqc \sup_{|t|\leq \tau}\norm{e^{tX_I^*}w - w}_{L^1}. 
\end{align}
\end{lemma}
\begin{proof}
By Minkowski's inequality,
\[
\begin{aligned}
\norm{(\varphi_{\tau X_I})^\ast w - w}_{L^1} &\leq \int_\R \norm{e^{t X_I^*} w - w}_{L^1} \varphi_{\tau}(t) \,\dt \leq \sup_{|t|\leq \tau}\norm{e^{t X_I^*} w - w}_{L^1}.
\end{aligned}
\]
\end{proof}

The following Lemma will be useful when measuring regularity of $\varphi_{\tau X_I}$ with respect to $X_J$. 
\begin{lemma}\label{lem:Regphi}
For all open, bounded geodesic balls $U \subset \cM$ and $w \in C^\infty_c(U)$, for $I,J\in \mathcal{J}$, there holds
\begin{align}
	\sup_{|t|\leq \tau^{m(J)}} \norm{e^{t X_J^*}(\varphi_{\tau^{m(I)} X_I})^*w -(\varphi_{\tau^{m(I)} X_I})^*w}_{L^1}  & \leqc \sup_{|t|\leq \tau^{m(J)}}\norm{e^{t X_J^*}w -w}_{L^1} \\ & \quad + \sup_{|t|\leq \tau^{m(I)}}\norm{e^{t X_I^*}w -w}_{L^1}. 
\end{align}
\end{lemma}
Now, we are ready to define the regularizer $S_\tau$. Let us now give $\mathcal{J}$ a total ordering so that $m(I)$ is an increasing function of $I \in \mathcal{J}$ and we denote $\cJ_\infty = \cJ \cup \set{\infty}$.
We define $S_\tau$ in terms of an \emph{ascending}, ordered composition of regularizing operators 
\begin{align}
S_\tau w := \left(\prod_{I \in \cJ} \varphi_{\tau^{m(I)} X_I}\right) \Phi_{\tau^{1/\sigma}} w. 
\end{align}
This regularizer is similar, but not quite exactly the same as that defined in \cite{H67} due to the inclusion of more regularization operators. 
However, we will ultimately use $S_\tau^*$ as the regularizer, which is a little more subtle to work with. 
Analogous to \cite{H67}, we also define the truncated regularizer, for all $J \in \mathcal{J}$,  
\[
S_\tau^Jw := \left(\prod_{I \in \cJ : I \geq J} \varphi_{\tau^{m(I)} X_I}\right)\Phi_{\tau^{1/\sigma}}w,
\]
with $S_\tau^\infty := \Phi_{\tau^{1/\sigma}}$ when $J = \infty$.
The remainder of the subsection is dedicated to proving Lemma \ref{lem:Sreg}. The first step is to obtain the $L^1$ convergence. 
\begin{lemma}
For all open, bounded geodesic balls $U \subset \cM$ and $w \in C^\infty_c(U)$,  
\begin{align*}
\norm{S^\ast_\tau w - w}_{L^1} \lesssim \tau \abs{w}_{M}.
\end{align*}
\end{lemma}
\begin{proof}
For any finite family of $L^1 \to L^1$ bounded linear operators $Z_1, Z_2,...,Z_k$ we have 
\begin{align*}
\norm{Z_1 Z_2 ... Z_k w - w}_{L^1} \leq \sum_{j=1}^k \norm{(Z^1...Z^{j-1})(Z_j w  - w)}_{L^1} \lesssim \sum_{j=1}^k \norm{Z_j w - w}_{L^1}. 
\end{align*}
The result then follows from Lemma \ref{lem:L1varphiConv}. 
\end{proof}

The next Lemma is crucial for characterizing the regularization properties of $(S_\tau^J)^\ast$ in $L^1$.  
This is the adjoint analogue of [Lemma 5.2; \cite{H67}], which is a little more technical. 
\begin{lemma} \label{lem:SJXIs}
For all open, bounded geodesic balls $U \subset \cM$ and $w \in C^\infty_c(U)$,  there holds for any multi-indices $J \leq I$, 
\begin{align}
\norm{(\tau^{1/\sigma}Y S_\tau^J)^\ast w}_{L^1} & \lesssim \tau \norm{w}_{\Lambda^\sigma} \label{ineq:EatTheError} \\ 
\norm{(\tau^{m(I) }X_I S_\tau^{J})^\ast w}_{L^1} & \lesssim \sum_{I' \in \cJ : I' \geq J} \sup_{\abs{t} \leq \tau^{m(I^\prime)}} \norm{e^{t X_{I'}^*} w - w}_{L^1} + \tau \norm{w}_{\Lambda^\sigma}. \label{ineq:SstarMultReg} 
\end{align}
\end{lemma}
Before we continue, we define for two vector fields $X$ and $Y$
\begin{align}
e^{t\,\ad(X)}Y := (e^{t X})_\sharp Y, 
\end{align}
where $(e^{t X})_\sharp Y$ denotes the pushforward of $Y$ as a vector field under the diffeomorphism $e^{t X}$. This is just the adjoint representation of group element $e^{t X}$ on the Lie algebra of vector fields. It will be useful to expand $e^{t\, \ad(X)}Y$ in a Taylor expansion.
\begin{lemma}\label{eq:Taylor-expansion}
For two smooth vector fields $X,Y$, $t\in [-1,1]$ and $N\in \N$, there exists a smooth bounded vector field $Y_{N,t}$, locally uniformly bounded in $C^k$ ($\forall k$) on $t \in [-1,1]$, such that
\[
(e^{t\,\ad(X)}Y) = \sum_{0\leq k< N}\frac{t^k}{k!}(\ad(X)^kY) +  \frac{t^N}{N!}\,Y_{N,t}
\]
\end{lemma}
A simple consequence of the chain rule implies that $Y(e^{t X}w) = e^{t X}(e^{t\,\mathrm{ad}(X)}Yw)$ and gives the following useful representation for the smoothing operators
\[
	Y \varphi_{\tau X}w = \int_\R \left(e^{t X}(e^{-t \,\ad(X)}Yw)\right)\,\varphi_{\tau}(t)\dt. 
\]
Lemma \ref{eq:Taylor-expansion} then gives the following formula for $Y\varphi_{\tau X}$ (used also in \cite{H67}). 
\begin{lemma} \label{lem:RegsExpansion}
For each $\varphi\in C^\infty_c((-1,1))$, $k\in \N$ and vector field $X$ define
\begin{align}
\hat{\varphi}^k_{\tau X}w := \int_\R (e^{t X}w)\,\hat{\varphi}^k_\tau(t)\,\dt, \quad \text{where}\quad\hat{\varphi}^k(t) := \frac{t^k}{k!} \varphi(t) \in C^\infty_c((-1,1)). \label{def:phihatk}
\end{align}
Then for two smooth vector fields $X,Y$, $\tau \in (0,1]$ and $N\in \N$, the following holds
\begin{align}\label{eq:commutator-reg}
Y\varphi_{\tau X}w
&= \sum_{0\leq k < N} \tau^{k}\left(\hat{\varphi}^k_{\tau X}(\ad(-X)^kYw)\right) + \tau^N R^N_{\tau X}w,
\end{align}
where
\begin{align}
	R^N_{\tau X} w := \int_\R (e^{tX} Y_{N,t} w)\, \hat{\varphi}^N_\tau(t)\,\dt. \label{def:RNtX}
\end{align}
\end{lemma}

We are now ready to prove Lemma \ref{lem:SJXIs}. 
\begin{proof}[Proof of Lemma \ref{lem:SJXIs}] \hspace{1em}

\textbf{Proof of \eqref{ineq:EatTheError}}.
We proceed by induction. For $J = \infty$ the result follows from \eqref{ineq:DRegPhiLamsInt}.
Hence, we next assume that the result holds for all $J'$ with $J' > J$ and prove that it also holds for $J$.
We begin with the decomposition 
\begin{align*}
(S_\tau^J)^\ast = (S_\tau^{J'})^\ast (\varphi_{\tau^{m(J)} X_J})^*.  
\end{align*}
By a trivial application of Lemma \ref{lem:RegsExpansion} with $N=1$ and $X=X_J$, there exists a smooth bounded vector field $Y_{1,t}$ such that (recall Definition \eqref{def:RNtX}), 
\[
	(\tau^{1/\sigma} YS_\tau^J)^* = (\tau^{1/\sigma} YS_\tau^{J^\prime})^*(\varphi_{\tau^{m(J)}X_J})^* + \tau^{m(J)+1/\sigma}(S_\tau^{J^\prime})^*(R^1_{\tau^{m(J)} X_J})^*.
\]
By the induction hypothesis and Lemma \ref{lem:RegssBds} we have for the first term above
\[
	\|(\tau^{1/\sigma} YS_\tau^{J^\prime})^*(\varphi_{\tau^{m(J)}X_J})^*w\|_{L^1} \leqc \tau \|(\varphi_{\tau^{m(J)}X_J})^*w\|_{\Lambda^\sigma} \leqc \tau \|w\|_{\Lambda^\sigma}.
\]
A similar estimate holds for the second term using Minkowski's inequality 
\[
\begin{aligned}
	\|(\tau^{1/\sigma}S_\tau^{J^\prime})^*(R^1_{\tau^{m(J)} X_J})^*w\|_{L^1} &\leq \int_{\R} \|(\tau^{1/\sigma}Y_{1,t}S^{J^\prime}_\tau)^* e^{t X_J^*}w \|_{L^1}\hat{\varphi}^1_{\tau^{m(J)}}(t)\dt\\
	 &\leqc \tau\|e^{t X_J^*}w\|_{\Lambda^\sigma} \leqc \tau \|w\|_{\Lambda^\sigma}
\end{aligned}
\]
and the estimate \eqref{ineq:EatTheError} now follows.

\textbf{Proof of \eqref{ineq:SstarMultReg}}.  First we note that if $I = J$ then we have for $J^\prime$ the smallest element such that $J^\prime > J$ by Lemma \ref{lem:varphiRegs} and the $L^1$ boundedness of $(S^J_\tau)^*$
\[\begin{aligned}
	\|(X_J S^J_\tau)^*w\|_{L^1} &= \|(S^{J^\prime}_\tau)^*(X_J \varphi_{\tau^{m(J)}X_J})^*w\|_{L^1} \leqc \sup_{|t|\leq \tau^{m(J)}}\|e^{tX_J^*}w - w\|_{L^1}. 
\end{aligned}
\]
When $I > J$, we proceed by induction. First of all, the result follows by definition of \eqref{def:Phij} if $J = \infty$. Hence, we next assume that the result holds for all $J'$ with $J' > J$ and prove that it also holds for $J$ the largest element less than $J^\prime$. Again writing
\[
	(S^J_\tau)^* = (S^{J^\prime}_\tau)^*(\varphi_{\tau^{m(J)}X_J})^*
\]
and using Lemma \ref{lem:RegsExpansion} we obtain, $\forall N \geq 1$, 
\[
\begin{aligned}
(\tau^{m(I)}X_IS^{J}_\tau)^* &= \sum_{0\leq k < N}(\tau^{m(I^\prime_k)}X_{I^\prime_k} S^{J^\prime}_\tau)^*(\hat{\varphi}^k_{\tau^{m(J)}X_J})^* + (\tau^{m(I) + Nm(J)}S^{J^\prime}_\tau)^*(R^N_{\tau^{m(J)}X_J})^*\\
&=: T_1 + T_2, 
\end{aligned}
\]
where
\[
	X_{I_k^\prime} := \ad(-X_J)^kX_I, \quad\text{and}\quad  m(I^\prime_k) = m(I) + km(J).
\]
If we choose $N$ large enough so that
\[
	m(I) + Nm(J) \geq \frac{1}{\sigma},
\]
we can treat the ``error'' term $T_2$ by applying Minkowski's inequality and \eqref{ineq:EatTheError},
\[
\begin{aligned}
	\|T_2w\|_{L^1} &\leq \int_\R\|(\tau^{1/\sigma}Y_{N,t}S^{J^\prime}_\tau)^*e^{tX_J^*}w\|_{L^1} \,\hat{\varphi}_{\tau^{m(J)}}^N(t)\dt \leqc \tau \|e^{t X_J^*}w\|_{\Lambda^\sigma} \leqc \tau \|w\|_{\Lambda^\sigma}. 
\end{aligned}
\]
Since $I_k^\prime \geq J^\prime$ for all $k\geq 0$, we can use the induction hypothesis and Lemmas \ref{lem:Regphi}, \ref{lem:phiXLpBds} (it is straightforward to check that $\hat{\varphi}^k$ satisfies the same) to treat the first term
\[
\begin{aligned}
	\norm{T_1w}_{L^1} &\leqc \sum_{0\leq k < N}\sum_{I^\prime \geq J^\prime} \sup_{|t|\leq \tau^{m(I^\prime)}}\norm{e^{t X_{I^\prime}^*}(\hat{\varphi}^k_{\tau^{m(J)}X_J})^*w-(\hat{\varphi}^k_{\tau^{m(J)}X_J})^*w}_{L^1} + \tau \norm{(\hat{\varphi}^k_{\tau^{m(J)}X_J})^*w}_{\Lambda^\sigma}\\
	&\leqc \sum_{I^\prime \geq J} \sup_{|t|\leq \tau^{m(I^\prime)}}\|e^{t X_{I^\prime}^*}w - w\|_{L^1} + \tau \|w\|_{\Lambda^\sigma}
\end{aligned}
\]
as desired.
\end{proof} 

The main commutator estimate is a consequence of the following.
\begin{lemma}
For all $J \in \mathcal{J}$, $U \subset \cM$ open, bounded geodesic ball, there holds $\forall w \in C^\infty_c(U)$ and $\tau \in (0,1)$, 
\begin{align}
\norm{[\tau^2 X_0,S_\tau^J]^*w}_{L^1} & \lesssim \sum_{I \in \mathcal{J}: I \geq J} \sup_{|t|\leq \tau^{m(I)}}\norm{e^{t X_I^*} w - w}_{L^1} + \tau \norm{w}_{\Lambda^s}. 
\end{align}
\end{lemma}
\begin{proof}
As in the proof of \eqref{ineq:SstarMultReg} above, we proceed by induction.
Firstly, the estimate holds for $J = \infty$ due to the commutator estimate Lemma \ref{lem:PhiComm}. 
As above, assume the result holds for all $J'$ with $J' > J$ and prove that it also holds for $J$, writing
\begin{align}
(S_t^J)^\ast = (S_t^{J'})^\ast (\varphi_{t^{m(J)} X_J})^*. 
\end{align}
Then,
\begin{align}
[\tau^2 X_0, (S_\tau^J)]^* = [\tau^2 X_0,S_\tau^{J'}]^* (\varphi_{\tau^{m(J)} X_J})^* + (S_\tau^{J'})^\ast [\tau^2 X_0, \varphi_{\tau^{m(J)} X_J}]^*. \label{eq:SstCsplit}
\end{align}
By the inductive hypothesis and Lemmas \ref{lem:Regphi} and Lemmas \ref{lem:phiXLpBds} we have  
\begin{align}
&\norm{[\tau^2 X_0,S_\tau^{J'}]^* (\varphi_{\tau^{m(J)} X_J})^*w}_{L^1}\\
 &\hspace{1in}\lesssim \sum_{I' \geq J'} \sup_{|t|\leq \tau^{m(I')}}\norm{e^{tX_{I'}^*}(\varphi_{\tau^{m(J)} X_J})^* w - (\varphi_{\tau^{m(J)} X_J} )^*w}_{L^1} + \tau \norm{(\varphi_{\tau^{m(J)} X_J})^* w}_{\Lambda^\sigma}\\
 &\hspace{1in}\leqc \sum_{I^\prime \geq J} \sup_{|t|\leq \tau^{m(I^\prime)}}\norm{e^{tX_{I^\prime}^*}w - w}_{L^1} +\tau \|w\|_{\Lambda^\sigma}.
\end{align}
This term in \eqref{eq:SstCsplit} is consistent with the desired result.

For the second term in \eqref{eq:SstCsplit}, by Lemma \ref{lem:RegsExpansion} we have (note the cancellation which eliminates the $k=0$ term) 
\[
	\begin{aligned}
	(S^{J^\prime}_\tau)^*[\tau^2 X_0,\varphi_{\tau^{m(J)}X_J}]^* &= \sum_{0 < k < N} (\tau^{m(I^\prime_k)} X_{I^\prime_k}S^{J^\prime}_\tau)^*(\hat{\varphi}^k_{\tau^{m(J)}X_J})^* + (\tau^{2 + Nm(J)}S^{J^\prime}_\tau)^*(R^N_{\tau^{m(J)}X_J})^*
	\end{aligned}
\]
where
\[
	X_{I_k^\prime} := \ad(-X_J)^kX_0, \quad\text{and}\quad  m(I^\prime_k) = 2 + km(J).
\]

The treatment of these terms is exactly the same as in the proof of \eqref{ineq:SstarMultReg} upon taking $2 + N m(J) \geq \frac{1}{\sigma}$. We omit the repetitive details for the sake of brevity. 
\end{proof}

Finally, we prove the required $L^\infty$ regularization estimate.
\begin{lemma}
Let $U \subset \cM$ be open, bounded geodesic ball. Let $I$ be any multi-index and $J \leq I$. Then, $\forall w \in C^\infty_c(U)$, 
\begin{align*}
\norm{\tau^{m(I)} X_I S_\tau^J  w}_{L^\infty} \lesssim_U \norm{w}_{L^\infty}  
\end{align*}
\end{lemma}
\begin{proof}
This is done by induction as in previous lemmas.
The case $J=\infty$ follows from \eqref{ineq:DRegPhiLp}.
Assume that the result holds for all $J'$ with $J' > J$ and write
\begin{align*}
S_\tau^J = \varphi_{\tau^{m(J)} X_J} S_\tau^{J'}. 
\end{align*}

\textbf{Case 1: $I > J$.} We apply the Taylor expansion of Lemma \ref{lem:RegsExpansion} (recalling definitions \eqref{def:phihatk} and \eqref{def:RNtX})
\[
	\tau^{m(I)}X_I S_\tau^J = \sum_{0 \leq k < N}\hat{\varphi}^k_{\tau^{m(J)}X_J}\tau^{m(I^\prime_k)}X_{I^\prime_k}S_\tau^{J^\prime} + \tau^{m(I) + Nm(J)} R_{\tau^{m(J)}X_J}^N S^{J^\prime}_\tau, 
\]
where
\[
	X_{I_k^\prime} := \ad(-X_J)^kX_I, \quad\text{and}\quad  m(I^\prime_k) = m(I) + km(J).
\]
Then we have,
\begin{align*}
\norm{\tau^{m(I)}X_I S_\tau^J w}_{L^\infty} & \lesssim \sum_{0 \leq \nu < N} \norm{\tau^{m(I^\prime_k)}X_{I^\prime_k}S_\tau^{J^\prime} w}_{L^\infty}  + \sup_{Z \in \mathfrak{X} :\norm{Z}_{C^3} \leq 1} \norm{\tau^{m(I) + Nm(J)} Z S_\tau^{J'} w}_{L^\infty}. 
\end{align*}
For $N$ chosen sufficiently large, the latter term is estimated by Lemma \ref{lem:DbdsRegs} an \eqref{ineq:DRegPhiLp} to produce
\begin{align*}
\sup_{Z \in \mathfrak{X} :\norm{Z}_{C^3} \leq 1} \norm{\tau^{m(I) + Nm(J)} Z S_\tau^{J'} w}_{L^\infty} \lesssim \norm{w}_{L^\infty}.  
\end{align*}
Since $I'_k \geq J'$ and hence the first $N$ terms in the summation are estimated by induction.

\textbf{Case 2: $I = J$.} In this case, the desired result follows immediately from Lemma \ref{lem:LinfRegVarphi}.
\end{proof}

\section{Projective hypoellipticity and irreducibility for Euler-like systems} \label{ProjBraks}

Our focus for the remainder of the paper is to prove Theorem \ref{thm:critForEulerLikeIntro} for the Lyapunov exponents of Euler-like systems (see \eqref{eq:SDEintroFD22}), 
as well as Corollary \ref{L96} for 
our concrete example, the Lorenz 96 (L96) model \eqref{def:L96intro}. 
Our tools, the Fisher information identity (Theorem \ref{thm:fishInfoIntro}) 
and accompanying hypoelliptic regularity estimate (Theorem \ref{thm:unifHypoRegEstIntro}), 
 depend on the unique existence of absolutely continuous stationary densities for 
projective processes (c.f. Assumptions \ref{ass:WP}, \ref{ass:WP2} and \ref{ass:PHC}). 
As discussed in Section \ref{sec:FI}, these can be reduced to checking 
H\"ormander's condition (Definition \ref{def:Hormander}) and topological irreducibility (Definition \ref{def:Irr})
for the projective process $w_t = (x_t, v_t)$ on $\S M$.
The primary aim of this section is to obtain useful sufficient conditions for checking these for Euler-like models in general, and then to verify them for the L96 model in particular. 

The plan is as follows: in Section \ref{sec:projSpanEulerLike} we present a sufficient condition (Corollary \ref{cor:sln-Lie-suff}) for the Euler-like class to satisfy projective bracket spanning; in Section \ref{sec:L96projSpanning}, we check this condition for the L96 model; and in Section \ref{sec:Irr} we show that topological irreducibility for Euler-like models can be \emph{reduced} to bracket spanning (Proposition \ref{prop:projective irreducibility}) under a natural cancellation-type condition satisfied by most models of interest (including the L96 and Galerkin Navier-Stokes equations). 

\subsubsection*{Standing assumptions}

For the remainder of the paper, we return to the setting of Euler-like SDE 
\begin{align}
\dee x_t^\epsilon = \big(B(x_t^\epsilon, x_t^\epsilon) + \epsilon A x_t^\epsilon \big) \dee t + \sqrt{\epsilon} \sum_{k =1}^r X_k \dee W_t^k \,, 
\end{align}
where as in Section \ref{sec:Intro} we assume $B$ is a nontrivial bilinear form on $\R^n, n \geq 1$
satisfying $\Div B = 0$ and $x \cdot B(x,x) = 0$; the $\{ X_k\}_{k = 1}^r$ are constant vector fields; and $A$ is a symmetric negative-definite $n \times n$ matrix. Recall the notation $X_0^\epsilon(x) := B(x,x) + \epsilon A x$. 

\subsection{Projective spanning for Euler-like systems}
\label{sec:projSpanEulerLike}

For many systems of interest, it can be significantly harder to verify H\"ormander's condition for $\{ \tilde X_k\}$ on $\S \R^n$ than to verify it for the vector fields $\{X_k\}$ on $\R^n$, for the simple reason that the dimension increases substantially ($\dim \S \R^n = 2 n -1$) and the fact that it is significantly harder to isolate `simple' vector fields on $\S \R^n$; this is already the case for the L96 model with additive noise (which we treat in Section \ref{sec:L96projSpanning} below). 

First, we show that for the class of Euler-like models \eqref{eq:SDEintroFD22}, projective spanning as in Assumption \ref{ass:PHC} can be reduced to a transitivity condition on a finite dimensional matrix Lie algebra.
This provides both a reasonable way to check projective spanning for \emph{fixed} dimension and parameters using a brute force computer calculation as well as a framework to begin analytical studies that work in arbitrary dimension and over entire ranges of parameters (such as what is carried out for Lorenz 96 in Section \ref{sec:L96projSpanning} and for Galerkin-Navier-Stokes in \cite{BPS21}).  
To make this more precise, we present here an argument for reducing projective spanning to a combination of 
\begin{itemize}
\item[(i)] the spanning condition for $\set{X_0^\epsilon, X_1,...,X_r}$ on $\R^d$; and
\item[(ii)] the \emph{purely linear} condition that $\mathfrak{sl}(\R^n)$ is generated by a collection $\{H^i\}$ of \emph{constant-valued} $n\times n$ real matrices (defined explicitly in terms of derivatives of the nonlinearity $B(x,x)$) under the standard matrix Lie bracket. 
\end{itemize}

By Proposition \ref{prop:class-proj-span}, we know that verifying the parabolic H\"ormander condition for $\{\tilde{X}_0^\ep,\tilde{X}_1,\ldots,\tilde{X}_r\}$ on $\S \R^n = \R^n\times \S^{n-1}$ is equivalent to checking that $\{X_0^\ep, X_1,\ldots,X_r\}$ satisfies the parabolic H\"ormander condition on $\R^n$ and that the matrix Lie algebra $\mathfrak{m}_x(X_0^\ep;X_1,\ldots,X_r)$ defined by $\eqref{eq:m-lie-alg-def}$ satisfies the transitivity condition
\[
	\{Av - v\langle Av,v\rangle \,:\, A \in \mathfrak{m}_x(X_0^\ep;X_1,\ldots,X_r)\} = T_v\S^{n-1}
\]
for each $(x,v)\in \R^n\times \S^{n-1}$. In general it is a challenge to directly work with $\mathfrak{m}_x(X_0^\ep;X_1,\ldots,X_r)$ as it is not a simple task to classify all vector fields in $\mathrm{Lie}(X_0^\ep;X_1,\ldots X_r)$ that vanish at each $x\in \R^n$. However, in $\R^n$ with additive noise it is often the case that the parabolic Lie algebra generated by $\{X_k\}_{k=0}^r$ contains a spanning collection of constant vector fields $\{\partial_{x_k}\}_{k=1}^n$ (this is often how parabolic H\"ormander as in Definition \ref{def:Hormander} is checked on $\R^n$). In this case, $\mathfrak{m}_x$ can be described more explicitly.
\begin{lemma} \label{lem:consVs}
Let $\{Z_k\}_{k=0}^r \subseteq \mathfrak{X}(\R^n)$ and suppose that $\mathrm{Lie}(Z_0;Z_1,\ldots,Z_r)$ contains the constant vector fields $\{\partial_{x_k}\}_{k=1}^n$. Then
\[
	\mathfrak{m}_x(Z_0;Z_1,\ldots,Z_r) = \{M_Z(x)\,:\, Z \in \mathrm{Lie}(Z_0;Z_1,\ldots,Z_r)\}.
\]
\end{lemma}
\begin{proof}
Our hypothesis $ \{\partial_{x_k}\}_{k=1}^n \subseteq \mathrm{Lie}(X_0;X_1,\ldots,X_r)$ implies that for each $X\in \mathrm{Lie}(X_0;X_1,\ldots,X_r)$ and $x\in \R^n$, the vector field $\hat{X} = X - X(x)$ also belongs to $\mathrm{Lie}(X_0;X_1,\ldots,X_r)$ and satisfies $\hat{X}(x) = 0$. Since $\nabla \hat X = \nabla X$, we have $M_{\hat X}(x) = M_X(x)$, hence $M_X(x)\in \mathfrak{m}_x$.
\end{proof}

Using that $X_0^\epsilon(x) = B(x,x) + \epsilon A x$ allows for further simplification. Define for each $k=1,\ldots,n$ the linear operator
\[
	H^k := \nabla[\partial_{x_k},X_0^\ep] = \partial_{x_k}\nabla B = M_{[\partial_{x_k},X_0^\ep]} \in \mathfrak{sl}(\R^n). 
\]
 Note that $H^k$ is independent of both $x\in M$ and $\ep$. 
Below, $\mathrm{Lie}(H^1,\ldots,H^n)$ denotes the matrix Lie subalgebra
of  $\mathfrak{sl}(\R^n)$ generated by $\{H^1, \cdots, H^n\}$. It is a simple matter to check that if $\{\partial_{x_k}\}_{k=1}^n \subseteq \mathrm{Lie}(X_0^\ep;X_1,\ldots,X_r)$, then one has
\[
    \mathrm{Lie}(H^1,\ldots,H^n) = \mathfrak{m}_x(X_0^\ep;X_1,\ldots,X_r).
\] 
Lemma \ref{lem:consVs} now yields the following. 

 \begin{corollary}\label{cor:sln-Lie-suff}
 Assume (i) $\{\partial_{x_k}\}_{k=1}^n \subseteq \mathrm{Lie}(X_0^\ep;X_1,\ldots,X_r)$ and (ii)
that
 \begin{equation}\label{eq:Lie-H-suff-cond}
\mathrm{Lie}(H^1,\ldots,H^n) = \mathfrak{sl}(\R^n) \, .
\end{equation}
 Then, $\{\tilde{X}_0^\ep,\tilde{X}_1,\ldots,\tilde{X}_r\}$ satisfy the uniform parabolic H\"ormander condition in the sense of Definition \ref{def:UniHormander} as $\epsilon$ is varied in $(0,1]$.
\end{corollary}
\begin{proof}{}
By Lemma \ref{lem:consVs} and Proposition \ref{prop:class-proj-span} it suffices to check that 
\[
\{ A v - v\langle v , A v \rangle : A \in \mathfrak{sl}(\R^n) \} = T_v \S^{n-1}.
\]
holds. This is true since $\mathfrak{sl}(\R^n)$ acts transitively on $\R^n\backslash\{0\}$, i.e., for all $w,v \in \R^n\backslash\{0\}$
there exists $A \in \mathfrak{sl}(\R^n)$ such that $A v = w$ (see, e.g., \cite{boothby1979determination} for a complete classification of transitive algebras on $\R^n\backslash\{0\}$). In particular, this implies transitive action of $V_A$ on $\S^{d-1}$ since for each $v$, the projection of a spanning set in $\R^n$ onto the subspace $T_v\S^{n-1} = v^\perp = \{ v^\prime\in \S^{n-1}\,:\, \langle v,v^\prime\rangle =0\}$ is still spanning. Indeed, is is not hard to show that the smaller subalgebra $\mathfrak {so}(\R^n)$ of skew symmetric matrices also acts transitively on $\S^{n-1}$ (though not on $\R^n\backslash\{0\}$), and so 
the conclusion of this Corollary holds under the following weakening of \eqref{eq:Lie-H-suff-cond}: 
\begin{align*}
\mathfrak{so}(\R^n) \subseteq \mathrm{Lie}(H^1,\ldots,H^n) \, ,
\end{align*}
c.f. Remark \ref{rem:so-remark}. 
\end{proof}

\begin{remark}\label{rmk:genSetsldR}
Let us comment briefly on how one might verify \eqref{eq:Lie-H-suff-cond}. 
Since $\mathfrak{sl}(\R^n)$ is $n^2 -1$ dimensional, it is clear that one must use commutators that go several generations deep if one has any hope of generating $\mathfrak{sl}(\R^n)$. However, it can
simplify things to instead look to build a suitable generating set for $\mathfrak{sl}(\R^n)$ out of
brackets of $H^i$'s. 
A particularly useful generating set for $\mathfrak{sl}(\R^n)$ is the collection of elementary matrices
\[
	E^{1,2},E^{2,3},\ldots,E^{n,1},
\]
where $E^{i,j}$ is the matrix with $1$ in $(i,j)$ entry and $0$ elsewhere. For these, we have the commutation relation
\[
	[E^{i,j},E^{k,\ell}] = E^{i,\ell}\delta_{j,k} - E^{k,j}\delta_{\ell,i},
\]
so that, e.g., 
\[
	[E^{1,2},E^{2,3}]= E^{1,3} \quad \text{ and } 	[E^{1,2},E^{2,1}] = E^{1,1} - E^{2,2} \, .
\]
Continuing like this allows to generate the off-diagonal matrices $\{E^{i,j}\}_{i\neq j}$ as well as 
the directions $E^{1,1} - E^{2,2},\ldots E^{n,n} - E^{n-1,n-1}$ needed to complete a basis for $\mathfrak{sl}(\R^n)$. Therefore, $\set{E^{1,2},E^{2,3},\ldots,E^{n,1}}$ generates $\mathfrak{sl}(\R^n)$.
\end{remark}

\subsection{Projective spanning for Lorenz 96}\label{sec:L96projSpanning}
Now we turn to verifying the uniform projective spanning for stochastically forced Lorentz 96 \eqref{def:L96intro}.
Recall the stochastic Lorenz 96 is an SDE on $\R^{J}$ defined by 
\begin{equation}\label{eq:L96-second}
\dee u_\ell = (u_{\ell + 1}  - u_{\ell-2}) u_{\ell-1}\dt - \ep u_\ell\dt + \sqrt{\ep}q_\ell \dee W_t^{\ell}. 
\end{equation}
Here, we assume a periodic ensemble of coupled oscillators, i.e., $u_{i+k J} := u_i$.
Naturally we can write \eqref{eq:L96-second} in the general form \eqref{eq:SDEintroFD22} for Euler-like SDE by defining $X_0(u) = B(u,u) + \ep Au$, where the $\ell$-th coordinate of $F(u) := B(u,u)$ is given by
\[
	F_\ell(u) = u_{\ell+1}u_{\ell -1} - u_{\ell-2}u_{\ell-1}, \quad (Au)_\ell = u_\ell \, .
\]

First, we verify uniform hypoellipticity for the process $(u_t)$ on $\R^J$. 
\begin{lemma}
Let $J \geq 2$ and assume $q_1, q_2 \neq 0 $.  Then $\mathrm{Lie}(F; q_1\partial_{u_1},q_2\partial_{u_2})$ contains $\{\partial_{u_j}\}_{j=1}^J$ and spans $\R^J$ uniformly in $\epsilon$ on compact sets.
\end{lemma}
\begin{proof}
Since the nonlinearity is bilinear, we readily observe that
\begin{align*}
[\partial_{u_2},[\partial_{u_1}, F]] & = -\partial_{u_3}. 
\end{align*}
Iterating this observation allows to generate all brackets of the form $[\partial_{u_{i+1}},[\partial_{u_i}, F]] = -\partial_{u_{i+2}}$.
\end{proof}

In order to prove uniform projective spanning we first observe that
\[
	(\nabla F(u))_{\ell m} = DF_\ell(u)_m = u_{\ell-1}\delta_{m=\ell+1} + u_{\ell+1}\delta_{m=\ell} - u_{\ell -1}\delta_{m=\ell-2} - u_{\ell-2}\delta_{m=\ell-1},
\]
hence it follows that for each $k \in \{1,\ldots,J\}$ we have
\[
	H^k = \partial_{u_k} DF(u) = E^{k+1,k+2} + E^{k-1,k-2} - E^{k+1,k-2} - E^{k+2,k+1}.
\]
The following implies projective spanning for Lorenz-96 when combined Corollary \ref{cor:sln-Lie-suff}. 
\begin{lemma}
\[
	\mathrm{Lie}(H^1,\ldots,H^J) = \mathfrak{sl}(\R^J) \,. 
\]
\end{lemma}
\begin{proof}
Throughout, we regard the indices in $E^{i,j}$ modulo $J$, so that $E^{i + k J, j + \ell J} = E^{i,j}$ for all $i,j,k,\ell$. 

Let $\mathfrak g$ denote the smallest Lie algebra containing $\{ H^k\}$. To start, let $1 \leq k \leq J$. We compute
\[
[H^k, H^{k + 4}] = E^{k +3,k+1} \, ,
\]
hence $E^{k, k-2} \in \mathfrak g$ for all $1 \leq k \leq J$. Continuing,
\[
[H^k, E^{k-2, k-4}] = E^{k-1, k-4} \, , 
\]
hence $E^{k, k-3} \in \mathfrak g$ for all $k$. Inductively, assuming $E^{k, k - \ell} \in \mathfrak g$, we have that
\begin{align}\label{eq:inductBracketL96}
[H^k, E^{k-2, k - (\ell + 2)}] = E^{k-1, k-1-(\ell + 1)} \, , 
\end{align}
hence $E^{k, k - (\ell + 1)} \in \mathfrak g$ for all $k$. The induction step in \eqref{eq:inductBracketL96}
 continues to hold as long as $k - (\ell + 2)$ is disjoint from $\{ k-1, k+1, k+2\}$ modulo $J$, which is assured
 so long as $\ell < J - 4$. 
 
 Fix $\ell_0 \in \{ 2,3, \cdots, J-5\}$ so that $J - \ell_0$ is co-prime to $J$. 
  In particular, $\{ 1-\ell_0, 1-2 \ell_0, \cdots, 1-(J-1) \ell_0\}$ coincides
 with the complete set of residue classes $\{0, 1, \cdots, J-1\}$ in $\Z / J \Z$. 
We observe that 
\[
\set{E^{1,1-\ell_0}, E^{1-\ell_0, 1-2 \ell_0}, \cdots, 
 E^{1 - (J-2) \ell_0, 1 - (J-1) \ell_0}, E^{1 - (J-1) \ell_0, 1}} \subset \mathfrak g
 \]
 is a re-ordering of the generating set identified in Remark \ref{rmk:genSetsldR},
and conclude $\mathfrak g = \mathfrak{sl}(\R^J)$. 
\end{proof}

\begin{remark}
We note that the proof presented here for L96 heavily relies on the ``local'' coupling of unknowns in the nonlinearity, which greatly simplifies the application of Corollary \ref{cor:sln-Lie-suff} in this case. However, for models which are the Galerkin truncations of PDEs, coupling between unknowns has a more `global' character, and so even with the simplifications presented here, 
verifying projective spanning takes significantly more work. 
For Galerkin Navier-Stokes, this is the topic of the recent paper \cite{BPS21} of the first and third author. 
\end{remark}

\subsection{Projective irreducibility for Euler-like systems} \label{sec:Irr}

We close this section with a criterion for irreducibility of the projective process $w_t = (x_t, v_t)$ for 
our class of Euler-like systems. In what follows, we will assume that the nonlinearity $B(x,x)$ satisfies the following {\em cancellation property} in addition to volume conservation $\Div B(x,x) = 0$ and $x\cdot B(x,x)=0$:
 
\begin{assumption}\label{ass:bilinear2}
There exists a collection of vectors $\{e_1,\ldots e_s\}\subset \R^n$ with \[
	\Span\{e_1,\ldots e_s\} = \Span\{X_1,\ldots,X_r\}
	\]
	such that for each $1\leq k \leq s$, $B(e_k,e_k) = 0$.
\end{assumption}
Systems for which Assumption \ref{ass:bilinear2} hold include the Lorenz-96 model \eqref{def:L96intro} and Galerkin truncations of the Navier-Stokes equations. Cancellation properties for even degree nonlinearities are commonplace when dealing with controllability and also feature prominently in the works of e.g. \cite{Romito2011-ul,Herzog2015-vj,Glatt-Holtz2018-wx}.

The main result of the section shows that projective spanning combined with the cancellation condition in Assumption \ref{ass:bilinear2} is sufficient to imply the projective process $(w_t)$ is irreducible. This is commonly used as a vital ingredient in obtaining uniqueness of the stationary measure via the Doob-Khasminskii theorem \cite{DPZ96}. 
\begin{proposition}\label{prop:projective irreducibility}
Assume the bilinear term $B$ satisfies Assumption \ref{ass:bilinear2}. 
Then, the projective process $(w_t)$ is irreducible for all $\eps> 0$ if $\{\tilde{X}_0^\ep,\tilde{X}_1,\ldots,\tilde{X}_r\}$ satisfies the parabolic H\"ormander condition $\forall \epsilon > 0$. 
\end{proposition}

\subsubsection{Preliminaries on geometric control theory}

To prove Proposition \ref{prop:projective irreducibility}, we use the Stroock-Varadhan support theorem \cite{Stroock1972-nc} to connect irreducibility to exact controllability (see Theorem \ref{thm:support} below). 
We first outline some basic ideas from geometric control theory, for which we mostly follow \cite{JurdGCT}.
Consider the following affine control system 
\begin{align}
	\dot{x}_t = X_0(x_t) + \sum_{k=1}^r X_k(x_t) \alpha^k_t \label{eq:control-sys}
\end{align}
on a manifold $M$, where $(\alpha^1_t,\ldots,\alpha^r_t)$ are piecewise constant controls taking values in $\R$ and $\{X_0,X_1,\ldots,X_r\}$ are smooth vector fields in $\mathfrak{X}(M)$. In what follows, let $\mathcal{X} = \mathrm{span}\{X_1,\ldots,X_r\}$ and define the controlled distribution of vector fields
\begin{align*}
\mathcal{F}_0 := X_0 + \mathcal{X} = \{X_0+ X\,:\, X\in\mathcal{X}\} \subseteq \mathfrak{X}(M).
\end{align*} 
We also define the time $t>0$ accessible set starting from $x\in M$ associated to $\mathcal{F}_0$ by
\begin{align*}
\mathcal{A}_x^t(\mathcal{F}_0) := \left\{e^{t_nY_n}\ldots e^{t_1Y_1}x\,:\, Y_i\in \mathcal{F}_0,\quad t_i> 0,\quad \sum t_i = t,\quad n\geq 1\right\}\subseteq M,
\end{align*}
where for a vector field $X\in \mathfrak{X}(M)$, $e^{t X}:M\to M$ denotes the flow of diffeomorphisms associated to $X$. 
Note that the set $\mathcal{A}_x^t(\mathcal{F}_0)$ is exactly the set of all points accessible by \eqref{eq:control-sys} at time $t>0$ starting from $x\in M$ using piecewise constant controls. 
We then say that $\mathcal{F}_0$ is {\em exactly controllable} if for every $x$ and $t>0$ we have $\mathcal{A}_x^t(\mathcal{F}_0) = M$.
The following celebrated theorem due to Stroock and Varadhan \cite{Stroock1972-nc} links exact controllability to irreducibility.
\begin{theorem}[Support Theorem; \cite{Stroock1972-nc}]\label{thm:support} Suppose that the SDE
\[
	\dee x_t = X_0(x_t)\,\dt + \sum_{k=1}^r X_k(x_t)\strat \dee W^k_t
\]
has a globally defined flow on a manifold $M$ and that $\mathcal{F}_0 = X_0 + \Span\{X_k\,:\, k=1\ldots,r\}$ is exactly controllable, then for each $x\in M$ and $t>0$, and every open $O\subset M$
\[
	P_t(x,O) = \P(x_t\in O\,|\, x_0 = x) >0.
\]
In other words, exact controllability implies that the process $(x_t)$ is topologically irreducible.
\end{theorem}
Due to the rigidness of the time constraint $\sum_{k}t_i = t$, it is convenient to define the relaxed accessibility set of states reached by $\mathcal{F}_0$ before time $t$
\[
\mathcal{A}_x^{\leq t}(\mathcal{F}_0) := \bigcup_{s\leq t} \mathcal{A}_x^s(\mathcal{F}_0).
\]
We say that $\mathcal{F}_0$ is {\em strongly controllable} if for each $x\in M$ and $t>0$, $\mathcal{A}_x^{\leq t}(\mathcal{F}_0) = M$.

In order to go from strong to exact controllability, an important role is played by the {\em zero-time ideal} $\mathcal{I}(\mathcal{F})$ of $\mathrm{Lie}(\mathcal{F})$, defined for any collection of vector fields $\mathcal{F}\subseteq \mathfrak{X}(M)$ by
\[
\mathcal{I}(\mathcal{F}) := \mathrm{span}\{X-Y\,:\, X, Y \in \mathcal{F}\} + \mathcal{D}(\mathcal{F}),
\]
where $\mathcal{D}(\mathcal{F})$ is the {\em derived algebra} ideal generated by $\mathcal{F}$, defined as the algebra of iterated brackets at least one bracket deep
\[
\mathcal{D}(\mathcal{F}) := \mathrm{span}\{\mathrm{ad}(Y_r)\ldots\mathrm{ad}(Y_2)Y_1\,:\, Y_i\in \mathcal{F},\, r\geq 2 \}.
\]
One can check that $\mathcal{I}(\mathcal{F})$ is a Lie algebra ideal of $\mathrm{Lie}(\mathcal{F})$ in the sense that if $X\in \mathrm{Lie}(\mathcal{F})$ and $Y\in \mathcal{I}(\mathcal{F})$ then $[X,Y]\in \mathcal{I}(\mathcal{F})$. The follow results connects the zero-time ideal to exact controllability.

\begin{theorem}[Theorem 13 Chapter 3 \cite{JurdGCT}] \label{thm:strong-control}
	Suppose that $\mathcal{F}$ is strongly controllable, then it is exactly controllable if for each $x\in M$ 
	\[
	\mathcal{I}_x(\mathcal{F}) := \{X(x)\,:\, X\in \mathcal{I}(\mathcal{F})\} = T_xM.
	\]
\end{theorem}

For the controlled distribution of vector fields $\mathcal{F}_0 = X_0 +\mathcal{X}$, we have the following useful characterization of $\mathcal{I}(\mathcal{F}_0)$ that intricately links exact controllability to the parabolic H\"ormander condition. It is certainly known among experts in control theory, but we could not find a proof in the literature.
\begin{proposition}
	If $\mathcal{F}_0 = X_0 +\mathcal{X}$, where $\mathcal{X} = \mathrm{span}\{X_1,\ldots,X_r\}$, we have
	\[
	\mathcal{I}(\mathcal{F}_0) = \mathrm{Lie}(\mathcal{X},[\mathcal{X},X_0]).
	\]
\end{proposition}
\begin{proof}
	First we note that
	\[
	\mathrm{span}\{X-Y\,:\, X, Y \in \mathcal{F}_0\} = \mathrm{span}\{X-Y\,:\, X,Y\in\mathcal{X}\} = \mathcal{X}.
	\]
	so that $\mathcal{I}(\mathcal{F}_0) = \mathcal{X} + \mathcal{D}(\mathcal{F})$. Next, we note that we can describe $\mathrm{Lie}(\mathcal{X},[\mathcal{X},X_0])$ as
	\[
	\begin{aligned}
		\mathrm{Lie}(\mathcal{X},[\mathcal{X},X_0]) &= \mathrm{span}\{\mathrm{ad}(Y_r)\ldots\mathrm{ad}(Y_1)X\,:\, Y_i\in\mathrm{span}\,\mathcal{F}_0,\, X\in \mathcal{X},\, r\in \N\}\\
		&=\mathcal{X} + \mathrm{span}\{\mathrm{ad}(Y_r)\ldots\mathrm{ad}(Y_1)X\,:\, Y_i\in\mathrm{span}\,\mathcal{F}_0,\, X\in \mathcal{X},\, r\geq 1\},
	\end{aligned}
	\]
	since for any $X \in \mathcal{X}$ we can apply the Jacobi identity
	\[
	\mathrm{ad}([X_0,X]) = \mathrm{ad}(X_0)\mathrm{ad}(X) - \mathrm{ad}(X)\mathrm{ad}(X_0).
	\]
	The proof is complete upon realizing that $\mathcal{D}(\mathcal{F}_0)$ can also be described by
	\[
	\mathcal{D}(\mathcal{F}_0)= \mathrm{span}\{\mathrm{ad}(Y_r)\ldots\mathrm{ad}(Y_1)X\,:\, Y_i\in\mathrm{span}\,\mathcal{F}_0,\, X\in \mathcal{X},\, r\geq 1\}
	\]
	which follows from the fact that allowing $X$ to take values in $\mathrm{span}\,\mathcal{F}_0 = \mathrm{span}\{X_0,X_1,\ldots,X_r\}$ instead of $\mathcal{X}$ doesn't change the span above since $\mathrm{ad}(X_0)X_0= 0$ and
$\mathrm{ad}(X_i)X_0 = - \mathrm{ad}(X_0)X_i$.
\end{proof}

In order to prove strong controllability, we will use the notion of a Lie saturate originally introduced by Jurdjevic and Kupka \cite{JK81}. Following \cite{JurdGCT}, we introduce an equivalence class on subsets of $\mathfrak{X}(M)$ induced by the mapping $\mathcal{F} \mapsto \mathrm{cl}(\mathcal{A}^{\leq t}_x(\mathcal{F}))$, where $\mathrm{cl}(A)$ denotes the closure of a set $A\subseteq M$. Let $\mathcal{F},\mathcal{G}\subseteq \mathfrak{X}(M)
$, we say that $\mathcal{F}$ and $\mathcal{G}$ are equivalent, denoted by $\mathcal{F}\sim \mathcal{G}$, if for each $x\in M$ and $t>0$
\[
\mathrm{cl}(\mathcal{A}_x^{\leq t}(\mathcal{F}))= \mathrm{cl}(\mathcal{A}_x^{\leq t}(\mathcal{G})).
\]
It is easy to see that if $\mathcal{G} \sim \mathcal{F}$ then $\mathcal{F} \sim \mathcal{F}\cup \mathcal{G}$. This naturally motivates the definition of the {\em Lie saturate} of a family $\mathcal{F}$ defined by
\[
\mathrm{LS}(\mathcal{F}) := \bigcup_{\mathcal{G}\sim \mathcal{F}} \mathcal{G}\cap \mathrm{Lie}(\mathcal{F}).
\]
Note that we always have $\mathcal{F}\sim \mathrm{LS}(\mathcal{F})$. 

We have the following very important theorem relating spanning properties of the Lie saturate to strong controllability.

\begin{theorem}[Theorem 12 Ch 3 \cite{JurdGCT}]
	$\mathcal{F}_0= X_0 + \mathcal{X}$ is strongly controllable if for every $x$
	\[
	\mathrm{LS}_x(\mathcal{F}_0) := \{X(x)\,:\, X\in \mathrm{LS}(\mathcal{F}_0)\} = T_xM.
	\]
\end{theorem}

An important corollary of this is the following sufficient condition for exactly controllability, which will prove useful for what follows.

\begin{corollary}\label{cor:sufficient-control}
	Suppose that $X_0,\mathcal{X}$ satisfy the parabolic H\"ormander condition. Then $\mathcal{F}_0 = X_0 + \mathcal{X}$ is exactly controllable if
	\[
	\mathrm{Lie}(\mathcal{X},[\mathcal{X},X_0]) \subseteq \mathrm{LS}(\mathcal{F}_0).
	\]
\end{corollary}

In light of this theorem, our goal will usually be to show that $\mathrm{LS}(\mathcal{F}_0)$ contains $\mathrm{Lie}(\mathcal{X},[\mathcal{X},X_0])$. Typically this is done by repeatedly ``enlarging'' the initial set $\mathcal{F}_0$ by vector fields that are equivalent to $\mathcal{F}_0$.
The Lie saturate has the following very useful symmetry properties
\begin{proposition}[Proposition 2.24 \cite{Jurdjevic2016-cr}]\label{prop:saturate-growth} The Lie saturate $\mathrm{LS}(\mathcal{F})$ is convex and closed in the $C^\infty(M)$ topology and is invariant under the following enlargements:
	\begin{enumerate}
		\item If $\mathcal{F}\subseteq \mathrm{LS}(\mathcal{F})$, then the convex hull $\mathrm{co}(\mathcal{F})$ also belongs to $\mathrm{LS}(\mathcal{F})$.
		\item If $\mathcal{V}\subseteq \mathrm{LS}(\mathcal{F})$ is a vector space then $\mathrm{Lie}(\mathcal{V})$ also belongs to $\mathrm{LS}(\mathcal{F})$.
		\item If $\pm X$ belongs to $\mathrm{LS}(\mathcal{F})$, then the pushforward \footnote{the push forward $\phi_\sharp Y$ of a vector field $Y$ by a diffeomorphism $\phi$ is defined by $(\phi_\sharp Y)(y) := d\phi_x(Y(x))$, where $x = \phi^{-1}(y)$.} $(e^{\alpha X})_\sharp\mathcal{F}$ also belongs to $\mathrm{LS}(\mathcal{F})$ for each $\alpha \in \R$.
	\end{enumerate}
	
\end{proposition}

\begin{remark}
	One can think of the pushforward $(e^{\alpha X})_\sharp Y$ as an infinite dimensional version of the adjoint action $e^{\alpha\,\mathrm{ad}(X)}$ by the group element $e^{\alpha X}\in \mathrm{Diff}(M)$ on $\mathfrak{X}(M)$ in the sense that $(e^{\alpha X})_\sharp Y$ satisfies the following Taylor expansion
	\[
	(e^{\alpha X})_\sharp Y = \sum_{k=0}^\infty \frac{\mathrm{ad}(\alpha X)^k Y}{k!},
	\]
	which expresses $(e^{\alpha X})_\sharp$ as a Lie polynomial in $\mathrm{ad}(\alpha X)$.
\end{remark}

\subsubsection{Proof of Proposition \ref{prop:projective irreducibility}}

We now want to apply this machinery to the bilinear system \eqref{eq:SDEintroFD22} to prove Proposition \ref{prop:projective irreducibility}. In what follows, let
\[
	X_0(x) := B(x,x) + \ep Ax
\]
and $\mathcal{X} := \Span\{X_k\,:\, k=1\ldots,r\}$. In light of the assumption of the parabolic H\"ormander condition, Corollary \ref{cor:sufficient-control}, and Theorem \ref{thm:support}, all we need to do is show that $\mathrm{Lie}(\widetilde{\mathcal{X}},[\widetilde{\mathcal{X}},\widetilde{X}_0]) \subseteq \mathrm{LS}(\widetilde{\mathcal{F}}_0)$.

\begin{proof}
	Lemma \ref{lem:lifting-isomorph} shows that $X\to \widetilde{X}$ is a Lie algebra isomorphism on to $\widetilde{\mathfrak{X}}(M) = \{\widetilde{X}\,:\, X\in\mathfrak{X}(M) \}$, and so $[\tilde{X},\tilde{Y}]$ is the same as the lift of $[X,Y]$.  
	
	By property 1 of Proposition \ref{prop:saturate-growth}, for each $\widetilde{X}\in \widetilde{\mathcal{X}}$ and $\alpha\in [0,1]$, we have $\alpha(\widetilde{X}_0 + \alpha^{-1} \widetilde{X}) = \alpha \widetilde{X}_0 + \widetilde{X} \in \mathrm{LS}(\widetilde{\mathcal{F}}_0)$. By closedness of $\mathrm{LS}(\widetilde{\mathcal{F}}_0)$, it follows that for each $\widetilde{X}\in\widetilde{\mathcal{X}}$
	\[
	\widetilde{X} = \lim_{\alpha \to 0}(\alpha \widetilde{X}_0 + \widetilde{X}) \in \mathrm{LS}(\widetilde{\mathcal{F}}_0).
	\]
	We conclude that $\widetilde{\mathcal{X}} \in \mathrm{LS}(\widetilde{\mathcal{F}}_0)$. 
	By the spanning of the $\set{e_k}_{k=1}^s$ in Assumption \ref{ass:bilinear2}, it then follows that for all $k=1\ldots,s$, the constant vector fields $\pm\widetilde{e}_k\in \mathrm{LS}(\widetilde{\mathcal{F}}_0)$. 
	Then, from property $3$ of Proposition \ref{prop:saturate-growth}, and the fact $\widetilde{X}_0$ is a polynomial in $x$ of degree two, $\forall \alpha \in\R$ and $k\in\mathcal{K}$, the following vector field is also in $\mathrm{LS}(\widetilde{\mathcal{F}}_0)$ 
	\[
	(e^{\alpha\widetilde{e}_k})_\sharp \widetilde{X}_0 = \widetilde{X}_0 + \alpha [\widetilde{e}_k,\widetilde{X}_0] + \frac{\alpha^2}{2}[\widetilde{e}_k,[\widetilde{e}_k,\widetilde{X}_0]].
	\]
	Using the Lie algebra isomorphism (Lemma \ref{lem:lifting-isomorph})  $[\widetilde{e}_k,[\widetilde{e}_k,\widetilde{X}_0]] =[e_k,[e_k,X_0]]\,\,\widetilde{}$ and cancellation condition of Assumption \ref{ass:bilinear2} that $[e_k,[e_k,X_0]] = B(e_k,e_k) = 0$, we deduce that $[\widetilde{e}_k,[\widetilde{e}_k,\widetilde{X}_0]]=0$. 
	Therefore for each $\lambda \in \R$
	\[
	\lim_{\alpha \to \infty} \frac{1}{\alpha} (e^{\lambda\alpha\widetilde{e}_k})_\sharp\widetilde{X}_0 = \lambda[\widetilde{e}_k,\widetilde{X}_0] \in \mathrm{LS}(\widetilde{\mathcal{F}_0}).
	\]
	It follows by property 1 of Proposition \ref{prop:saturate-growth}, that $\mathrm{span}\{\widetilde{\mathcal{X}},[\widetilde{\mathcal{X}},\widetilde{X}_0]\}$ is contained in $\mathrm{LS}(\widetilde{\mathcal{F}}_0)$ and so by property 2 of Proposition \ref{prop:saturate-growth}
	\[
	\mathrm{Lie}(\widetilde{\mathcal{X}},[\widetilde{\mathcal{X}},\widetilde{X}_0]) \subseteq \mathrm{LS}(\widetilde{\mathcal{F}}_0).
	\]
	Exact controllability now follows from Corollary \ref{cor:sufficient-control} and irreducibility follows from Theorem \ref{thm:support}, completing the proof of Proposition \ref{prop:projective irreducibility}. 
\end{proof}

\begin{remark}
The spanning property $\mathfrak{m}_x(\mathcal{F}) = \mathfrak{sl}(\R^n)$ actually implies a much stronger form of controllability for the linearized process on $SL(\R^d)$.
Indeed, let  
\begin{align*}
\hat{J}_t = \frac{D_x \Phi^t}{ \mathrm{det} \left(D_x \Phi^t\right)^{1/n}}\in SL(\R^n), 
\end{align*}
then the normalized Jacobian Markov process is given by $Z_t = (x_t,\hat{J}_t) \in \R^n \times SL(\R^n)$ and solves an SDE system 
\begin{align}
\dee Z_t = \hat{X}_0(Z_t)\dt + \sum_{k=1}^r \hat{X}_k(Z_t) \circ \dee W_t^{(k)}, \label{eq:MXproc} 
\end{align}
defined by the lifts
\begin{align*}
\hat{X}(x,\hat{J}) = (X(x),M_X(x)\hat{J}),
\end{align*}
where $M_X(x)\in \mathfrak{sl}(\R^n)$ is defined above in \eqref{def:MX}. 
The assumption of Corollary \ref{cor:sln-Lie-suff} then implies that this lifted process is hypoelliptic on $\R^n \times SL(\R^n)$. 
The proof of Proposition \ref{prop:projective irreducibility} extends to this process analogously as to the projective process (replacing vector fields $\tilde{X}$ with $\hat{X}$), hence yielding exact controllability of the associated affine control problem as well as irreducibility of the SDE system \eqref{eq:MXproc} through Theorem \ref{thm:support}.  
\end{remark}

\section{Rigidity of invariant measures of the deterministic, projective process} \label{sec:Inviscid-Limit}

\newcommand{\pd}{\partial}
\renewcommand{\a}{\alpha}
\newcommand{\dist}{\operatorname{dist}}
\newcommand{\Gc}{{\mathcal G}}

In this section, we will complete the proof of Theorem \ref{thm:critForEulerLikeIntro} on the Lyapunov exponents of 
Euler-like systems and Corollary \ref{L96} for the L96 system. 
Throughout, $\Phi^t_\omega$ is the stochastic flow of diffeomorphisms
corresponding to an Euler-like SDE as in standing assumptions at the beginning of Section \ref{ProjBraks}. 
Lastly, we assume that the SDE satisfies the H\"ormander bracket spanning condition uniformly in $\epsilon \in (0,1)$, and that the corresponding projective process $w_t^\epsilon = (x_t^\epsilon, v_t^\epsilon)$ admits a unique stationary density $f^\epsilon$ on $\S \R^n$ (see Section \ref{sec:Irr} for sufficient conditions).  

The plan for this section is as follows. In Section \ref{subsec:dich6} we carry out the main argument in the 
proof of Theorem \ref{thm:critForEulerLikeIntro} and Corollary \ref{L96}, which we reduce to checking Proposition \ref{prop:mainRigidityIntro} below. 
This proposition is checked in Sections \ref{sec:prelimEuler5}, \ref{sec:sec51}, and \ref{sec:53}.

\subsection{Dichotomy: rigidity or positive Lyapunov exponents}\label{subsec:dich6}

Assumptions \ref{ass:WP} and \ref{ass:WP2}, as well the moment estimates needed 
for the application of Theorem A (see Proposition \ref{prop:FI-gen-intro}), are checked for Euler-like systems
in Theorem \ref{thm:BasicProjDensity} in the Appendix. 
Theorem \ref{thm:non-random-MET} ensures the existence of the top Lyapunov exponent $\lambda_1$
and the sum Lyapunov exponent $\lambda_\Sigma$. 
Applying Theorem \ref{thm:fishInfoIntro}, the Fisher information identity \eqref{eq:gen-FI-ID} reads as follows: 
\begin{align}\label{eq:fisherInfoIdentity11}
FI(f^\epsilon)= \frac{n \lambda_1^\ep}{\ep} - 2\tr A 
\end{align}
This is immediate from Proposition \ref{prop:FI-gen-intro} on noting that $\lambda_\Sigma = \epsilon \tr A$ by Theorem \ref{thm:non-random-MET}. 
Theorem \ref{thm:unifHypoRegEstIntro} implies 
\begin{align}\label{hypoforbilinear}
\| f^\epsilon \|_{W^{s ,1}(U \times \S^{n-1})}^2 \leq C \left(1 + FI(f^\epsilon) \right),
\end{align}
for any $U \subset \R^n$ bounded, where $s \in (0,1)$ and $C = C_U$ are constants independent of $\epsilon$. 

In view of the form of \eqref{eq:fisherInfoIdentity11} and \eqref{hypoforbilinear} we see that if $\epsilon^{-1}\lambda_1^\epsilon$ were to remain bounded as $\epsilon \to 0$, then $f^{\epsilon}$ would be bounded in $W^{s,1}$ uniformly in $\epsilon$. As $W^{s, 1}$ is locally compactly embedded in $L^1$ (Lemma \ref{lem:Ws1Precompact}), 
this observation leads naturally to the following alternative. 
\begin{proposition}\label{cor:dichotomyTn}
At least one of the following holds:
\begin{itemize}
    \item[(a)] $\lim_{\epsilon \to 0} \frac{\lambda_1^\epsilon}{\epsilon} = \infty$; or
    \item[(b)] the zero-noise flow $(x_t^0, v_t^0)$ admits a stationary density $f^0 \in L^1(\R^n \times \S^{n-1})$ (and moreover $f^0 \in W^{s,1}_{loc}$).
\end{itemize}
\end{proposition} 
\begin{proof} 
Suppose that (a) fails, i.e. 
\[
\liminf_{\epsilon \to 0} \frac{\lambda^\epsilon_1}{\epsilon} < \infty. 
\]
In this case, \eqref{eq:fisherInfoIdentity11} implies that $\liminf_{\epsilon \to 0} FI(f^\epsilon) < \infty$ and the hypoelliptic regularity estimate \eqref{hypoforbilinear} implies $\liminf_{\epsilon \to 0} \norm{f^\epsilon}_{W^{s,1}(U)} < \infty$ for all bounded geodesic balls $U$. Combined with the uniform tightness of $\set{f^\epsilon}_{\epsilon > 0}$ in \eqref{ineq:tightness} and the compactness criterion Lemma \ref{lem:Ws1Precompact}, this yields precompactness in $L^1$ of $\set{f^\epsilon}_{\epsilon \in (0,1)}$ and so after extracting a subsequence $\set{\epsilon_j}$, we see that $\exists  f^0 \in L^1\cap W^{s,1}_{\loc}$ such that $f^{\epsilon_j} \to f$ in $L^1$.

Let us now briefly check that $f^0$ is an invariant density for the zero-noise flow $(x_t^0, v_t^0)$. 
For this, let $\tilde {\mathcal L}^\eps$ denote the infinitesimal generator for $w_t^\eps = (x_t^\eps, v_t^\eps)$ and let $\phi \in C^\infty_c(\S \R^n)$ be a compactly supported test function. Starting with the Kolmogorov equation and pairing with $\phi$ gives $\int (\tilde {\mathcal L}^\eps\phi) f^\epsilon \dq = 0$ 
for all $\epsilon > 0$, while taking $\eps = \epsilon_j \to 0$ yields
\[
\int (\tilde{\mathcal L}^0\phi) f^0 \dq = \int \big( \tilde X_0 \phi \big) f^0 \dq = 0
\]
using that $f^{\epsilon_j} \to f^\eps$ in $L^1$ while $\tilde {\mathcal L}^\eps \phi \to \tilde {\mathcal L}^0 \phi$ in $L^\infty$. This last equality holds for all smooth compactly supported $\phi$, and so we conclude 
$f^0$ is an invariant density for $w^0_t = (x^0_t,v^0_t)$.
\end{proof} 

In our setting, alternative (b) is ruled out by the following proposition, proved in the rest of Section \ref{sec:Inviscid-Limit}.

\begin{proposition}\label{prop:mainRigidityIntro}
 Assume that the bilinear mapping $B$ is not identically 0. Let $\nu$ be any invariant probability measure for $\hat \Phi^t$ with the property that $\nu(A \times \mathbb S^{n-1}) = \mu(A)$, where $\mu \ll \Leb_{\R^n}$. Then, $\nu$ is singular with respect to Lebesgue measure $\Leb_{\S \R^n}$ on $\S \R^n$. 
\end{proposition}

Finally, we give the proof of Corollary \ref{L96}. 
\subsubsection*{Completing the proof of Corollary \ref{L96}}
\begin{proof} 
To prove Corollary \ref{L96} we need only to verify the hypotheses of Theorem \ref{thm:critForEulerLikeIntro} for the L96 model \eqref{def:L96intro}.
The projective hypoellipticity condition, Assumption \ref{ass:PHC}, is verified in Section \ref{sec:L96projSpanning}. 
The cancellation Assumption \ref{ass:bilinear2} is immediate using the canonical coordinate basis vectors of $\R^J$. 
Hence, the Assumptions \ref{ass:WP} and \ref{ass:WP2} follow from Theorem \ref{thm:BasicProjDensity} in Appendix \ref{sec:Basics}. 
Therefore, all of the requirements to apply Theorem \ref{thm:critForEulerLikeIntro} are satisfied and the corollary follows. 
\end{proof}

It remains to check Proposition \ref{prop:mainRigidityIntro}, the proof of which occupies the remainder 
of this section. 

\subsection{Shearing between energy shells}\label{sec:prelimEuler5}

In this subsection we begin the proof of Proposition \ref{prop:mainRigidityIntro}. 
Here we show that the $\eps = 0$ limit of \eqref{eq:SDEintroFD22}, given by 
\begin{align}
\dot{x}_t = B(x_t, x_t),  \label{eq:Euler-like-ODE} 
\end{align}
must necessarily have some infinite-time growth in the gradient of the flow map. 
Some notation: let $\Phi^t$ be the flow map for the (deterministic) ODE \eqref{eq:Euler-like-ODE}, for $E > 0$ let us write $S_E := \{ x \in \R^n : |x|^2 = E\}$ for the ``energy shells'', preserved by the flow $\Phi^t$, i.e., $\Phi^t(S_E) = S_E$ for all $t \geq 0, E > 0$. 
Write $E(x) = |x|^2$.

\begin{lemma}\label{lem:preRig182} Let $x \in \R^n$ then the following identity holds
\begin{align}\label{eq:Dphi-id}
D_x\Phi^tx = \Phi^t(x) + t B(\Phi^t(x), \Phi^t(x)).
\end{align}
Moreover, for each $x \in \R^n$ and $t \geq 0$, we have that
\begin{align}\label{eq:normGrowth}
| D_x\Phi^t| \geq t \frac{|B(\Phi^t( x), \Phi^t( x))|}{|x|} \, .
\end{align}
\end{lemma}
\begin{proof}[Proof of Lemma]
  For a given $\alpha >0$, note that the rescaled flow $\alpha\Phi^{\alpha t}(x)$ also solves \eqref{eq:Euler-like-ODE} with initial data $\alpha x$. Therefore by uniqueness, we have
\begin{equation}\label{eq:phi_t-rescale}
	\Phi^t(\alpha x) = \alpha \Phi^{\alpha t}(x)
\end{equation}
Taking the derivative with respect to $\alpha$ on both sides of \eqref{eq:phi_t-rescale} yields
\[
D_{\alpha x}\Phi^t x = \Phi^{\alpha t}(x) + \alpha t B(\Phi^{\alpha t}(x),\Phi^{\alpha t}(x)).
\]
Setting $\alpha =1$ gives \eqref{eq:Dphi-id}. Inequality \eqref{eq:normGrowth} follows from part \eqref{eq:Dphi-id} and the fact that 
$\Phi^t(x)\cdot B(\Phi^t(x),\Phi^t(x)) = 0$ for all $x$, by assumption. 
\end{proof}

\subsection{Rigidity for invariant projective densities: review of general theory}\label{sec:sec51}

We recall here an abstract result (Theorem \ref{thm:rigidSec5}) indicating that the presence of an
 invariant projective density as in Proposition \ref{cor:dichotomyTn} implies rigid properties of the corresponding flow. 
 
 In this section, we will state everything in the following \emph{abstract linear cocycle} setting. 
Throughout, $T : (X, \mathcal B, m) \circlearrowleft$ is a (discrete-time) continuous
 transformation of a compact metric space $X$, with $\mathcal B$ the Borel $\sigma$-algebra. 
Let $A : (X, \mathcal B) \to SL_n(\R), x \mapsto A_x$ be a measurable mapping\footnote{Here, 
$SL_n(\R)$ is the group of $d \times d$ real matrices of determinant 1.}.
This generates the \emph{cocyle} of linear operators
 $\mathcal A: X \times \Z_{\geq 0} \to SL_n(\R)$ defined by
\[
\mathcal A(n, x) = \mathcal A^n_x := A_{T^{n-1} x} A_{T^{n-2} x} \cdots A_{T x} A_x \, .
\]
Note that $\mathcal A$ satisfies the \emph{cocycle identity}\footnote{When $T$ is a smooth mapping of a manifold and $\mathcal A_x^n := D_x (T^n)$ is the so-called \emph{derivative cocycle}, the cocycle identity 
is merely the chain rule for $T^n$.} $\mathcal A^{m + n}_x = 
\mathcal A^m_{T^n x} \mathcal A^n_x$ for all $m, n \geq 0, x \in X$. Associated to $T, \mathcal A$ is the \emph{projective action} 
$\hat T : X \times \mathbb S^{n-1} \circlearrowleft$ defined by
\[
(x, v) \mapsto \left(T x, \frac{A_x v}{| A_x v|}\right) \, , \quad x \in X, v \in \mathbb S^{n-1} \, ,
\]
which we regard as a dynamical system on $ X \times \mathbb S^{n-1}$ in its own right. 

Let $\hat m$ be any $\hat T$-invariant measure on $X \times \S^{n-1}$ projecting to $m$
(i.e., $\hat m(K \times \mathbb S^{n-1}) = m(K)$ for all measurable $K \subset X$), and
consider its disintegration
\[
\dee \hat m(u,  v) = \dee \hat m_x( v) \dee m(x) \, .
\]
In this context, it is well-known \cite{rokhlin1949fundamental, chang1997conditioning} that disintegrations $(\hat m_x)_{x \in X}$ exist
and are essentially unique (up to $m$-measure zero modifications) and $x\mapsto \hat{m}_x$ is weak-* measurably varying. Note that invariance
of $\hat m$ implies that
\[
	(A_x)_* \hat m_x = \hat m_{T x}\quad \text{for}\quad m\text{-a.e. } x \in X,
\]
where for a $d \times d$ matrix $A$ we write $A_*$ for the action of $A$ on probability measures
on $\mathbb S^{n-1}$. 

The following result (more-or-less Theorem 3.23 of \cite{arnold1999jordan}, up to a technical issue-see below) involves the rigidity of 
absolute continuity of the disintegration measures $\hat m_x$ with respect to $\Leb_{\mathbb S^{n-1}}$. 

\begin{theorem}\label{thm:rigidSec5}
Assume that $\hat m_x \ll \Leb_{\mathbb S^{n-1}}$ for $m$-almost every $x \in X$. Then, 
there exists a measurable family of inner products $X \ni x \mapsto g_x(\cdot, \cdot)$ on 
$\R^n$ and a $T$-invariant set $\Gamma \subset X$ of full $m$-measure such that for all 
$x \in \Gamma$ and $v, w \in \R^n$, we have that
\[
g_{T x}( A_x v, A_x w) = g_x(v, w) \, .
\]
That is, $A_x : (\R^n, g_x) \to (\R^n, g_{T x})$ is an isometry. 
\end{theorem}

This is slightly different from the form in Theorem 3.23 of \cite{arnold1999jordan}: 
there, it is supposed that $\hat m_x \sim \Leb_{\mathbb S^{n-1}}$, whereas for our purposes we need
the version with ``$\ll$''. For this reason, as well as for the sake of completeness, we sketch the
proof of Theorem \ref{thm:rigidSec5} here. 

\begin{proof}[Proof sketch]
To start, let us assume for now that 
$T : (X, \mathcal B, m) \circlearrowleft$ is ergodic (note that we do not assume $\hat m$ is ergodic). 
We require the following Lemma: 
\begin{lemma}[Corollary 3.7 in \cite{arnold1999jordan}; Lemma 6.2 in \cite{furstenberg1981rigidity}]\label{lem:measSelectionSec5}
Assume $(X, \mathcal B, m, T)$ is ergodic. Then, there is a full $m$-measure set of $x_0 \in X$ with the following property: there exists a measurable
mapping $G : X \to SL_n(\R)$, depending on the choice of $x_0$, such that 
\[
G(x)_* \hat m_{x_0} = \hat m_x \, \quad \text{ for } m-\text{almost every } x \in X \, .
\]
\end{lemma}
This version is slightly different from those appearing in \cite{arnold1999jordan, furstenberg1981rigidity}, so we briefly recall the proof below. 
\begin{proof}[Proof sketch of Lemma]
Let $\mathcal P(\mathbb S^{n-1})$ denote the space of probability measures on $\mathbb S^{n-1}$ with the 
weak$^*$ topology. Consider the quotient $\mathcal P(\mathbb S^{n-1}) / SL_n(\R)$, i.e.,
for $\xi, \eta \in \mathcal P(\mathbb S^{n-1})$ we set $\xi \sim \eta$ iff $\exists B \in SL_n(\R)$ so that
$B_* \xi = \eta$. 
Writing $[\eta]$ for the equivalence class of $\eta \in \mathcal P(\mathbb S^{n-1})$, note that 
$[\hat m_x] = [\hat m_{T^k x}]$ for all $k$, i.e., $x \mapsto [\hat m_x]$ is constant
along orbits. 
By Corollary 3.2.12 in \cite{zimmer2013ergodic}, the Borel $\sigma$-algebra on the quotient
space $\mathcal P(\mathbb S^{n-1}) / SL_n(\R)$ is countably generated. Using this 
along with the fact that $T : (X, m) \circlearrowleft$ is ergodic, it follows from Proposition 2.1.11 in \cite{zimmer2013ergodic} that $[\hat m_x]$ is constant $m$-almost surely. In particular, 
for $m$-a.e. $x_0, x \in X$, the measures 
$\hat m_x$ and $\hat m_{x_0}$ are related by the application of a matrix in $SL_n(\R)$. 
It is now straightforward to construct a measurable selection $G: X \to SL_n(\R)$ as above. 
\end{proof}

Fix $x_0$ so that $\hat m_{x_0} \ll \Leb_{\mathbb S^{n-1}}$ and let $G$ be as in 
Lemma \ref{lem:measSelectionSec5}. Observe that for any $n \geq 0$ and $m$-a.e. $x \in X$
we have that $G(T^n x)^{-1} \mathcal A^n_x G(x) \in H_{x_0}$, where
\[
H_{x_0} := \{ B \in SL_n(\R) : B_* \hat m_{x_0} = \hat m_{x_0}\} \, .
\]
Observe that $H_{x_0}$ is a subgroup of $SL_n(\R)$, which we claim to be compact. If not, then a lemma of Furstenberg (see, e.g., Claim 4.8 in \cite{bedrossian2018lagrangian}) would 
imply the existence of proper subspaces $V^1, V^2 \subset \R^n$ and a sequence $\{ B_n\} \subset H_{x_0}$ so that $dist(B_n v, V^2) \to 0$ for all $v \notin V^1$, which would contradict
$\hat m_{x_0} \ll \Leb_{\mathbb S^{n-1}}$. 

Since $H_{x_0}$ is compact, there exists an inner product $\langle \cdot, \cdot \rangle$
on $\R^{d}$ with respect to which all members of $H_{x_0}$ are isometries (Lemma 4.6 in \cite{bedrossian2018lagrangian}). The proof is complete on defining $g_x$ through
\begin{align}\label{eq:defineMeasRM}
g_x(v, w) = \langle G(x)^{-1} v, G(x)^{-1} w \rangle \, . 
\end{align}

To handle the case when $m$ is not ergodic, we use the ergodic decomposition \cite{walters2000introduction} 
\[
m = \int_{\mathcal E_T(X)} \xi \, \dee \tau_m(\xi) \, , 
\]
where $\mathcal E_T(X)$ is the space of $T$-ergodic measures on $X$ and $\tau_m$ 
a Borel probability measure (w.r.t. the weak$^*$ topology) on $\mathcal E_T(X)$. For each component
$\xi$, we define $\hat \xi$ through the formula
\[
\dee \hat \xi(x, v) = \dee \hat m_x(v) \dee \xi(x) \, ,
\]
noting that $\hat m_x \ll \Leb_{\S^{n-1}}$ for $\xi$-a.e. $x \in X$ and $\tau_m$-a.e. $\xi \in \mathcal E_T(X)$. The proof now goes through the same as before, the only difference being that 
the measurable inner product \eqref{eq:defineMeasRM} is defined along each
$\xi \in \mathcal E_T(X)$ one at a time. 
\end{proof}

\subsection{Proof of Proposition \ref{prop:mainRigidityIntro}}\label{sec:53}

To start, let $\nu$ be $\hat \Phi^t$-invariant, projecting to an absolutely continuous measure $\mu$
on $\R^n$, and assume that
\[
\nu = \nu^{ac}+ \nu^\perp
\]
where $\nu^{ac} \ll \Leb_{\S \R^n}$ is not identically zero (our contradiction hypothesis), 
while $\nu^\perp$ is singular. Since $\hat \Phi^t$ sends
absolutely continuous measures to absolutely continuous measures and singular to singular, 
it follows that $\nu^{ac}$ is $\hat \Phi^t$-invariant. Since 
$\nu^{ac} \leq \nu$, the measure $\mu^{ac} (K) := \nu^{ac}(K \times \mathbb S^{n-1})$ satisfies 
$\mu^{ac} \ll \mu \ll \Leb_{\R^n}$ and is likewise $\Phi^t$-invariant. 
On replacing $\nu$ with the normalization of $\nu^{ac}$, going forward we may assume
without loss that $\nu \ll \Leb_{\S \R^n}$. Finally, since the energy shells $S_E = \{ |x|^2 = E\}$
are invariant, we may replace $\nu$ with the normalization of its restriction to 
$B(0,R) \times \S^{n-1}$ for some large, fixed $R > 0$. 

Continuing, let $\dee \nu(x, v) = \dee \nu_x(v) \dee \mu(x)$ denote the disintegration measures of $\nu$
and note that $\nu_x \ll \Leb_{\S^{n-1}}$ for $\mu$-a.e. $x$. 
By Theorem \ref{thm:rigidSec5}, there exists a 
measurable family of inner products $g_x, x \in \R^n$ so that
\begin{align}\label{eq:isometrySec5Proof}
D_x\Phi^1 : (\R^n, g_x) \to (\R^n, g_{\Phi^1 x})
\end{align}
is an isometry for $\mu$-a.e. $x$. By a standard procedure, 
we may assume
that \eqref{eq:isometrySec5Proof} holds for $x \in \Gamma$, where $\Gamma \subset \R^n$
satisfies $\mu(\Gamma) = 1$ and $\Phi^1(\Gamma) = \Gamma$. 

For $L > 0$, define
\[
\Gamma_L = \left\{x \in \Gamma : L^{-1} \leq \frac{\sqrt{g_x(v, v)}}{|v|} \leq L \quad \text{ for all } v \in \S^{n-1} \right\} \cap \left\{ x \in \Gamma : |B(x,x)| \geq L^{-1} \right\} \, .
\]
and note that if $x, \Phi^n x \in \Gamma_L$ for some $n \geq 0$, then $|D_x \Phi^n| \leq L^2$ must hold by \eqref{eq:isometrySec5Proof}. Moreover, we have that
 $\mu(\Gamma_L) \nearrow \mu(\Gamma) = 1$ as $L \to \infty$. Observe that this 
 relies on the assumption that $B$ is not identically 0, hence $|B(x,x)| > 0$ Lebesgue-a.e. 
(here, we use the standard fact that $\{B(x,x) = 0\}$ is a proper variety in $\R^n$, hence
must have zero volume). 

Fix $L$ such that $\mu(\Gamma_L) \geq 1/2 > 0$. By the Poincar\'e Recurrence Theorem, $\mu$-a.e.  
$x \in \Gamma_L$ visits $\Gamma_L$ infinitely many times.  Fix such an $x \in \Gamma_L \setminus \{ 0 \}$ and let $0 := n_0 < n_1 < n_2 < \cdots, \lim_{\ell \to \infty} n_\ell = \infty$, so that $\Phi^{n_\ell} (x) \in \Gamma_L$ for all $n_\ell$, hence
\[
|D_x\Phi^{n_\ell}|\leq L^2
\]
for all such $n_\ell$. On the other hand, \eqref{eq:normGrowth} implies 
\[
|D_x\Phi^{n_\ell}| \geq \frac{n_\ell}{L |x|}
\]
as $n_\ell \to \infty$, a contradiction. This completes the proof of Proposition \ref{prop:mainRigidityIntro}.

\appendix
\section{Basic geometric preliminaries and compactness criterion}

In this appendix we summarize some basic lemmas surrounding Sobolev and Besov-type spaces on Riemannian manifolds, especially $\S \R^n := \R^n \times \S^{n-1}$.  
None of the results here are new, but we could not locate proofs in the literature exactly matching the statements we use in the proof and so we have provided sketches for the readers' convenience (though see \cite{Triebel} Chapter 7 for very similar results). 

\subsection{Fractional Sobolev norms and Gagliardo-Nirenberg-Sobolev inequalities}

First, we recall the definition of fractional Sobolev norms on $(\cM,g)$ a smooth, connected, geodesically complete, $d$-dimensional Riemannian manifold with bounded geometry and positive injectivity radius $\delta_0>0$ (see \cite{Triebel} Chapter 7 for more discussion on the meaning and significance of these assumptions for function spaces). For any $s\in(0,1)$ and $p \in [1,\infty)$, define the $W^{s,p}$ fractional Sobolev norm by
\begin{equation}\label{def:Wsp}
  \norm{w}_{W^{s,p}} = \norm{w}_{L^p} +  \left(\int_{\mathcal{M}} \int_{h \in T_x \mathcal{M} : \abs{h} < \delta_0} \frac{\abs{ w(\mathrm{exp}_x h) - w(x)}^p}{\abs{h}^{sp+d}}\dee h\dee q(x) \right)^{1/p}
\end{equation}
where $\mathrm{exp}_x:T_x\cM \to \cM$ denotes the standard exponential map on $\mathcal{M}$.
Note that if the manifold is geodesically complete but does not have bounded geometry and a globally positive injectivity radius, this definition can still be used locally (e.g. on closed, bounded, geodesic balls). 
In this work we will only need the spaces $W^{s,1}$ and $H^s := W^{s,2}$. 

By the aforementioned geometrical assumptions, if $0< \delta <\delta_0$, then there exists a locally finite covering $\{B_\delta(z_j)\}$ of geodesic balls and an associated smooth partition of unity $\{\chi_j\}$, where $\mathrm{supp} \chi_j \subset B_\delta(z_j)$ (see Proposition in 7.2.1 \cite{Triebel} ). We have that the coordinate maps $\xbf_j = \exp_{z_j}: B_{\delta}(0;\R^m)\to B_{\delta}(z_j;\cM)$ satisfy for each $k\in \Z_+$
\[
  \sup_j\|\nabla^k\xbf_j\|_{L^\infty} + \sup_j\|\nabla^k\xbf_j^{-1}\|_{L^\infty} <\infty
\]
and that in $\xbf_j$ coordinates there exists a $c>1$ such that for each $x\in B_{\delta}(0;\R^m)$, $J_j(x)= \sqrt{|\det g|}$ satisfies $c^{-1}\leq J_j(x)\leq c$.

To begin, we record two Gagliardo-Nirenberg-Sobolev-type inequalities on $\mathbb S \mathbb R^d$ which are used in Appendix \ref{sec:Basics} to give qualitative estimates on stationary measures. 
\begin{lemma} \label{lem:GNSSR}
For all $n \geq 1$, all $s \in (0,1)$, there exists a $\theta \in (0,1)$ such that the following holds $\forall f \in C^\infty_c(\S \R^n)$
\begin{align}
\norm{f}_{L^2} \lesssim_{n,s} \norm{f}_{L^1}^{1-\theta} \norm{f}_{H^s}^{\theta}. \label{ineq:L1HsSR} 
\end{align}
For the homogeneous case, there exist $\theta_1,\theta_2 \in (0,1)$ such that the following holds for all $f \in C^\infty_c(\S \R^n)$
\begin{align}
\norm{f}_{L^2} \lesssim_{n} \norm{f}_{L^1}^{1-\theta_1} \|(-\Delta_{x,v})^{1/2}f\|_{L^2}^{\theta_1} + \norm{f}_{L^1}^{1-\theta_2} \|(-\Delta_{x,v})^{1/2}f\|_{L^2}^{\theta_2}, \label{ineq:L1H1SR} 
\end{align}
where $\Delta_{x,v}$ denotes the (negative definite) Laplace-Beltrami operator on $\S \R^n$. 
\end{lemma}
\begin{proof}
Let $m = \mathrm{dim}\, \S \R^n = 2n-1$.
The inhomogeneous inequality \eqref{ineq:L1HsSR} follows easily from the corresponding estimate on $\R^m$ as the manifold $\S \R^n$ has uniformly bounded geometry and non-negative Ricci curvature (see e.g. [pg 301 inequality (15), \cite{Triebel}]).
Let us briefly sketch the argument for the readers' convenience.
  
Define $\tilde{f}_j := (\chi_j f) \circ \mathbf{x}_j^{-1}$.
Then, from the fact that \eqref{ineq:L1HsSR} holds on $\R^m$ and that, by Theorem 7.5.1 \cite{Triebel}, $\left(\sum_{j=1}^\infty\|\tilde{f}_j\|_{H^s(\R^m)}^2\right)^{1/2}$
is an equivalent norm for $H^s(\cM)$, we have
\begin{align*}
\norm{f}_{L^2(\cM)}^2  =\sum_{j=1}^\infty \norm{\tilde{f}_j}_{L^2(\R^m)}^2 &\lesssim \sum_{j=1}^\infty \norm{\tilde{f}_j}_{L^1(\R^m)}^{2(1-\theta)} \norm{\tilde{f}_j}_{H^s(\R^m)}^{2\theta} \\
& \leq \left(\sum_{j=1}^\infty \norm{\tilde{f}_j}_{L^1(\R^m)}^2 \right)^{1-\theta} \left(\sum_{j=1}^\infty \norm{\tilde{f}_j}_{H^s(\R^m)}^2 \right)^{\theta} \\
& \lesssim \norm{f}_{L^1(\cM)}^{1-\theta}\left(\sum_{j=1}^\infty \norm{\tilde{f}_j}_{L^1(\R^m)} \right)^{1-\theta} \norm{f}_{H^s(\cM)}^{2\theta} \\
&= \norm{f}_{L^1(\cM)}^{2(1-\theta)} \norm{f}_{H^s(\cM)}^{2\theta},  
\end{align*}
which is the desired result.

The homogeneous $\|(-\Delta_{x,v})^{-1}f\|$ norms in \eqref{ineq:L1H1SR} require a more intrinsic treatment.
Thus, we use heat semigroup methods, commonly used to treat such inequalities on manifolds; see e.g. \cite{Ledoux2003}.   
Denote $e^{t\Delta}$ the heat semigroup associated to the Laplace-Beltrami operator (omitting the $x,v$ subscript for simplicity).
By the Li-Yau inequality \cite{LiYau} and the fact that $\S \R^n$ has non-negative Ricci curvature, the following pointwise upper bound holds for $e^{t\Delta}f$
\begin{align}
\abs{e^{t\Delta }f(z)} \lesssim \int_{\S \R^d} \frac{1}{\mathrm{vol} B(z',\sqrt{t}) } e^{-cd(z,z')^2/t} \abs{f(z')} \dee q(z^\prime), \label{ineq:LiYau} 
\end{align}
where $B(z',\sqrt{t})$ is the geodesic ball of radius $\sqrt{t}$ centered at $z$ and $d(z,z')$ denotes geodesic distance and $0<c<1/4$ is some constant.   
Note that for $t< 1$ we have $\mathrm{vol} B(z',\sqrt{t}) \approx t^{-m/2}$ whereas for $t > 1$ we have $\mathrm{vol} B(z',\sqrt{t}) \approx t^{-n/2}$.  
From this, we obtain the following $L^2$ regularization estimate
\begin{align}
\norm{e^{t\Delta} f}_{L^2} & \lesssim \frac{1}{(t^{-n/2} + t^{-m/2})^{1/2}}\norm{f}_{L^1} \approx  \frac{1}{(t^{-n/4} + t^{-m/4})}\norm{f}_{L^1} \label{ineq:L1L2SR}. 
\end{align}
Next, observe that
\begin{align*}
\norm{e^{t\Delta} f - f}_{L^2} & = \norm{\int_0^t \Delta e^{\tau \Delta} f \dee \tau }_{L^2}  \leq \int_0^t \norm{ \Delta e^{\tau \Delta} f  }_{L^2} \dee \tau \lesssim t^{1/2} \norm{ (-\Delta)^{1/2} f }_{L^2}. 
\end{align*}

Let $g = f/\norm{f}_{L^1}$. 
Using the heat semigroup to mollify $g$, we obtain from \eqref{ineq:L1L2SR}
\begin{align*}
\norm{g}_{L^2} & \leq \norm{e^{t\Delta} g - g}_{L^2} + \norm{e^{t\Delta} g}_{L^2} \lesssim t^{1/2} \norm{(-\Delta)^{1/2} g}_{L^2} + \frac{1}{(t^{-n/4} + t^{-m/4})}. 
\end{align*}
Choosing $t = \norm{(-\Delta)^{1/2} g}_{L^2}^{-\gamma}$ yields
\begin{align*}
\norm{g}_{L^2} & \leq \norm{(-\Delta)^{1/2} g}_{L^2}^{1-\gamma/2} + \norm{(-\Delta)^{-1/2} g}_{L^2}^{n \gamma/4} + \norm{(-\Delta)^{1/2} g}_{H^s}^{m\gamma/4}. 
\end{align*}
As $m > n$, it makes sense to set $1-\gamma/2 = m\gamma/4 =: \theta_1$ to obtain $\gamma = (\frac{1}{2} + \frac{m}{4})^{-1}$ so that $\theta_1 \in (0,1)$
and $n \gamma /4 =: \theta_2 \in (0,1)$ as well. 
Hence we have 
\begin{align*}
\norm{g}_{L^2} & \lesssim \norm{(-\Delta)^{1/2}g}_{L^2}^{\theta_1} + \norm{(-\Delta)^{1/2}g}_{L^2}^{\theta_2}.  
\end{align*}
Inequality \eqref{ineq:L1HsSR} follows on recalling that $g = f/\norm{f}_{L^1}$.
\end{proof}

\subsection{Local embeddings and compactness in $L^1$}

For Theorem \ref{thm:HypoFI}, we need to prove the following embedding of the Besov-type norm $\|\cdot\|_{\Lambda^s}$ (defined in \eqref{def:Ls}) into $W^{s,1}$ locally. In the remainder of this section, we will assume that $(\cM,g)$ is a geodesically complete, connected smooth Riemannian manifold. Since we will be working locally, we do not need any assumptions about bounded geometry.
\begin{lemma} \label{lem:WsLs}
Let $U \subset \cM$ be any bounded, open geodesic ball and fix a suitable atlas of $U$, $\set{\mathbf{x}_j}_{j=1}^N$, as above to define the norm \eqref{def:Ls}.
For all $0 < s' < s < 1$ there holds the following $\forall w \in C^\infty_c(U)$
\begin{align*}
\norm{w}_{W^{s',1}} \lesssim_U \norm{w}_{\Lambda^s}. 
\end{align*}
\end{lemma}
\begin{proof}
First, note that we can restrict the integrals and suprema in the norms to be over sufficiently small sets depending on the local geometry.  
Let $B_R(z_0;\cM) = U$ and consider the canonical basis for the tangent space $T_{z_0} \cM$, $\set{\partial_j}$ defined by the pullback of the cannonical directions in $\R^d$ under the exponential map $\exp_{z_0}$.
Using geodesics, this basis can be parallel transported  to a full set of vector fields $\set{Z_k}_{k=1}^d$, over $B(z_0,\delta)$ for some $\delta > 0$ smaller than the injectivity radius and that this set of vector fields forms a basis for the tangent space at every point.
Then we can write (for some $\delta' > 0$ possibly smaller than $\delta$), 
\begin{align*}
\norm{w}_{W^{s',1}} \lesssim  \norm{w}_{L^1} + \int_{\mathcal{M}} \int_{c \in \R^d: \abs{c} \leq \delta'} \frac{\abs{ w(\mathrm{exp}_x \sum_j c_j Z_j(x)) - w(x)}}{\abs{c}^{d + s'}}\dee c\, \dee q(x). 
\end{align*}
Note that since $\sum_j \abs{Z_j(x)} \approx 1$, we have 
\[
d\left(\mathrm{exp}_x \sum_j c_j Z_j(x),x\right) = \left|\sum_j c_jZ_j(x)\right| \approx \abs{c}.
\]
The key point here is that we have re-written the integral over the unit ball in the tangent spaces $T_{z} \cM$ into a ball in $\R^d$ which is the same over all $x \in B(x_0,\delta)$ (more accurately we have estimated the previous integral from above by this quantity plus the $L^1$ norm).
Therefore, we can apply Fubini and conclude that for $s' < s$, 
\begin{align*}
\norm{w}_{W^{s',1}} & \lesssim  \norm{w}_{L^1} +  \int_{c \in \R^d: \abs{c} \leq \delta'} \int_{\mathcal{M}} \frac{\abs{ w(\mathrm{exp}_x \sum_j c_j Z_j(x)) - w(x)}}{\abs{c}^{d + s'}} \dee q(x)\, \dee c \\
& \lesssim  \norm{w}_{L^1} +  \sup_{c \in \R^d: \abs{c} \leq \delta'} \int_{\mathcal{M}} \frac{\abs{ w(\mathrm{exp}_x \sum_j c_j Z_j(x)) - w(x)}}{\abs{c}^{s}}\dee q(x). 
\end{align*}
It is not hard to see that by writing the integral in the coordinate parameterization used to define $\Lambda^s$, the right hand side is bounded above by the $\Lambda^s$ norm.  
\end{proof} 
Next, we give a proof that local $W^{s,1}$ Sobolev smoothness plus a tightness condition estimate yields precompactness in $L^1$. While this type of result is very standard in functional analysis, we could not find a reference that gives the proof in the form we need. A proof is provided for the reader's convenience using a standard compactness criterion in $L^1$ on metric spaces.
\begin{lemma}\label{lem:Ws1Precompact}
Consider a bounded sequence $\set{f_n}_{n=1}^\infty \subset L^1(\cM)$ that satisfies the following properties:
\begin{enumerate}
\item (Local uniform regularity) For every geodesic ball $ B = B_R(x;\cM)$, one has
\[
  \sup_n\norm{\chi_B f_n}_{W^{s,1}} < \infty,
\]
where $\chi_B$ is a smooth cut-off function equal to $1$ inside $B_R(x)$ and compactly supported in $B_{2R}(x)$.
\item (Tightness) There exists an $x_0\in \cM$ such that
\[
  \lim_{R\to \infty} \sup_{n}\|f_n\|_{L^1(\cM\backslash B_R(x_0))} = 0.  
\]
\end{enumerate}
Then the sequence $\{f_n\}_{n=1}^\infty$ is strongly precompact in $L^1(\cM)$ with limit points in $L^1 \cap W^{s,1}_{loc}$. 
\end{lemma}
\begin{proof}
To start, we will use a general Frech\'et-Kolmogorov compactness criterion in $L^1$ on bounded metric spaces due to Krotov \cite{Krotov2012-ne}. We state it below for convenience.

\begin{theorem}[\cite{Krotov2012-ne} Theorem 5]
Let $(X,d)$ be a bounded complete metric space with a finite measure $\mu$ satisfying the doubling condition, that is, there exists a constant $c_\mu >0$ such that
\[
  \mu(B_{2r}(x))\leq c_\mu \mu(B_r(x)), \quad x\in X, \quad r>0.
\]
Let $\mathcal{S}$ be a bounded subset of $L^1(X)$, then $\mathcal{S}$ is precompact in $L^1(X)$ if and only if the following condition is satisfied:
\[
  \lim_{r\to 0}\sup_{f\in \mathcal{S}}\int_{X}\left(\fint_{B_r(x)}|f(x)-f(y)|^{\frac{1}{2}}\dee \mu(y)\right)^{2}\dee\mu(x) = 0.
\]
\end{theorem}

Note that the volume measure on any $n$-dimensional compact Riemannian manifold satisfies $\text{vol}\,B_r(x)\approx r^n$ and is therefore a doubling measure on the complete metric space with metric given by the usual geodesic distance. In what follows fix $x_0$ as above and let $U_R := \overline{B_R(x_0)}$ be a closed geodesic ball for a fixed $R>0$. By the Hopf-Rinow theorem (since $\cM$ is complete) $U_R$ is itself a compact manifold and a bounded complete metric space with the geodesic metric $d$ inherited from $\cM$.
We denote $\hat{B}_r(x) = B_r(x)\cap U_R$ the metric ball on this space.
It follows by Cauchy-Schwarz and the fact that for $r$ small enough $\text{vol}\hat B(x,r) \geqc r^d$ gives
\[
\begin{aligned}
  \left(\fint_{\hat{B}_r(x)}|f(x)-f(y)|^{\frac{1}{2}}\dq(y)\right)^{2}
  &\leqc \left(\int_{\hat B_r(x)}\frac{|f(x)-f(y)|}{d(x,y)^{d+s}}\dq(y)\right) r^s. 
\end{aligned}
\]
Taking $r>0$ smaller than the injectivity radius of $U_R$ and changing coordinates using the inverse exponential map $\exp_x^{-1}$ and using that $d(x,\exp_xh) = |h|$ gives
\[
  \int_{\hat B_r(x)}\frac{|f(x)-f(y)|}{d(x,y)^{d+s}}\dq(y) = \int_{h\in T_x\cM:|h|<r}\frac{|w(\exp_xh) - w(x)|}{|h|^{d+s}}J_x(h)\dee h,
\]
where $J_x$ is the Jacobian factor that describes the volume measure on $U_R$.
Since $U_R$ is compact $J_x(h)$ is bounded and therefore we have
\[
  \int_{U_R}\left(\fint_{\hat B_r(x)}|f(x)-f(y)|^{\frac{1}{2}}\dee q(y)\right)^{2}\dq(x) \leqc \left(\|\chi_{U_{2R}} f\|_{W^{s,1}}\right)r^{s}.
\]

This implies that for each fixed $j$, the set $\{\chi_{U_{2^j}} f_n\}_n$ is a precompact set in $L^1(\cM)$ for all $j$. Diagonalization gives a limit $f \in L^1$ and a subsequence $\set{f_{n_k}}$ that converges in $L^1$ on every geodesic ball.
Now, for all $\eps > 0$, 
\begin{align*}
\norm{f_{n_k} - f}_{L^1} \leq \norm{f_{n_k} - f}_{L^1(U_R)} + \norm{f_{n_k}}_{L^1(\cM \setminus U_R)} +  \norm{f}_{L^1(\cM \setminus U_R)}.
\end{align*}
By tightness, we can choose the last $R > 0$ sufficiently large such that the last two terms are together less than $\eps/2$ and by the convergence on all compacts, choose $n_k$ large enough so that the first term is $\eps/2$.

The fact that $f\in W^{s,1}_{\loc}$ follows from lower semi-continuity of the $W^{s,1}_{\loc}$ semi-norms for $s\in (0,1)$ with respect to $L^1$ convergence (this can be easily be deduced by choosing a further sub-sequence so that $f_{n_k} \to f$ almost surely and applying Fatou's lemma to the double integral of the difference $|f_{n_k}(x)-f_{n_k}(y)|/d(x,y)^{s+d}$).
\end{proof}

\section{Qualitative properties of the projective stationary measure}\label{sec:Basics}
In this section we record basic properties of the SDE \eqref{eq:SDEintroFD22}. 
\begin{theorem} \label{thm:BasicProjDensity} 
Consider the Euler-like model \eqref{eq:SDEintroFD22}.
Suppose that $B$ satisfies Assumption \ref{ass:bilinear2} and that the vector fields $\{\tilde{X}_0^\ep, \tilde{X}_1,...,\tilde{X}_r\}$ satisfies the uniform parabolic H\"ormander condition on $\S\R^n$ as in Definition \ref{def:UniHormander} (Assumption \ref{ass:PHC}). 
Then $\forall \epsilon > 0$, the SDE \eqref{eq:SDEintroFD22} satisfies Assumptions \ref{ass:WP} and \ref{ass:WP2}. 
Moreover, the  stationary measure of the $(w_t)$ process $f^\epsilon$ has a smooth density with respect to Lebesgue measure $f^\epsilon \in L^1 \cap L^2 \cap C^\infty$ with $f^\epsilon \log f^\epsilon \in L^1$,
and $\exists C,\gamma > 0$ such that $\forall \epsilon \in (0,1]$,
\begin{align}
\int_{\S \R^n} f^\epsilon e^{\gamma \abs{x}^2} \, \dee q < C. \label{ineq:tightness}
\end{align}
and $\forall N > 0$ the following moment bound holds $\forall \epsilon \in (0,1]$ (not uniformly in $N$ or $\eps$ of course)
\begin{align}
\int_{\S \R^n} \brak{x}^N f^\epsilon \log f^\eps \, \dee q < \infty. \label{ineq:tightness}
\end{align}
Furthermore, the estimate in Assumption \ref{ass:WP} (iii) holds for all $\epsilon > 0$. 
\end{theorem}
\begin{proof}[Proof of Theorem \ref{thm:BasicProjDensity}]
Claims (i) and (ii) of Assumption \ref{ass:WP} are standard or proved in \cite{BL20}.
The proof of Assumption \ref{ass:WP} (iii) follows by providing suitable moment estimates on $\log \abs{\det D \Phi^t}$ and $\log \abs{D\Phi^t}$ using the SDE derived in the proof of Proposition \ref{lem:FK}.
Indeed, denoting $J_t = D\Phi^t(x) \xi$ we have
\begin{align*}
\frac{d}{dt} J_t = B(x,J_t) + B(J_t,x_t) + \eps AJ_t, 
\end{align*}
and hence, $\exists C> 0$, such that 
\begin{align*}
\norm{J_t} \leq \norm{\xi} \exp \left(C \int_0^t \norm{x_\tau} \dee \tau \right).
\end{align*}
By time-reversal of ODEs, $V_t = J_t^{-1}$ satisfies the analogous estimate, possibly by adjusting $C$ (the growth is due to the damping)
\begin{align*}
\norm{J_t^{-1}} \leq \norm{\xi} \exp \left(C \int_0^t (\eps + \norm{x_\tau}) \dee \tau \right). 
\end{align*}
Therefore,
\begin{align*}
\E \int_{\R^n} \left[\log^+{|D_x\Phi^{t}|} + \log^+{|(D_x\Phi^{t})^{-1}|} \right]\dee \mu(x) & \lesssim \E \int_{\R^n} \left[ \eps t + \int_0^t \norm{x_\tau} \dee \tau \right]\dee \mu(x)\\ 
& = t \left( \eps + \int_{\R^n} \norm{x} \dee \mu(x) \right), 
\end{align*}
which is finite by standard moment estimates (see e.g. the drift condition \eqref{ineq:fmombd} below). This completes the proof of Assumption \ref{ass:WP}. 

The results of Assumption \ref{ass:WP2} follow from similar methods
with the exception of the uniqueness of the stationary measure. By the parabolic H\"ormander's condition on $\{\tilde{X}_0, \tilde{X}_1,...,\tilde{X}_r\}$ together with H\"ormander's theorem \cite{H67,H11} and the Doob-Khasminskii theorem \cite{DPZ96}, it is sufficient to verify irreducibility as in Definition \ref{def:Irr}. 
This follows by Proposition \ref{prop:projective irreducibility}.
We are not aware of any existing results that directly imply $f^\epsilon \in L^2$ or $f^\epsilon \log f^\epsilon \in L^1$ in the literature and we therefore include the proof. 
For this we use some ideas that appear in \cite{BL20}, in which similar estimates are proved for the density of the base process $\rho^\eps$.
As in \cite{BL20}, a convenient method to justify many formal calculations begins by first regularizing the problem by adding elliptic Brownian motions. Recall that the generator $\tilde{\mathcal{L}}$ for the projective process $(w_t)$ is given by
\[
	\tilde{\mathcal{L}} = \tilde{X}_0 + \frac{1}{2}\sum_{k=1}^r \tilde{X}_k^2.
\]
We then regularize this by the perturbing the generator
\begin{align*}
\tilde{\mathcal{L}}^\delta = \tilde{\mathcal{L}} + \delta\Delta_{x,v}, 
\end{align*}
where $\Delta_{x,v} = \Delta_x + \Delta_v$ with $\Delta_x$ the usual Laplacian on $\R^n$ and $\Delta_v$ the Laplace-Beltrami operator on $\S^{n-1}$. This corresponds to perturbing the SDE \eqref{eq:SDEintroFD22} by a non-degenerate $\sqrt{\delta}$ Brownian motion on $\S\R^n$.
It is not hard to show that $\tilde{\mathcal{L}}^\delta$ satisfies a drift condition\footnote{The lifted part of the vector fields vanish due to the lack of $v$ dependence, and so this reduces to the same drift condition for the base process. This follows immediately noting that since $x \cdot B=0$ we have $X_0 e^{\gamma\abs{x}^2} = 0$, that  $\abs{\partial_{x_j x_j} e^{\gamma\abs{x}^2}} \lesssim (\gamma^2 \abs{x}^2 + \gamma)e^{\gamma \abs{x}^2}$, and by negative-definiteness, $Ax \cdot \grad e^{\gamma \abs{x^2}} \lesssim -\gamma\abs{x}^2 e^{\gamma \abs{x}^2}$. Hence, the condition follows for $\gamma$ chosen sufficiently small.} 
\[
	\tilde{\mathcal{L}}^\delta e^{\gamma |x|^2} \leq - \eps \alpha e^{\gamma |x|^2} + \eps K,
\]
for some $\alpha \in (0,1)$, $K \geq 1$ (uniformly in $\epsilon,\delta$).
This gives rise to a globally defined Markov process $(w_t^\delta)$. Moreover for a given initial density $f \in C^\infty_c(\S \R^d)$ with $\int f \dee q = 1$ and $f \geq 0$, such that $\mathrm{Law}(w^\delta_0) = f$ we denote $f_t = \mathrm{Law}(w^\delta_t)$, which solves the forward Kolmogorov equation
\begin{equation}\label{eq:kolmog-eq}
	\partial_t f_t = (\tilde{\mathcal{L}}^\delta)^*f_t + \delta \Delta_{x,v}f_t.
\end{equation}
From the drift condition we have, $\forall \gamma$ sufficiently small, $\exists \alpha \in (0,1)$ such that (uniformly in $\epsilon,\delta$), 
\begin{align}
\int_{\S \R^d} f_t e^{\gamma\abs{x}^2} \dee q \lesssim 1 + e^{-\alpha t} \int_{\S \R^d} f e^{\gamma \abs{x}^2} \dee q. \label{ineq:fmombd}
\end{align}
Let $\bar{\chi} \in C^\infty_c(B(0,1))$ with $0 \leq \bar{\chi} \leq 1$, and $\bar{\chi} = 1$ for $|x| \leq 1/2$ and define $\chi(x) = \bar\chi(x/2) - \bar\chi(x)$. Define $\chi_j = \chi(2^{-j} x)$, which defines the partition of unity $1 = \bar{\chi} + \sum_{j=0}^\infty \chi_j(x)$. 
From energy estimates on \eqref{eq:kolmog-eq} we have the following,
\begin{align*}
\frac{\dee}{\dt}\lnorm{f_t}_{L^2}^{2} + \delta\lnorm{(-\Delta_{x,v})^{1/2} f_t}_{L^2}^2 & \lesssim \lnorm{(1+ |x|) f_t}_{L^2}^2 \lesssim \lnorm{\bar{\chi} f_t}_{L^2}^2 + \sum_{j=1}^\infty 2^{2j} \lnorm{\chi_j f_t}_{L^2}^2, 
\end{align*}
(in order to justify such estimates one may apply smooth, $v$-independent radially symmetric cut-offs to the nonlinearity and pass to the limit).  
By the Gagliardo-Nirenberg-Sobolev-type inequality \eqref{ineq:L1H1SR} (Lemma \ref{lem:GNSSR}), for $\theta_1,\theta_2 \in (0,1)$ as therein, there holds (recall that $\lnorm{f_t}_{L^1}=1$)  
\begin{align*}
\frac{\dee}{\dt}\lnorm{f_t}_{L^2}^{2} + \delta\lnorm{(-\Delta_{x,v})^{1/2} f_t}_{L^2}^2 & \lesssim \lnorm{\bar{\chi} f_t}_{L^2}^2 + \sum_{j=1}^\infty 2^{2j} \lnorm{\chi_j f_t}_{L^2}^2 \\  
& \hspace{-4cm} \lesssim \lnorm{(-\Delta_{x,v})^{1/2} \bar{\chi} f_t}_{L^2}^{2\theta_1} \lnorm{\bar{\chi} f_t}_{L^1}^{1-2\theta_1} + \sum_{j=1}^\infty 2^{2j} \lnorm{(-\Delta_{x,v})^{1/2}\chi_j f_t}_{L^2}^{2\theta_1} \lnorm{ \chi_j f_t}_{L^1}^{1-2\theta_1} \\
& \hspace{-4cm} \quad + \lnorm{(-\Delta_{x,v})^{1/2} \bar{\chi} f_t}_{L^2}^{2\theta_2} \lnorm{\bar{\chi} f_t}_{L^1}^{1-2\theta_2} + \sum_{j=1}^\infty 2^{2j} \lnorm{(-\Delta_{x,v})^{1/2}\chi_j f_t}_{L^2}^{2\theta_2} \lnorm{ \chi_j f_t}_{L^1}^{1-2\theta_2} \\
& \hspace{-4cm} \lesssim \lnorm{(-\Delta_{x,v})^{1/2} f_t}_{L^2}^{2\theta_1} + \lnorm{\grad_x \bar{\chi} f_t}_{L^2}^{2\theta_1} + \sum_{j=1}^\infty 2^{2j} \left( \lnorm{(-\Delta_{x,v})^{1/2} f_t}_{L^2}^{2\theta_1} + \lnorm{\grad_x \chi_j f_t}_{L^2}^{2\theta_1} \right) \lnorm{ \chi_j f_t}_{L^1}^{1-2\theta_1} \\
& \hspace{-4cm} \quad +  \lnorm{(-\Delta_{x,v})^{1/2} f_t}_{L^2}^{2\theta_2} + \lnorm{\grad_x \bar{\chi} f_t}_{L^2}^{2\theta_2} + \sum_{j=1}^\infty 2^{2j} \left( \lnorm{(-\Delta_{x,v})^{1/2} f_t}_{L^2}^{2\theta_2} + \lnorm{\grad_x \chi_j f_t}_{L^2}^{2\theta_2} \right) \lnorm{ \chi_j f_t}_{L^1}^{1-2\theta_2}. 
\end{align*}
Therefore, 
\begin{align*}
\frac{\dee}{\dt}\lnorm{f_t}_{L^2}^{2} + \delta\lnorm{(-\Delta_{x,v})^{1/2} f_t}_{L^2}^2
& \lesssim_\delta  \lnorm{(-\Delta_{x,v})^{1/2}f_t}_{L^2}^{2\theta_1} + \lnorm{(-\Delta_{x,v})^{1/2}f_t}_{L^2}^{2\theta_2} + \int_{\S \R^d} f_t e^{\gamma \abs{x}^2} \dee q. 
\end{align*}

Hence, from \eqref{ineq:fmombd}, there holds for some $q > 2$, 
\begin{align*}
\frac{1}{t}\int_0^t \lnorm{(-\Delta_{x,v})^{1/2} f_\tau}_{L^2}^2  \dee\tau \lesssim \delta^{-q} \int_{\S \R^d} f e^{\gamma \abs{x}^2} \dee q.  
\end{align*}
Combined with the uniform drift condition, this allows to pass to the limit $t \to \infty$ and conclude that the density of the unique stationary measure, denoted below as $f^{\epsilon,\delta}$ is in $H^1(\S \R^d)$ (recall \eqref{ineq:L1HsSR}); we note that $f^{\epsilon,\delta}$ is a smooth solution of the Kolmogorov equation
\begin{align}
\left(\tilde{\mathcal{L}}^\ast +\delta\Delta_{x,v} \right) f^{\epsilon,\delta} = 0. \label{eq:fed}
\end{align}
Next, we obtain an $L^2$ estimate that is uniform in $\delta$ in order to pass to the $\delta \to 0$ limit.
For this, we clearly need to depend on hypoelliptic regularity. 
Define the regularized H\"ormander norm pair (see discussions in \cite{H67,AM19,BL20} for motivations),  
\begin{align*}
\lnorm{w}_{\mathcal{H}_\delta} & := \lnorm{w}_{L^2} + \sum_{k=1}^r \lnorm{X_k w}_{L^2} + \delta \lnorm{(-\Delta_{x,v})^{1/2}w}_{L^2} \\ 
\lnorm{w}_{\mathcal{H}^\ast_\delta} & := \sup_{\varphi: \lnorm{\varphi}_{\mathcal{H}_\delta} \leq 1} \abs{ \int_{\S \R^n}  (\tilde{X}_0\varphi) w\, \dee q}
\end{align*}
The proof is similar to  [Lemma 2.3; \cite{BL20}] provided we have the following quantification of H\"ormander's inequality. 
\begin{lemma}[Quantitative H\"ormander inequality for the projective process] \label{lem:HineqBall}
Suppose that $\{\tilde{X}_0, \tilde{X}_1,...,\tilde{X}_r\}$ satisfies the uniform parabolic H\"ormander condition on $B(0,2) \times \S^{n-1}$ as in Definition \ref{def:UniHormander}. There exists $s > 0$ and $q > 0$, such that for any $R \geq 1$,  $w \in C^\infty_c(B_R \times \S^{n-1})$ and $\delta \in [0,1]$ there holds
\begin{align}
\|w\|_{H^s} \leqc R^{q} (\lnorm{w}_{\mathcal{H}_\delta}  + \lnorm{w}_{\mathcal{H}^\ast_\delta}), \label{ineq:H1hypProj}
\end{align}
where both $s>0$ and the implicit constant do not depend on $\epsilon$, $\delta$, or $R$, and $H^{s}= W^{s,2}$.

\end{lemma}
\begin{proof}
The proof begins with a re-scaling as in [Lemma 3.2; \cite{BL20}]. Define $h(x,v) = w(Rx,v)$ which solves a PDE of the following form for suitable vector fields $N$, $V$, $Y$,
\begin{align*}
\epsilon \delta \Delta_{x,v} h + \frac{1}{2}\sum_{j=1}^r \epsilon (\tilde{X}_j^*)^2 h - N h + R^{-1} V^\ast h - \frac{\epsilon}{R} Y^* h = 0. 
\end{align*}
where $N(x) = B(x,x)$, $Y(x) = Ax$ and $V(x,v) = \Pi_v \nabla F(x) v$, and their action on $h$ is interpreted as a differential operator.  
We see that the proof here is more subtle than in the corresponding [Lemma 3.2; \cite{BL20}] as  $R^{-1} V$ is required to span the directions in projective space.
From Proposition \ref{prop:class-proj-span}, we see that the spanning in $x$ and $v$ can be considered essentially separately, first choosing brackets to span in $x$ and then correcting by choosing suitable brackets in $\mathfrak{m}_x(X_0;X_1,\ldots,X_r)$ to span in $v$.
Using this structure we see that given a vector field $Z \in \mathfrak{X}(\S\R^n)$ and $q_0 \in B(0,1)\times \S^{n-1}$,
there exists $p_j < ... < p_2 < p_1 \leq k$ (with $k$ as in Definition \ref{def:UniHormander})
such that for $q$ in a neighborhood of $q_0$, there are finitely many smooth coefficients $c_j$ and vectors $Z_j \in \mathscr{X}_k$ with 
\begin{align*}
Z(q) = \sum_{j} R^{p_j} c_j(q) Z_j(q), 
\end{align*}
where if $Z$ varies in a bounded set in $C^m$, then $\set{c_j}_{j}$ varies in a similarly bounded set as well.
A careful reading of \cite{H67} shows that this introduces powers of $R$ matching the powers of $t$ into all of the estimates in [Sections 4 and 5; \cite{H67}], the maximal power arising being $R^{k}$. In particular,  the error estimates come in the form $\mathcal{O}(R^{k/\sigma})$, provided that $R^{k} t < 1$ and $0 < \sigma < s^*$ as in \cite{H67}. 
This restriction on $t$ in the estimates further introduces only polynomial dependence on $R$, as for any $Z \in \mathfrak{X}(\S\R^n)$,
\begin{align*}
\sup_{\abs{t} \leq 1} \abs{t}^{-\sigma} \norm{e^{t Z} g - g}_{L^2}\lesssim R^{k \sigma}\norm{g}_{L^2} + \sup_{\abs{t} \leq R^{-k}} \abs{t}^{-\sigma} \norm{e^{t Z} g - g}_{L^2}.
\end{align*}
Combining the above observations with those of \cite{H67} implies that the constant in \eqref{ineq:H1hypProj} remains polynomial in $R$ (exponential would also be sufficient for our purposes, as we only use that the constant is bounded above by $e^{\eta R^2}$ for  $\eta < \gamma$). 
\end{proof}
Once one has Lemma \ref{lem:HineqBall}, the proof of Theorem \ref{thm:BasicProjDensity} follows easily, given that we are not seeking $\epsilon$-independent bounds, as these such bounds will be false for all but the most degenerate models (see [Lemma 2.4; \cite{BL20}] for the corresponding argument on $\rho^\epsilon$, which does yield $\epsilon$-independent estimates). 
Let $\bar{\chi} \in C^\infty_c(B(0,1))$ with $0 \leq \bar{\chi} \leq 1$, and $\bar{\chi} = 1$ for $x \leq 1/2$ and define $\chi(x) = \chi(x/2) - \chi(x)$. Define $\chi_j = \chi(2^{-j} x)$, which defines the partition of unity $1 = \bar{\chi} + \sum_{j=0}^\infty \chi_j(x)$. 

We now obtain a uniform-in-$\delta$ $L^2$ estimate. By Lemma \ref{lem:HineqBall} and the Gagliardo-Nirenberg-Sobolev-type estimate \eqref{ineq:L1HsSR} (recall $f^{\epsilon,\delta}$ is a probability measure), there holds   
\[\lnorm{f^{\epsilon,\delta}}_{L^2}  \lesssim \lnorm{\bar{\chi}f^{\epsilon,\delta} }_{H^1_{\textup{Hyp},\delta}}^{1-\theta} + \sum_{j=1}^\infty 2^{jq(1-\theta)}\lnorm{\chi_j f^{\epsilon,\delta}}_{L^1}^{\theta}\lnorm{\chi_j f^{\epsilon,\delta} }_{H^1_{\textup{Hyp},\delta}}^{1-\theta}, 
\]
where we have denoted $\|\cdot\|_{H^1_{\textup{Hyp},\delta}} = \|\cdot\|_{\mathcal{H}_\delta} + \|\cdot\|_{\mathcal{H}^*_{\delta}}$. Pairing \eqref{eq:fed} with $\bar{\chi} f^{\epsilon,\delta}$ and $\chi_j f^{\epsilon,\delta}$ followed by standard manipulations gives
\begin{align*}
\lnorm{\bar{\chi} f^{\epsilon,\delta}}_{H^1_{\textup{Hyp},\delta}} +  \sup_j \lnorm{\chi_j f^{\epsilon,\delta}}_{H^1_{\textup{Hyp},\delta}} \lesssim_\epsilon \lnorm{f^{\epsilon,\delta}}_{L^2}. 
\end{align*}
Therefore, we have
\begin{align*}
\lnorm{f^{\epsilon,\delta}}_{L^2} & \lesssim \lnorm{f^{\epsilon,\delta}}_{L^2}^{\theta} + \sum_{j=1}^\infty 2^{jq(1-\theta)} \lnorm{\chi_j f^{\epsilon,\delta} }_{L^1}^{\theta} \lnorm{f^{\epsilon,\delta}}_{L^2}^{\theta}, 
\end{align*}
which implies that for $\theta_1,\theta_2 \in (0,1)$ we have a uniform-in-$\delta$ estimate on the $L^2$ norm.
Note that the estimate still depends badly on $\epsilon$.
Passing to the $\delta \to 0$ limit shows that $f^\epsilon \in L^2$ for each $\epsilon > 0$.

Finally, observe that $\brak{x}^N f^{\epsilon} \log f^{\epsilon} \in L^1$ for all $N > 0$ indeed,
\begin{align*}
\int_{\S \R^n} \brak{x}^N f^{\epsilon} \abs{\log f^\epsilon} \dee q & \lesssim \int_{\S \R^n} \brak{x}^N \left(\sqrt{f^{\epsilon}}  + (f^{\epsilon})^{3/2} \right) \dee q  \lesssim \lnorm{f^{\epsilon} e^{\gamma \abs{x}^2}}_{L^1} + \lnorm{f^{\epsilon} e^{\gamma \abs{x}^2}}_{L^1}^{1/4}\norm{f^\epsilon}_{L^2}^{3/4}. 
\end{align*}
Note that in fact, one can obtain similar moment estimates also on $L^2$. 
This completes the proof of Theorem \ref{thm:BasicProjDensity}. 
\end{proof} 
\section{Properties of projective lifts and the Furstenberg-Khasminskii formula}\label{sec:Basics2}

\subsection{Sufficient conditions for projective spanning: Proof of Proposition \ref{prop:class-proj-span}}\label{subsec:projSpanSuffCondAppendix}

 In this section, we give a proof of Proposition \eqref{prop:class-proj-span} that characterizes when a collection of lifted vector fields $\{\widetilde{X}_k\}_{k=0}^r$ satisfy the H\"ormander condition on $\S M$ in terms of transitivty of the matrix lie algebra $\mathfrak{m}_x(X_0,\ldots,X_r)$ (defined in \eqref{eq:m-lie-alg-def}).

Before we prove Proposition \ref{prop:class-proj-span}, we will need some preliminary results. As we will be taking commutators of the above vector fields, it is important to record how projective vector fields behave under the Lie bracket.

\begin{lemma} \label{lem:MatProj}
Let $A, B \in \mathfrak{sl}(\R^n)$, then the following identity holds
\[
[V_A,V_B](v) = -V_{[A,B]}(v),
\]
where $[A, B] := A B - B A$ denotes the usual commutator on linear operators.
\end{lemma}
\begin{proof}
Let $\nabla$ denote the Levi-Civita connection on $\S^{n-1}$, then since $\nabla$ is torsion-free, we have the following formula for the Lie bracket in terms of the covariant derivative
\[
	[V_A,V_B] = \nabla_{V_A}V_B - \nabla_{V_B}V_A.
\]
Recall from the proof of Lemma \ref{eq:Div-identity} that using the embedding of $\S^{n-1}$ into $\R^n$, we have the following formula for the total covariant derivative of $V_A$ (viewed as a linear operator on $T_v\S^{n-1}$)
\[
	\nabla V_A(v) = \Pi_v A - \langle v, Av\rangle I.
\]
It follows that
\[
\begin{aligned}
	[V_A,V_B](v) &= \nabla V_B(v) V_A(v) - \nabla V_A(v)V_B(v)\\
	&= \Pi_v B \Pi_vAv - \Pi_v A \Pi_vBv - \langle v, Bv\rangle V_A(v) +\langle v, Av\rangle V_B(v).
\end{aligned}
\]
Using the fact that $\Pi_v u = u - \langle u, v \rangle v$ for $u \in T_x M$,  we find
\[
	\Pi_v B \Pi_vAv + \langle v,Av\rangle V_B(v) = \Pi_v BAv,
\]
and 
\[
	\Pi_v A \Pi_vBv + \langle v,Bv\rangle V_A(v) = \Pi_v ABv, 
\]
hence
\[
[V_A,V_B]= V_{BA} - V_{AB} = -V_{[A,B]} \, . \qedhere
\]
\end{proof}



Of fundamental importance is the following observation for the lifting operation $X\mapsto \tilde{X}$.
\begin{lemma}\label{lem:lifting-isomorph}
Any two vector fields $X,Y \in \mathfrak{X}(M)$ satisfy the identity
\[
	[\tilde{X},\tilde{Y}] = [X,Y]{\,\,}^{\widetilde{}}\,.
\]
Thus the lifting operation $X \mapsto \tilde{X}$ is a Lie algebra isomorphism onto $\tilde{\mathfrak{X}}(M) = \{\tilde{X}\,:\, X\in \mathfrak{X}(M)\}$.
\end{lemma}
\begin{proof}
Given a vector field $X$ on $M$, it's lift $\widetilde{X}$ on $\S M$ we can always be split into horizontal $\hat{X}$ and vertical $\hat{V}_{\nabla X}$ fields as 
\[
	\widetilde{X} = \hat{X} + \hat{V}_{\nabla X},
\]
according to the orthogonal splitting $T_{(x,v)}\S M = T_xM\oplus T_v\S_x M$ induced by the Sasaki-Metric $\tilde{g}$ and the associated lift of the Levi-Civita connection $\tilde{\nabla}$ on $\S M$. Explicitly, the horizontal $\hat{X}$ and vertical $\hat{V}_{\nabla X}$ components of $\tilde{X}$ are given by
\[
	\hat{X}(x,v) = (X(x),0) \, , \quad \hat{V}_{\nabla X}(x,v) = (0, V_{\nabla X(x)}(v)).
\]
Let $U(x,v) = (v, 0)$ be the ``canonical' horizontal vector field on $\S M$. Note that $U$ is parallel to $\hat{X}$ in the sense that $\widetilde{\nabla}_{\hat{X}}U = 0$. Additionally, define a linear mapping $\hat{\Pi}_{(x,v)}$ on $T_{(x,v)}\S M$ by
\[
	\hat{\Pi}_{(x,v)}( u_1, u_2) = (0, \Pi_{(x,v)} u_1),
\]
so that we can express the vertical field $\hat{V}_{\nabla X}$ as
\[
	\hat{V}_{\nabla X} = \hat{\Pi}\tilde{\nabla}_U\hat{X}.
\]
Note that for any ``horizontal'' vector field $\hat{X}$, $\widetilde{\nabla}_{\hat X}\hat{\Pi} =0$ holds since $\nabla$ preserves the metric $g$. In light of the fact that $[\hat{X},\hat{Y}] = [X,Y]{\,}^{\widehat{}}$, our proof will be complete if we show that
\[
	[\tilde{X},\tilde{Y}] = [\hat{X},\hat{Y}] + \hat{\Pi}\tilde{\nabla}_U[\hat{X},\hat{Y}].
\]
The Lie bracket of $\tilde{X}$ and $\tilde{Y}$ can be written as
\begin{equation}\label{eq:commutator-expand}
	[\tilde{X},\tilde{Y}] = [\hat{X},\hat{Y}] + [\hat{X}, \hat{V}_{\nabla Y}] - [\hat{Y},\hat{V}_{\nabla X}] + [\hat{V}_{\nabla X},\hat{V}_{\nabla Y}].
\end{equation}
Likewise, a simple consequence of Lemma \ref{lem:MatProj} implies
\[
	[\hat{V}_{\nabla X},\hat{V}_{\nabla Y}] = -\hat{\Pi}[\tilde{\nabla} \hat{X},\tilde\nabla \hat Y] U = \hat{\Pi} \left(\tilde{\nabla}_{[U,\hat{X}]}\hat{Y} - \tilde{\nabla}_{[U,\hat{Y}]}\hat{X}\right),
\]
where above $[\tilde{\nabla}\hat X, \tilde\nabla \hat Y]$ denotes the commutator of $\tilde\nabla \hat X, \tilde\nabla \hat Y$ viewed as linear endomorphisms on a fixed tangent space $T_{(x,v)}\S M$. The remaining terms in equation \eqref{eq:commutator-expand} can be computed as
\[
	[\hat{X}, \hat{V}_{\nabla Y}] - [\hat Y,\hat{V}_{\nabla X}] = \tilde{\nabla}_{\hat{X}}\hat{V}_{\nabla Y} - \tilde{\nabla}_{\hat{Y}}\hat{V}_{\nabla X} = \hat\Pi\left(\tilde\nabla_{\hat{X}}\tilde\nabla_{U}\hat{Y} - \tilde\nabla_{\hat{Y}}\tilde\nabla_{U}\hat{X}\right).
\]
Therefore, putting everything together, we find
\[
	[\tilde{X},\tilde{Y}] = [X,Y]{\,}^{\widehat{}} + \hat\Pi\left(\tilde\nabla_{\hat{X}}\tilde\nabla_{U}\hat{Y} - \tilde\nabla_{\hat{Y}}\tilde\nabla_{U}\hat{X} + \tilde{\nabla}_{[U,\hat{X}]}\hat{Y} - \tilde{\nabla}_{[U,\hat{Y}]}\hat{X}\right).
\]
The proof will be complete once we show the identity
\begin{equation}\label{eq:covariant-comm-id}
	\tilde\nabla_{\hat{X}}\tilde\nabla_{U}\hat{Y} - \tilde\nabla_{\hat{Y}}\tilde\nabla_{U}\hat{X} + \tilde{\nabla}_{[U,\hat{X}]}\hat{Y} - \tilde{\nabla}_{[U,\hat{Y}]}\hat{X} = \tilde{\nabla}_U[\hat{X},\hat{Y}] \,. 
\end{equation}
For this, we can use the Riemann curvature tensor on $\S M$
\[
	\tilde{R}(X,Y)Z  := \tilde{\nabla}_X\tilde\nabla_Y Z - \tilde\nabla_Y\tilde\nabla_X Z - \tilde\nabla_{[X,Y]}Z
\] 
to change the order of covariant derivatives, giving
\begin{equation}
\begin{aligned}
	\tilde\nabla_{\hat{X}}\tilde\nabla_{U}\hat{Y} - \tilde\nabla_{\hat{Y}}\tilde\nabla_{U}\hat{X} + \tilde{\nabla}_{[U,\hat{X}]}\hat{Y} - \tilde{\nabla}_{[U,\hat{Y}]}\hat{X}
	&= \tilde{R}(\hat X,U)\hat Y - \tilde R(\hat Y,U)\hat X + \tilde\nabla_U\tilde\nabla_{\hat X} \hat Y - \tilde \nabla_U\tilde \nabla_{\hat Y} \hat X\\
	&= \tilde{R}(\hat X,U)\hat Y + \tilde R(U,Y)\hat X + \tilde{\nabla}_U[\hat X,\hat Y].
\end{aligned}
\end{equation}
The first Bianchi identity implies that
\[
	\tilde R(\hat X,U)\hat Y + \tilde R(U,\hat Y)\hat X = \tilde R(\hat X,\hat Y)U,
\]
and therefore identity \eqref{eq:covariant-comm-id} follows from the fact that $R(\hat X, \hat Y)U = 0$ since, for any vector field $Z \in \mathfrak X(M)$, we have that 
 $\tilde\nabla_{\hat Z} U = 0$. 
\end{proof}

We are now ready to prove Proposition \ref{prop:class-proj-span}.

\begin{proof}[Proof of Proposition \ref{prop:class-proj-span}]
A simple consequence of Lemma \ref{lem:lifting-isomorph} that for any collection of vector fields $\{X_k\}_{k=0}^r$ on $M$ we have the following identification 
\[
	\mathrm{Lie}(\tilde{X}_0;\tilde{X}_1,\ldots,\tilde{X}_r) = \{\tilde{X}\,:\, X \in \mathrm{Lie}(X_0;X_1,\ldots,X_r)\}.
\]
Therefore the parabolic H\"ormander condition for $\{\tilde{X}_k\}_{k=0}^r$ is equivalent to 
\begin{equation}\label{eq:lifing-para-equiv}
	\left\{(X(x),V_{M_X(x)}(x,v)) \,:\, X \in \mathrm{Lie}(X_0;X_1,\ldots,X_r)\right\} = T_x M \oplus T_v \S_x M.
\end{equation}
Clearly if the above condition is satisfied then $\{X_k\}_{k=0}^r$ satisfies the parabolic H\"ormander condition and \ref{eq:proj-trans} holds. The converse follows from the fact that \eqref{eq:proj-trans} implies that for each $X\in \mathrm{Lie}(X_0;X_1,\ldots,X_r)$, $(x,v)\in \S M$ and $h\in T_v\S_xM$ there exists a $Y\in \mathrm{Lie}(X_0;X_1,\ldots,X_r)$ with $Y(x) = 0$ such that
\[
	 V_{M_Y(x)}(x,v) = h - V_{M_X(x)}(x,v)
\]
and therefore $Z= X+ Y$ satisfies
\[
	\widetilde{Z} = (X(x),h).
\]
Combining this with the fact that
\[
	\left\{X(x) \,:\, X \in \mathrm{Lie}(X_0;X_1,\ldots,X_r)\right\} = T_x M,
\]
concludes the proof.
\end{proof}

\subsection{Furstenberg-Khasminskii formula}

The following identity is useful relating the divergence of a projective lift $\tilde{X}$ to that of $X$; it is well-known in the RDS community and the proof is simple, but as we could not locate a reference with the precise form we need, we provide a short proof here for the readers' convenience.
\begin{lemma}
Let $X\in \mathfrak{X}(M)$. Then the following identity holds
\begin{equation}\label{eq:Div-identity}
	\Div \tilde{X}(x,v) = 2\Div X(x) - n \langle v , \nabla X(x)v\rangle_x.
\end{equation}
\end{lemma}

Similarly, as we could not locate a statement that exactly matched the one we require (particularly the form of $\tilde{Q}$ stated below), we include a proof of the Furstenberg-Khasminskii formula for \eqref{eq:general-SDE} (see \cite{Baxendale1992-xe} for more details).

\begin{proposition}[Furstenberg-Khasminskii] \label{lem:FK}
Define for each $x\in M$
\[
	Q(x) := \Div X_0(x) + \frac{1}{2}\sum_{k=1}^rX_k\Div X_k(x),
\]
and each $w\in \S M$
\[
	\tilde{Q}(w) := \Div \tilde{X}_0 (w) + \frac{1}{2}\sum_{k=1}^r\tilde{X}_k\Div \tilde{X}_k(w).
\]
Suppose that $(w_t)$ has a unique stationary probability measure $\nu$ on $\S M$ that projects to $\mu$ on $M$, and that $Q\in L^1(\mu)$ and $\tilde{Q}\in L^1(\nu)$, then the following formulas hold
\begin{equation}\label{eq:sum-lyap-form}
	\lambda_\Sigma = \int_{M}Q\,\dee \mu,
\end{equation}
\begin{equation}\label{eq:Furstenberg-Khasminskii}
	n\lambda_1 - 2\lambda_\Sigma  =  - \int_{\S M} \tilde{Q}\,\dee\nu.
\end{equation}
\end{proposition}
\begin{proof}
We note that \eqref{eq:sum-lyap-form} is standard and can be found in a number of references (see for instance \cite{Baxendale1992-xe}).
To prove \eqref{eq:Furstenberg-Khasminskii}, we see that a straight forward computation and formula \eqref{eq:Div-identity} yields
\[
\begin{aligned}
	\dee \log(|D\Phi^t(x)v|) &= \langle v_t, \nabla X_0(x_t) v_t\rangle \dt + \sum_{k=1}^r \langle v_t, \nabla X_k(x_t) v_t \rangle \circ \dee W_t^k\\
	&= \frac{1}{n}\left(2\Div X_0(x_t) - \Div\tilde{X}_0(w_t)\right)\dt + \frac{1}{n}\sum_{k=1}^r \left(2\Div X_k(x_t) - \Div\tilde{X}_k(w_t)\right) \circ \dee W_t^k.
\end{aligned}
\]
Converting to It\^{o} and integrating in time gives
\[
\begin{aligned}
	\frac{1}{t}\log(|D\Phi^t(x)v|) 
	&= \frac{1}{nt} \int_0^t 2Q(x_s)\ds  - \frac{1}{nt} \int_0^t \tilde{Q}(w_s)\ds + \frac{1}{t}M_t,
\end{aligned}
\]
with $M_t$ is a mean-zero martingale arising from the It\^{o} integral whose exact form is not important.

We now take $t \to \infty$, using a corollary of, e.g., Theorem III.1.2 in \cite{kifer2012ergodic} to ensure the LHS converges to $\lambda_1$, while 
the first term on the RHS converges to $\int \tilde{Q} \,\dee \nu$ by the ergodic theorem. In particular, 
$\frac{1}{t} M_t$ must also converge, both pointwise and in $L^1(\P \times \nu)$, hence
$\frac{1}{t} M_t \to 0$ by the martingale law of large numbers. 
\end{proof}

\phantomsection
\addcontentsline{toc}{section}{References}
\bibliographystyle{abbrv}
\bibliography{bibliography}

\end{document}